\documentclass{article}
\usepackage{latexsym}
\usepackage{amssymb}
\usepackage{amsfonts}
\usepackage{amsmath}
\usepackage{amsthm}
\usepackage{ulem}	
\usepackage{sectsty}	
\usepackage[all]{xy}	
\usepackage{bm}		
\usepackage{wasysym}	
\usepackage{ifthen}
\usepackage{graphics}
\usepackage{psfrag, psfig}
\usepackage[dvips]{graphicx}
\usepackage[pdftex]{lscape} 
\usepackage{rotating} 
\usepackage{epsfig}
\usepackage[toc,page]{appendix}
\usepackage{h1-macro}

\title{H\'enon-like maps with arbitrary stationary combinatorics}
\author{P.~E.~Hazard}
\date{\today}

\newcounter{bean}
\renewenvironment{enumerate}{\begin{list}{(\roman{bean})}{\usecounter{bean}\setlength{\rightmargin}{\leftmargin}}}{\end{list}\setcounter{bean}{1}}
\newenvironment{a-enumerate}{\begin{list}{(A-\arabic{bean})}{\usecounter{bean}\setlength{\rightmargin}{\leftmargin}}}{\end{list}\setcounter{bean}{1}}
\newenvironment{b-enumerate}{\begin{list}{(B-\arabic{bean})}{\usecounter{bean}\setlength{\rightmargin}{\leftmargin}}}{\end{list}\setcounter{bean}{1}}
\newenvironment{u-enumerate}{\begin{list}{(U-\arabic{bean})}{\usecounter{bean}\setlength{\rightmargin}{\leftmargin}}}{\end{list}\setcounter{bean}{1}}
\newenvironment{t-enumerate}{\begin{list}{(T-\arabic{bean})}{\usecounter{bean}\setlength{\rightmargin}{\leftmargin}}}{\end{list}\setcounter{bean}{1}}
\newenvironment{ur-enumerate}{\begin{list}{(UR-\arabic{bean})}{\usecounter{bean}\setlength{\rightmargin}{\leftmargin}}}{\end{list}\setcounter{bean}{1}}
\newenvironment{hr-enumerate}{\begin{list}{(HR-\arabic{bean})}{\usecounter{bean}\setlength{\rightmargin}{\leftmargin}}}{\end{list}\setcounter{bean}{1}}

\newcounter{condition}
\newenvironment{Alpha-enumerate}{\begin{list}{(\Alph{condition})}{\usecounter{condition}\setlength{\rightmargin}{\leftmargin}}}{\end{list}\setcounter{condition}{1}}

\theoremstyle{plain}
\newtheorem{thm}{Theorem}[section]
\newtheorem{lem}[thm]{Lemma}
\newtheorem{prop}[thm]{Proposition}
\newtheorem{cor}[thm]{Corollary}

\theoremstyle{definition}
\newtheorem{defn}[thm]{Definition}

\theoremstyle{remark}
\newtheorem{rmk}[thm]{Remark}
\newtheorem{notn}[thm]{Notation}

\numberwithin{equation}{section}

\begin{document}

\maketitle
\def\IMSmarkvadjust{0 pt}
\def\IMSmarkhadjust{0 pt}
\def\IMSmarkhpadding{0 pt}
\def\IMSpubltext{Published in modified form:}
\def\SBIMSMark#1#2#3{
 \font\SBF=cmss10 at 10 true pt
 \font\SBI=cmssi10 at 10 true pt
 \setbox0=\hbox{\SBF \hbox to \IMSmarkhpadding{\relax}
                Stony Brook IMS Preprint \##1}
 \setbox2=\hbox to \wd0{\hfil \SBI #2}
 \setbox4=\hbox to \wd0{\hfil \SBI #3}
 \setbox6=\hbox to \wd0{\hss
             \vbox{\hsize=\wd0 \parskip=0pt \baselineskip=10 true pt
                   \copy0 \break%
                   \copy2 \break%
                   \copy4 \break}}
 \dimen0=\ht6   \advance\dimen0 by \vsize \advance\dimen0 by 8 true pt
                \advance\dimen0 by -\pagetotal
	        \advance\dimen0 by \IMSmarkvadjust
 \dimen2=\hsize \advance\dimen2 by .25 true in
	        \advance\dimen2 by \IMSmarkhadjust

%
%
  \openin2=publishd.tex
  \ifeof2\setbox0=\hbox to 0pt{}
  \else 
     \setbox0=\hbox to 3.1 true in{
                \vbox to \ht6{\hsize=3 true in \parskip=0pt  \noindent  
                {\SBI \IMSpubltext}\hfil\break
                \input publishd.tex 
                \vfill}}
  \fi
  \closein2
  \ht0=0pt \dp0=0pt
 \ht6=0pt \dp6=0pt
 \setbox8=\vbox to \dimen0{\vfill \hbox to \dimen2{\copy0 \hss \copy6}}
 \ht8=0pt \dp8=0pt \wd8=0pt
 \copy8
 \message{*** Stony Brook IMS Preprint #1, #2. #3 ***}
}

\SBIMSMark{2010/1}{February 2010}{}

\vspace{-1cm}

\begin{abstract}
We extend the renormalisation operator introduced in~\cite{dCML} from period-doubling H\'enon-like maps to H\'enon-like maps with arbitrary stationary combinatorics. We show the renormalisation picture holds also holds in this case if the maps are taken to be \emph{strongly dissipative}. We study infinitely renormalisable maps $F$ and show they have an invariant Cantor set $\Cantor$ on which $F$ acts like a $p$-adic adding machine for some $p>1$. We then show, as for the period-doubling case in~\cite{dCML}, the sequence of renormalisations have a universal form, but the invariant Cantor set $\Cantor$ is non-rigid. We also show $\Cantor$ cannot possess a continuous invariant line field.
\end{abstract}

\section{Introduction}
\subsection{H\'enon Renormalisation}
In~\cite{dCML}, de Carvalho, Lyubich and Martens constructed a period-doubling renormalisation operator for H\'enon-like mappings of the form
\begin{equation}
F(x,y)=(f(x)-\e(x,y),x).
\end{equation}
Here $f$ is a unimodal map and $\e$ was a real-valued map from the square to the positive real numbers of small size (we shall be more explicit about the maps under
consideration in Sections~\ref{sect:unimodal} and~\ref{sect:henonlike}). They showed that for $|\e|$ sufficiently small the unimodal renormalisation picture carries over to this case. Namely, there exists a unique renormalisation fixed point (which actually coincides with unimodal period-doubling renormalisation fixed point) which is hyperbolic with codimension one stable manifold, consisting of infinitely renormalisable period-doubling maps, and dimension one local unstable manifold. They later called this regime \emph{strongly dissipative}.

In the same paper they then studied the dynamics of infinitely renormalisable H\'enon-like maps $F$. They showed that such a map has an invariant Cantor set, $\Cantor$, upon which the map acts like an adding machine. This allowed them to define the \emph{average Jacobian} given by
\begin{equation}
b=\exp \int_\Cantor \log|\jac{F}{z}|d\mu(z)
\end{equation}
where $\jac{F}{z}=\det \D{F}{z}$ is the Jacobian determinant and $\mu$ denotes the unique $F$-invariant measure on $\Cantor$ induced by the adding machine. This quantity played an important role in their study of the local behaviour of
such maps around the Cantor set. They took a distinguished point $\tau$ of the Cantor set $\Cantor$ called the \emph{tip}. They examined the dynamics and geometry of the Cantor set asymptotically around $\tau$. Their two main results can then be stated as follows.
\begin{thm}[Universality at the tip]
There exists a universal constant $0<\rho<1$ and a universal real-analytic real-valued  function $a(x)$ such that the following holds: 
Let $F$ be a strongly dissipative, period-doubling, infinitely renormalisable H\'enon-like map. Then
\begin{equation}
\RH^nF(x,y)=(f_n(x)-b^{2^n}a(x)y(1+\bigo(\rho^n)),x)
\end{equation}
where $b$ denotes the average Jacobian of $F$ and $f_n$ are unimodal maps converging exponentially to the unimodal period-doubling renormalisation fixed point.
\end{thm}
\begin{thm}[Non-rigidity around the tip]
Let $F$ and $\tilde F$ be two strongly dissipative, period-doubling, infinitely renormalisable H\'enon-like maps. Let their average Jacobians be $b$ and $\tilde b$ and their Cantor sets be $\Cantor$ and $\tilde\Cantor$ respectively. Then for any conjugacy $\pi\colon \Cantor\to\tilde\Cantor$ between $F$ and $\tilde F$ the H\"older exponent $\alpha$ satisfies
\begin{equation}
\alpha\leq \frac{1}{2}\left(1+\frac{\log b}{\log\tilde b}\right)
\end{equation}
In particular if the average Jacobians $b$ and $\tilde b$ differ then there cannot exist a $C^1$-smooth conjugacy between $F$ and $\tilde F$.  
\end{thm}
For a long time it was assumed that the properties satisfied by the one-dimensional unimodal renormalisation theory would also be satisfied by any renormalisation theory in any dimension.
In particular, the equivalence of the universal (real and complex a priori bounds) and rigid (pullback argument) properties in this setting made it natural to think that such a relation would hold more generally. That is, if universality controls the geometry of an attractor and we have a topological conjugacy between two attractors it was expected that such a conjugacy could be promoted to a smooth map, since the geometry of infinitesimally close pairs of orbits cannot differ too much. The above shows that this intuitive reasoning is incorrect.

Let us now outline the structure of the present work.
After recalling preliminary results in Section~\ref{sect:unimodal}, in Section~\ref{sect:henonlike} we generalise this renormalisation operator to other combinatorial types. We show that in this case the renormalisation picture also holds if $\bar\e$ is sufficiently small. Namely, for any stationary combinatorics there exists a unique renormalisation fixed point, again coinciding with the unimodal renormalisation fixed point, which is hyperbolic with codimension one stable manifold, consisting of infinitely renormalisable maps, and dimension-one local unstable manifold. 

We then study the dynamics of infinitely renormalisable maps of stationary combinatorial type and show that such maps have an $F$-invariant Cantor set $\Cantor$ on which $F$ acts as an adding machine. We note that the strategy to show that the limit set is a Cantor set in the period-doubling case does not carry over to maps with general stationary combinatorics. In both cases the construction of the Cantor set is via `Scope Maps', defined in sections~\ref{sect:unimodal} and~\ref{sect:henonlike}, which we approximate using the so-called `Presentation function' of the renormalisation fixed point. In the period-doubling case this is known to be contracting as the renormalisation fixed point is convex (see the result of Davie~\cite{Dav1}) and the unique fixed point lying in the interior of the interval is expanding (see the theorem of Singer~\cite[Ch. 3]{dMvS}). In the case of general combinatorics this is unlikely to be true. The work of Eckmann and Wittwer~\cite{EandW} suggests the convexity of fixed points for sufficiently large combinatorial types does not hold. The problem of contraction of branches of the presentation function was also asked in~\cite{JMS}.

Once this is done we can define the average Jacobian and the tip of an infinitely renormalisable H\'enon-like map in a way analogous to the period-doubling case. This then allows us, in Section~\ref{sect:applications}, to generalise the universality and non-rigidity results stated above to the case of arbitrary combinatorics.
We also generalise another result from~\cite{dCML}, namely the Cantor set of an infinitely renormalisable H\'enon-like map cannot support a continuous invariant line field. Our proof, though, is significantly different. This is because in the period-doubling case they observed a `flipping' phenomenon where orientations were changed purely because of combinatorics. Their argument clearly breaks down in the more general case where there is no control over such things.

\paragraph{Acknowledgements.} Firstly, I thank my thesis adviser Marco Martens for suggesting these problems and giving many useful insights into renormalisation. Secondly, I would like to thank Misha Lyubich for many useful comments about my work and for reading a draft of this manuscript. Finally, I also thank Michael Benedicks, Andr\'e de Carvalho and Sebastian van Strien for their insights on the contents of this paper and for their continuing interest in my work.

\subsection{Notations and Conventions}

Given a function $F$ we denote by $\Dom(F)$ its domain of definition.
Typically this will be a subset of $\RR^n$ or $\CC^n$.
For $i\geq 0$ we denote its $i$-th iterate by $\o{i}{F}$ whenever it exists. For $S\subset \Dom(F)$ we denote its $i$-th
preimage by $\o{-i}{F}\colon \o{i}{F}(S)\to S\subset\RR^n$ whenever this exists.

Now we restrict our attention to the one- and two-dimensional cases, both real and complex.
Let $\pi_x,\pi_y\colon \RR^2\to \RR$ denote the projections onto the $x$- and $y$- coordinates. We will identify these with their extensions to $\CC^2$. (In fact we
identify all real functions with their complex extensions whenever they exist.)

Given $a,b\in\RR$, denote the closed interval between them by $[a,b]=[b,a]$. Denote the interval $[0,1]$ by $J$.
For any interval $T\subset\RR$ denote its boundary by $\del T$, its left endpoint by $\del^{-}T$ and its right endpoint by $\del^{+}T$. 
Given two intervals $T_0,T_1\subset J$ denote an affine bijection from $T_0$ to $T_1$ by $\ii_{T_0\to T_1}$. 
Typically it will be clear from the situation whether we are using the unique orientation preserving or orientation reversing bijection. 

Let us denote the square $[0,1]\times [0,1]=J^2$ by $B$. 
We call $S\subset B$ a \emph{rectangle} if it is the Cartesian product of two intervals.
Given two rectangles $B_0, B_1\subset B$ we will denote an affine bijection from $B_0$ to $B_1$ preserving horizontal and vertical lines by $\III_{B_0\to B_1}$. 
Again the orientations of its components will be clear from the situation.

Let $\Omegax\subseteq\Omegay\subset\CC$ be simply connected domains compactly containing $J$ and let $\Omega=\Omegax\times\Omegay$ denote the resulting polydisk containing $B$.

If $F\colon \RR^2\to\RR^2$ is differentiable at a point $z\in\RR^2$ we will denote the derivative of $F$ at $z$ by $\D{F}{z}$. The \emph{Jacobian} of $F$ is given by
\begin{equation}
\jac{F}{z}=\det \D{F}{z}
\end{equation}
Given a bounded region $S\subset \RR^2$ we will define the \emph{distortion} of $F$ on $S$ by
\begin{equation}
\dis{F}{S}=\sup_{z,\tilde z\in S}\log\left|\frac{\jac{F}{z}}{\jac{F}{\tilde z}}\right|.
\end{equation}

\section{Unimodal Maps}\label{sect:unimodal}
Let us now briefly review some parts of one-dimensional unimodal renormalisation theory. In particular, the presentation function theory associated with it developed by Rand~\cite{Rand1}, Feigenbaum~\cite{F3}, Sullivan~\cite{Sull2} and Birkhoff, Martens and Tresser~\cite{BMT}.
\subsection{The Space of Unimodal Maps}
Let $\beta>0$ be a constant, which we will think of as being small. Let $\U_{\Omegax,\beta}$ denote the space of maps $f\in C^\omega(J,J)$ satisfying the following properties:
\begin{enumerate}
\item there is a unique critical point $c_0=c_0(f)$, which lies in $(0,1-\beta]$;
\item $f$ is orientation preserving to the left of $c$ and orientation reversing to the right of $c$;
\item $f(\del^+ J)=f(\del^-J)=0$ and $c_1=f(c_0)>c_0$;
\item there is a unique fixed point $\alpha=\alpha(f)$ lying in $\interior(J)$. Both fixed points are expanding;
\item $f$ admits a holomorphic extension to the domain $\Omegax$, upon which it can be factored as $\psi\circ Q$, where $Q\colon\CC\to\CC$ is given by $Q(z)=4z(1-z)$ and
$\psi\colon Q(\Omegax)\to\CC$ is an orientation preserving univalent mapping which fixes the real axis;
\end{enumerate}
Such maps will be called \emph{unimodal maps}. Given any interval $T\subset \RR$ we will say a map $g\colon T\to T$ is \emph{unimodal on $T$} if there exists an affine bijection $h\colon J\to T$ such that $h^{-1}\circ g\circ
h\in \U_{\Omegax,\beta}$. We will identify all unimodal maps with their holomorphic extensions. 

We make following observations: first, this extension will be
$\RR$-symmetric (i.e. $f(\bar z)=\oline{f(z)}$ for all $z\in\Omegax$) and secondly, there are two fixed points, one with positive multiplier lying in $\del J$ and the other lying in the interior with negative multiplier.


\subsection{The Renormalisation Operator}
\begin{defn}
Let $p>1$ be an integer and let $W=W_p$ denote the set $\{0,1,\ldots,p-1\}$. A permutation $\perm$ of $W$ is said to be \emph{unimodal of length $p$} if there exists an order preserving embedding $\i\colon W\to J$ and a unimodal map $f\colon J\to J$ such that $f(\i(k-1))=\i(k\!\!\mod p)$.
\end{defn}
\begin{defn}
Let $p>1$ be an integer.
A map $f\in \U_{\Omegax,\beta}$ has a \emph{renormalisation interval of period $p$} if 
\begin{enumerate}
\item there is a closed subinterval $\oo{0}{J}\subset J$ such that $\o{p}{f}\left(\oo{0}{J}\right)\subset \oo{0}{J}$;
\item there exists an affine bijection $h\colon J\to\oo{0}{J}$ such that
\begin{equation}
\RU f=h^{-1}\circ \o{p}{f}\circ h\colon J\to J
\end{equation}
is an element of $\U_{\Omega_x,\beta}$. Note there are exactly two such affine bijections, but there will only be one such that $\RU f\in \U_{\Omega_x,\beta}$;
\end{enumerate}
The interval $\oo{0}{J}$ is called a \emph{renormalisation interval of period $p$} for $f$. The collection $\left\{\oo{w}{J}\right\}_{w\in W}$ is called the \emph{renormalisation cycle}.
\end{defn}
\begin{defn}
Let $p>1$ be an integer and let $\perm$ be a unimodal permutation of length $p$. 
A map $f\in \U_{\Omegax,\beta}$ is \emph{renormalisable of combinatorial type $\perm$} if, 
\begin{enumerate}
\item $f$ has a renormalisation interval $\oo{0}{J}$ of type $p$ which contains the critical point $c_0$;
\item letting $\oo{w}{J}$ denote the connected component of $\o{p-w}{f}\left(\oo{0}{J}\right)$ containing $\o{w}{f}\left(\oo{0}{J}\right)$, the interiors of the subintervals $\oo{w}{J}$, $w\in W$, are pairwise disjoint;
\item $f$ acts on the set $\left\{\oo{0}{J},\oo{1}{J},\ldots \oo{p-1}{J}\right\}$, embedded in the line with the standard orientation, as $\upsilon$ acts on the symbols in $W$. More precisely, if
$J',J''\in\left\{ \oo{w}{J}\right\}_{w\in W}$ are the $i$-th and $j$-th sub-intervals from the left endpoint of $J$ respectively, then $f\left(J'\right)$ lies to the left of
$f\left(J''\right)$ if and only if $\upsilon(i)<\upsilon(j)$.
\end{enumerate}
In this case the map $\RU f$ is called the \emph{renormalisation of $f$} and the operator $\RU$ the \emph{renormalisation operator} of combinatorial type $\perm$.
\end{defn}
\begin{figure}[t]
\centering
\psfrag{J0}{$\oo{0}{J}$}
\psfrag{J1}{$\oo{1}{J}$}
\psfrag{J2}{$\oo{2}{J}$}
\psfrag{f}{\tiny $y=f(x)$}
\psfrag{f3}{\tiny $y=\o{3}{f}(x)$}
\includegraphics[scale=0.62]{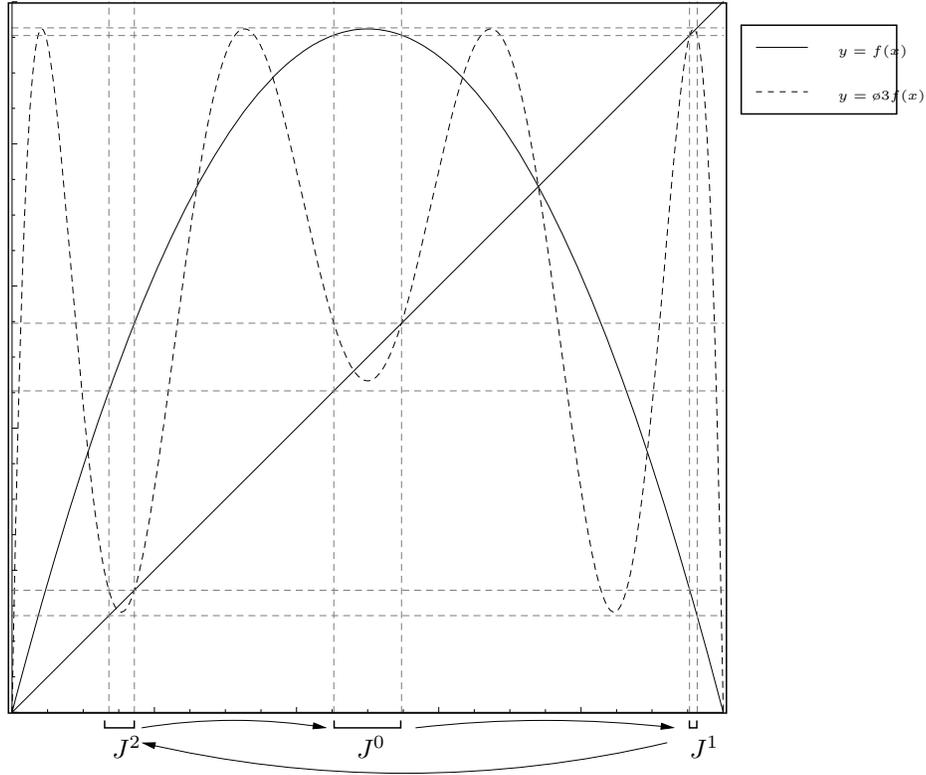}
\caption{The graph of a renormalisable period-three unimodal map $f$ with renormalisation interval $\oo{0}{J}$ and renormalisation cycle $\{\oo{i}{J}\}_{i=0,1,2}$. Note that for $p=3$ there is only one admissable combinatorial type. }\label{fig:unimodal-renormalisation-domain}
\end{figure}
\begin{defn}
Given a renormalisable $f\in \U_{\Omegax,\beta}$ of combinatorial type $\perm$ the renormalisation interval $\oo{0}{J}$ is called the \emph{central interval}. Given $\oo{w}{J}$, $w\in W$, the \emph{maximal extension} of $\oo{w}{J}$ is the largest open interval $\oo{w}{J'}$ containing $\oo{w}{J}$ such that $\o{p-w}{f}|\oo{w}{J'}$ is a diffeomorphism onto its image. 
\end{defn}
\begin{rmk}
The assumption that $\RU f$ lies in $\U_{\Omegax,\beta}$ implies that the boundary of $\oo{0}{J}$ consists of a $p^n$-periodic point and one of its preimages. Moreover, in $\oo{0}{J}$ there is no other preimage of this point and there is at most one periodic point of period $p^n$.
\end{rmk}
Let $\U_{\Omegax,\beta,\perm}$ denote the subspace consisting of unimodal maps $f\in\U_{\Omegax,\beta}$ which are renormalisable of combinatorial type $\perm$.  
If $f\in\U_{\Omegax,\beta,\perm}$ is infinitely renormalisable there is a nested sequence $\uline{J}=\left\{\oo{\word{w}{}}{J}\right\}_{\word{w}{}\in W^*}$ of subintervals such that
\begin{enumerate}
\item $f\left(\oo{\word{w}{}}{J}\right)=\oo{{1+\word{w}{}}}{J}$ for all $\word{w}{}\in W^*$;
\item $\interior\oo{\word{w}{}}{J}\cap\interior\oo{\word{\tilde w}{}}{J}=\emptyset$ for all $\word{w}{}\neq \word{\tilde w}{}\in W^*$ of the same length;
\item $\bigcup_{w\in W}\oo{\word{w}{w}}{J}\subset\oo{\word{w}{}}{J}$ for each $\word{w}{}\in W^*$.
\end{enumerate} 
where $W^*$ denotes the totality of all finite words $\word{w}{}=w_0\ldots w_n$ over $W$ and $1+\word{w}{}$ denotes the adding machine map,
\begin{equation}
1+w_0\ldots w_n=\left\{\begin{array}{ll}
(1+w_0)w_1w_2\ldots w_n & w_0\neq p-1 \\ 
0^k(1+w_k)w_{k+1}\ldots & w_0,\ldots,w_{k-1}=p-1, w_k\neq p-1
\end{array}\right.
\end{equation}
We will denote by $W^n\subset W^*$ the subset of words of length $n>0$ and by $\bar W$ the set of all infinite words over $W$. (Note that $W^*$ corresponds to the space of all cylinder sets of $\bar W$.) We will call this indexing of $\uline{J}$ the \emph{spatial indexing}.

There is a second indexing of $\uline{J}$ given as follows: any $J'\in\uline{J}$ is the $m$-th preimage of $\oo{0^n}{J}$ for some $m, n>0$ satisfying $0\leq m<p^n$. This indexing we will call the \emph{dynamical indexing} and the quantity $n$ will be called the \emph{depth} of $\oo{\word{w}{}}{J}$. We note that this indexing is used in~\cite[Chapter VI.3]{dMvS}.

For each depth $n\geq 0$ the correspondence is given by the map $\word{q}{}\colon W^n\to \{0,1,\ldots,p^n-1\}$ where if $\word{w}{}=w_0\ldots w_n$ then\footnote{This is since $\oo{w}{J}$ are indexed by images but $\oo{w}{\phi}$ is constructed from the preimage of $\oo{0}{J}$ to $\oo{w}{J}$.} $\word{q}{}(\word{w}{})=\sum_{0\leq i\leq n} p^i(p-w_i)$ . More explicitly, if $\word{w}{}=w_0\ldots w_n$ then $\oo{\word{w}{}}{J}=\o{\word{r}{}(\word{w}{})}{f}(\oo{0^n}{J})$ where $\word{r}{}(\word{w}{})=\sum_{0\leq i\leq n}p^iw_i$ and so, as the first return time of $\oo{0^n}{J}$ is $p^n$, we find the transfer time of $\oo{\word{w}{}}{J}$ to $\oo{0^n}{J}$ is $p^n-\word{r}{}(\word{w}{})=\word{q}{}(\word{w}{})$.

\begin{notn}
If $f\in \U_{\Omegax,\beta,\perm}$ is an infinitely renormalisable unimodal map let $f_n=\RU^n f$. Then all objects associated to $f_n$ will also be given this subscript. For
example we will denote by $\uline{J}_n=\left\{\oo{\word{w}{}}{J}_n\right\}_{\word{w}{}\in W^*}$ the nested collection of intervals constructed for $f_n$ in the same way that $\uline{J}$ was constructed for $f$.
\end{notn}

The following plays a crucial role in the renormalisation theory of unimodal maps. (See~\cite{dMvS} for the proof and more details.)
\begin{thm}[real $C^1$-a priori bounds]\label{thm:real-ap-bounds}
Let $f\in\U_{\Omegax,\beta,\perm}$ be an infinitely renormalisable unimodal map. Then there exist constants $L(f), K(f)>1$ and $0<k_0(f)<k_1(f)<1$, such that for all $\word{w}{}\in W^*, w,\tilde{w}\in W$ and each $i=0,1\ldots,p^n-\word{q}{}(\word{w}{})$ the following properties hold,
\begin{enumerate}
\item[(i-a)] $\dis{\o{i}{f}}{\oo{\word{w}{}}{J}}\leq L(f)$;
\item[(i-b)] the previous bound is \emph{beau}: there exists a constant $L>1$ such that for each $f$ as above $L(\RU^n f)<L$ for $n$ sufficiently large;
\item[(ii-a)] $K(f)^{-1}<|\oo{\word{w}{w}}{J}|/|\oo{\word{w}{\tilde{w}}}{J}|<K(f)$;
\item[(ii-b)] the previous bound is \emph{beau}: there exists a constant $K>1$ such that for each $f$ as above $K(\RU^n f)<K$ for $n$ sufficiently large;
\item[(iii-a)] $k_0(f)<|\oo{\word{w}{w}}{J}|/|\oo{\word{w}{}}{J}|<k_1(f)$;
\item[(iii-b)] the previous bound is beau: there exist constants $0<k_0<k_1<1$ such that for each $f$ as above $k_0<k_0(\RU^n f)<k_1(\RU^n f)<k_1<1$ for $n$ sufficiently large.
\end{enumerate}
\end{thm}
The term \emph{beau} for such a property was coined by Sullivan - being an acronym for bounded eventually and universally.
Now we show some properties of the renormalisation operator and renormalisable maps. The Proposition below follows directly by observing that the renormalisation cycle of a renormalisable unimodal map is determined by a hyperbolic periodic orbit and a collection of its preimages. The proof of injectivity can be found in~\cite[Chapter VI]{dMvS}.

\begin{prop}
Let $p>1$ be an integer. Let $0<\gamma<1$. Let $\perm$ be a unimodal permutation of length $p$. 
If $f\in\U_{\Omegax,\beta}$ has renormalisation interval $\oo{0}{J}$ of type $p$ and satisfies the following conditions,
\begin{itemize}
\item $\o{p}{f}(\oo{0}{J})\subsetneq \oo{0}{J}$;
\item $f$ is renormalisable with combinatorics $\perm$;
\end{itemize}
then there exists a neighbourhood $\mathcal N\subset \U_{\Omegax,\beta}$ of $f$ such that for any $\tilde f\in \mathcal N$ the following properties hold, 
\begin{enumerate}
\item
$\tilde f$ is renormalisable with combinatorics $\perm$;
\item
there exists a constant $C>0$, depending upon $f$ only, such that
\begin{equation}
|\RU f-\RU\tilde f|_{\Omegax}<C|f-\tilde f|_{\Omegax};
\end{equation}
\item the operator $\RU$ is injective.
\end{enumerate}
\end{prop}

\subsection{The Fixed Point and Hyperbolicity}
As was noted in the introduction, real a priori bounds was an important component in Sullivan's proof of the following part of the Renormalisation conjecture.
For the proof we refer the reader to~\cite[Chapter VI]{dMvS}. This also contains substantial background material and references. 
\begin{thm}[existence of fixed point]\label{thm:unimodal-fix-point}
Given any unimodal permutation $\perm$ and any domain $\Omegax\subset\CC$ containing $J$, if $\beta>0$ is sufficiently small there exists an $f_*=f_{*,\perm}\in \U_{\Omegax,\beta,\perm}$ such that 
\begin{equation}
\RU f_*=f_*.
\end{equation}
\end{thm}
\begin{notn}
Henceforth we assume that the unimodal permutation $\perm$ and the positive real number $\beta>0$ are fixed, with $\beta$ small enough to ensure $\U_{\Omegax,\beta,\perm}$ contains the renormalisation fixed point. We therefore drop $\beta$ from our notation.
\end{notn}
Sullivan's result was then strengthened by McMullen. See~\cite{McM1}.
\begin{thm}[weak convergence]\label{thm:1d-weak-convergence}
Given any unimodal permutation $\perm$ and any domain $\Omegax\subset\CC$ containing $J$, there exists
\begin{enumerate}
\item a domain $\Omegax'\cc\Omegax$ containing $J$;
\item an integer $N>0$;
\end{enumerate}
both dependent upon $\perm$ and $\Omegax$, such that for any $n>N$ if $f\in\U_{\Omegax,\perm}$ is $n$-times renormalisable then 
\begin{equation}
|\RU^n f-f_\infb |_{\Omegax'}\leq \quarter|f-f_*|_{\Omegax'}.
\end{equation}
\end{thm}
Proof of the full renormalisation conjecture was then completed by Lyubich in~\cite{L3}.
\begin{thm}[exponential convergence]\label{thm:1d-exp-convergence}
Given any unimodal permutation $\perm$ and any domain $\Omegax\subset\CC$ containing $J$, there exists
\begin{enumerate}
\item a domain $\Omegax'\cc\Omegax$, containing $J$;
\item an $\RU$-invariant subspace, $\U_{adapt}\subset\U_{\Omegax',\perm}$;
\item a metric, $d_{adapt}$, on $\U_{adapt}$ which is Lipschitz-equivalent to the sup-norm on $\U_{\Omegax',\perm}$;
\item a constant $0<\rho<1$;
\end{enumerate}
such that, for all $f\in\U_{adapt}$,
\begin{equation}
d_{adapt}(\RU f,f_\infb)\leq \rho d_{adapt}(f,f_\infb).
\end{equation}
\end{thm}
\begin{thm}[codimension-one stable manifold]
For any unimodal permutation $\perm$ and any domain $\Omegax\subset\CC$ containing $J$, the renormalisation operator $\RU\colon \U_{\Omega_x,\perm}\to\U_{\Omegax}$ has a codimension-one stable manifold $\mathcal W_\perm$ at the renormalisation fixed point $f_{*,\perm}$.
\end{thm}
\begin{cor}
Let $\perm$ be a unimodal permutation on $W$. Let $f\in\U_{\Omegax,\perm}$ be an infinitely renormalisable unimodal map. Then the cycle, $\{\oo{w}{J_n}\}_{w\in W}$, of the central interval of $f_n$ converges exponentially, in the Hausdorff topology, to the corresponding cycle, $\{\oo{w}{J_*}\}_{w\in W}$, of the renormalisation fixed point $f_*$.
\end{cor}
\subsection{Scope Maps and Presentation Functions}\label{sect:unimodal-scopemaps}
Now we rephrase the renormalisation of unimodal maps in terms of convergence of their \emph{scope maps} defined below. Again the statements here are well known and we make no claim to the originality of their proofs. We merely collect them here for completeness. 

Scope maps were studied initially by Rand~\cite{Rand1}, then by Sullivan~\cite{Sull2}, Feigenbaum~\cite{F3} and Birkhoff, Martens and Tresser~\cite{BMT}. Mostly this was in the case of the so-called \emph{presentation function}, which is the scope map of the renormalisation fixed point (so called because they permute the presentation of the invariant Cantor set). See also the paper of Jiang, Sullivan and Morita~\cite{JMS}. 
\begin{figure}[htbp]
\centering
\psfrag{A}{$\ldots$}

\psfrag{MT1}{$\mt_0$}
\psfrag{f1}{$f_0$}
\psfrag{aj0}{$\oo{0}{J_0}$}
\psfrag{aj1}{$\oo{1}{J_0}$}
\psfrag{aj2}{$\oo{2}{J_0}$}

\psfrag{MT2}{$\mt_1$}
\psfrag{f2}{$f_1$}
\psfrag{bj0}{$\oo{0}{J_1}$}
\psfrag{bj1}{$\oo{1}{J_1}$}
\psfrag{bj2}{$\oo{2}{J_1}$}

\psfrag{MT3}{$\mt_2$}
\psfrag{f3}{$f_2$}
\psfrag{cj0}{$\oo{0}{J_2}$}
\psfrag{cj1}{$\oo{1}{J_2}$}
\psfrag{cj2}{$\oo{2}{J_2}$}
\includegraphics[width=1.0\textwidth]{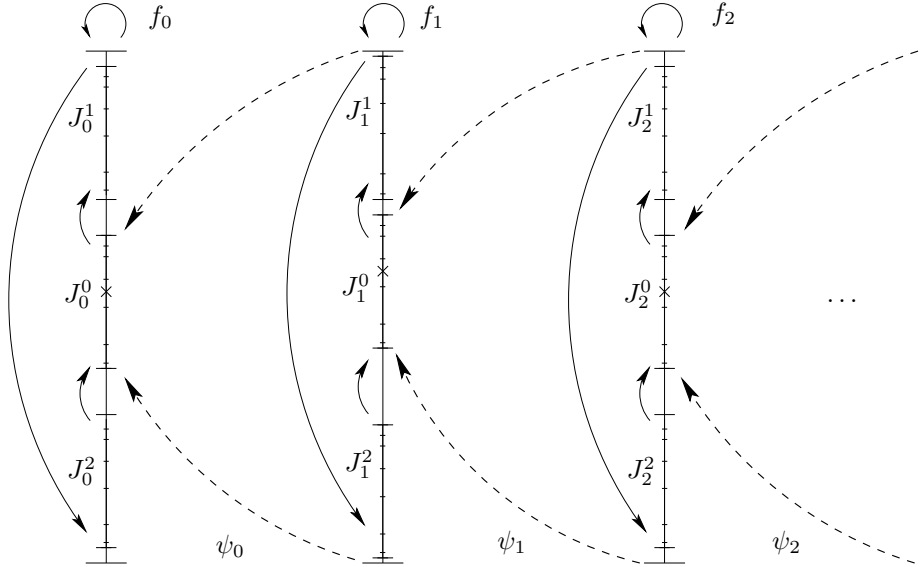}
\caption{The collection of scope maps $\oo{w}{\mt}_n$ for an infinitely renormalisable period-tripling unimodal map. Here $f_n$ denotes the $n$-th renormalisation of $f$}\label{fig:scopemaps-unimodal}
\end{figure}

Let $f\in\U_{\Omegax,\perm}$ have cycle $\{\oo{w}{J}\}_{w\in W}$. Consider the functions
\begin{equation}
\begin{array}{rl}
h_0^{-1}\colon \oo{0}{J}\to J, & \quad w=0, \\
h_0^{-1}\circ \o{p-w}{f}\colon \oo{w}{J}\to J, & \quad w=1,\ldots,p-1,
\end{array}
\end{equation}
where $h_0$ is the affine map satisfying $\RU f=h_0^{-1}\circ \o{p}{f}\circ h_0$.
The inverses of these maps are called the \emph{scope maps} of $f$ which we denote by $\oo{w}{\mt_f}\colon J\to \oo{w}{J}$.
 For each $w\in W$ we call $\oo{w}{\mt_f}\colon J\to \oo{w}{J}$ the \emph{$w$-scope map}. We denote the multi-valued function they
form by $\pf_f\colon J\to\bigcup_{w\in W} \oo{w}{J}$. Similarly, given an $n$-times renormalisable $f\in \U_{\Omegax,\perm}$ let $\oo{w}{\mt_n}=\oo{w}{\mt_{f_n}}$ denote the $w$-th scope function of $f_n$ and the multi-valued function they form by $\pf_n$.  The multi-valued function $\pf_*=\pf_{f_*}$ associated to the renormalisation fixed point $f_*$ is called the \emph{presentation function}.

If $f\in\U_{\Omegax,\perm}$ is infinitely renormalisable we can extend this construction by considering, for each $\word{w}{}=w_0\ldots w_n\in W^*$, the function $\oo{\word{w}{}}{\mt_f}=\oo{w_0}{\mt_0}\circ\cdots\circ\oo{w_n}{\mt_n}\colon J\to \oo{\word{w}{}}{J}$ and we set $\uline{\mt_f}=\{\oo{\word{w}{}}{\mt_f}\}_{\word{w}{}\in W^*}$.

\begin{prop}\label{prop:stepm}
Let $f_*$ denote the unimodal fixed point of renormalisation with presentation function $\pf_*$. Then, for each $\word{w}{}\in W^m$, the following properties hold,
\begin{enumerate}
\item\label{cor:stepm-switch} 
$\oo{\word{w}{}}{\mt_*}=\oo{-\word{q}{}(\word{w}{})}{f_*}\circ\ii_{J\to \oo{0^n}{J}}$;
\item 
$\oo{\word{w}{}}{\mt_*}(\bigcup_{\word{w}{}\in W^n}\oo{\word{w}{}}{J_*})\subset \bigcup_{\word{w}{}\in W^{n+m}}\oo{\word{w}{}}{J_*}$;
\end{enumerate}
\end{prop}
\begin{proof}
We show the first item inductively. Trivially it is true for $m=0$. Assume it holds for $\word{w}{}\in W^m$, some $m\geq 0$, and consider $w\word{w}{}\in W^{m+1}$. Since $\RU f_*=f_*$ implies $\oo{p}{f_*}\circ\ii_{J\to\oo{0}{J_*}}=\ii_{J\to\oo{0}{J_*}}\circ f_*$, we find
\begin{align}
\oo{w\word{w}{}}{\mt_*}
&=\oo{w}{\mt_*}\circ\oo{\word{w}{}}{\mt_*} \notag \\ 
&=\o{-(p-w)}{f_*}\circ\ii_{J\to \oo{0}{J_*}}\circ \o{-\word{q}{}(\word{w}{})}{f_*}\circ\ii_{J\to\oo{0^m}{J}} \notag \\
&=\o{-(p-w)}{f_*}\circ \o{-p\word{q}{}(\word{w}{})}{f_*}\circ\ii_{J\to\oo{0}{J}}\circ\ii_{J\to\oo{0^m}{J}} \notag \\
&=\o{-\word{q}{}(w\word{w}{})}{f_*}\circ\ii_{J\to\oo{0^{m+1}}{J}}
\end{align}
This proves the first statement. The second statement then follows from the first since, given $\word{w}{}\in W^n$ and $\word{\tilde w}{}\in W^m$, the image of $\oo{\word{w}{}}{J_*}$ under $\oo{\word{w}{}}{\mt_*}$ can be expressed as a preimage of $\oo{0^{m+n}}{J_*}$ under $f_*$.
\end{proof}
\begin{cor}\label{cor:step*}
Let $f_*$ denote the unimodal fixed point of renormalisation with presentation function $\pf_*$. 
Let $\Cantor_*$ denote the invariant Cantor set for $f_*$. 
Let $\uline{J}_*=\{\oo{\word{w}{}}{J}_*\}_{\word{w}{}\in W^*}$ denote the collection of all central cycles of all depths.
Then, for each $\word{w}{}\in W$, the following properties hold,
\begin{enumerate}
\item $\oo{\word{w}{}}{\mt_*}(\uline{J}_*)\subset \uline{J}_*$;
\item $\oo{\word{w}{}}{\mt_*}(\Cantor_*)\subset \Cantor_*$;
\end{enumerate}
\end{cor}
\begin{proof}
Simply take limits in the previous Proposition.
\end{proof}
\begin{lem}\label{lem:univ-linearisation} 
Let $f_*$ denote the unimodal fixed point of renormalisation with presentation function $\pf_*$. Then, for each $w\in W$, the following properties hold,
\begin{enumerate}
\item $\oo{w}{\mt_*}$ has a unique attracting fixed point $\alpha$;
\item if $[\oo{w^n}{\mt_*}]$ denotes the orientation preserving affine rescaling of $\oo{w^n}{\mt_*}$ to $J$ then $\oo{w}{u_*}=\lim_{n\to\infty}[\oo{w^n}{\mt_*}]$ exists, and
the convergence is exponential in the sup-norm on some complex domain containing $J$.
\end{enumerate}
\end{lem}
\begin{proof}
It is clear that there exists a fixed point $\alpha$, as $\oo{w}{\mt_*}$ maps $J$ into itself. It is also unique, since
by construction $\oo{w^{n+1}}{J_*}=\oo{w}{\mt_*}(\oo{w^{n}}{J_*})$, and all images of $J$ must contain all fixed points. However, Theorem~\ref{thm:real-ap-bounds} implies
$|\oo{w^{n}}{J_*}|\to 0$ as $n\to\infty$, and hence there can be only one fixed point.

Now let us show $\alpha$ is attracting.
Theorem~\ref{thm:real-ap-bounds} tells us, since $\oo{w^{n+1}}{J_*}\subset\oo{w^{n}}{J_*}$, that $|\oo{w^{n+1}}{J_*}|/|\oo{w^{n}}{J_*}|<k_1<1$. By the Mean Value Theorem this
implies there are points $\alpha_n\in \oo{w^n}{J_*}$ such that $|(\oo{w}{\mt_*})'(\alpha_n)|=|\oo{w^{n+1}}{J_*}|/|\oo{w^{n}}{J_*}|<k_1$. Also, since $\alpha\in \oo{w^n}{J_*}$
for all $n>0$ and $|\oo{w^{n}}{J_*}|\to 0$ as $n\to\infty$ , we have $\alpha_n\to \alpha$. As $\oo{w}{\mt_*}$ is analytic we must have $|(\oo{w}{\mt_*})'(\alpha)|<k_1$. Hence
$\alpha$ if $\alpha$ has multiplier $\sigma_w$, $|\sigma_w|<1$ and so $\alpha$ is attracting.

For the second item let $\oo{w}{u_n}=\ii_{\oo{w^n}{J_*}\to J}\circ \oo{w^n}{\mt_*}\colon J\to J$. 
First we claim that there is a domain $U\subset\CC$ containing $J$ on which $\oo{w}{u_n}$ has a univalent extension. This follows as $\alpha$ is an attracting fixed point and
$\oo{w}{\mt_*}$ is real-analytic on $J$, so there exists a domain $V\subset\CC$ containing $\alpha$ on which $\oo{w}{\mt_*}$ is univalent and $\oo{w}{\mt_*}(V)\subset V$. By Theorem~\ref{thm:real-ap-bounds} there exists an integer $N>0$ such that $\o{n}{(\oo{w}{\mt_*})}(J)\subset V$ for all $n\geq N$. 
Therefore take any domain $U$ containing $J$ such that $\o{n}{(\oo{w}{\mt_*})}(U)$ is bounded away from the set of the first $p^N$ critical values of $f_*$ for all $n<N$.
Then $\oo{w}{u_n}$ will be univalent on $U$.

Observe that, letting $\oo{w}{v_n}=\Z_{\oo{w^n}{J_*}}\oo{w}{\mt_*}$ where $\Z_T$ denotes the zoom operator on the interval $T$, we can write
\begin{equation}
\oo{w}{u_n}=\oo{w}{v_n}\circ\cdots\circ \oo{w}{v_0}
\end{equation}
Also observe that the argument above gives a domain $W$ containing $J$ on which each of these composants has a univalent extension. 
Therefore
\begin{align}
|\oo{w}{u_n}-\oo{w}{u_{n+1}}|_W
&=|\oo{w}{v_n}\circ\cdots\circ \oo{w}{v_0}-\oo{w}{v_{n+1}}\circ\oo{w}{v_n}\circ\cdots\circ \oo{w}{v_0}|_W \notag \\
&= |\id-\oo{w}{v_{n+1}}|_{\oo{w}{u_n}(W)}.
\end{align}
Theorem~\ref{thm:real-ap-bounds} implies $|\oo{w^n}{J_*}|\to 0$ exponentially as $n\to\infty$. Analyticity of $\oo{w}{\mt_*}$ then implies $\dis{\oo{w}{\mt_*}}{\oo{w^n}{J_*}}\to 0$ exponentially as well. Moreover, also by analyticity, this holds on a subdomain $W^n$ of $W$ containing $\oo{w^n}{J_*}$. Hence, by the Mean Value Theorem,
\begin{equation}
\left|1-\frac{|\oo{w^{n}}{J_*}|}{|\oo{w^{n+1}}{J_*}|}(\oo{w}{\mt_*})'\right|_{\oo{w^n}{J}}\to 0
\end{equation}
exponentially, and this will also hold on $W^n$ if $n>0$ is sufficiently large. Integrating then gives us
\begin{equation}
|\id-\oo{w}{v_n}|_{\oo{w}{u_n}(W)}\to 0
\end{equation}
exponentially, and hence $|\oo{w}{u_n}-\oo{w}{u_{n+1}}|_W\to 0$ exponentially. This implies the limit $\oo{w}{u_*}$ exists and is univalent on $W$.
\end{proof}
\begin{rmk}
In the period-doubling case more precise information was given by Birkhoff, Martens and Tresser in~\cite{BMT}. Since $f_*$ is convex in this case (see~\cite{Dav1}), we know $f_*|\oo{1}{J_*}$ is expanding and hence $\oo{1}{\mt_*}$ is contracting. This
simplified the construction of the renormalisation Cantor set for a strongly dissipative nondegenerate H\'enon-like map given by de Carvalho, Lyubich and Martens in~\cite{dCML}. 
\end{rmk}
\begin{rmk}
For $(ii)$ of the above it is clear this can be extended from words $w^\infty$ to words $\word{w}{}=w_0w_1\ldots$ which are eventually periodic. For arbitrary words a different strategy must be used, see~\cite{BMT}.
\end{rmk}
\begin{prop}\label{prop:inducedpermutation}
Let $\perm$ be a unimodal permutation and let $\perm(n)$ be the unimodal permutation satisfying $\RUgen{\perm}^n=\RUgen{\perm(n)}$.
Given an $n$-times renormalisable $f\in\U_{\Omegax,\perm}$ let 
\begin{equation}
f_{\perm,i}=\RUgen{\perm}^i f \ \ \mbox{and} \ \ f_{\perm(n)}=\RUgen{\perm(n)}f
\end{equation}
for all $i=0,1,\ldots,n$. Let $\pf_{\perm,i}$ denote the presentation function for $f_{\perm,i}$ with respect to $\RUgen{\perm}$ and let $\pf_{\perm(n)}$ denote the presentation function for $f_{\perm(n)}$ with respect to $\RUgen{\perm(n)}$. 
Then
\begin{equation}
\pf_{\perm(n)}=\{\oo{w_0}{\mt}_{\perm,0}\circ\ldots\circ\oo{w_n}{\mt}_{\perm,n}\}_{w_0,\ldots,w_n\in W}.
\end{equation}  
\end{prop}
\begin{proof}
This follows from the fact that $\RUgen{\perm} f_{\perm,i}=f_{\perm,i+1}$ implies $\oo{p}{f_{\perm,i}}\circ\ii_{J\to\oo{0}{J_i}}=\ii_{J\to\oo{0}{J_i}}\circ f_{\perm,i+1}$ and the fact that the central interval of $f$ under $\RUgen{\perm(n)}$ is equal to $\oo{0^n}{J}$.
\end{proof}
Since the renormalisation interval $\oo{0}{J}$ of $f$ is determined by a $p$-periodic point and its preimage and the scope maps $\oo{w}{\mt}$ are compositions of preimages of $f$, and all of these vary continuously with $f$ we get the following Proposition.
\begin{prop}\label{prop:1d-scope-conv} 
Let $\perm$ be a unimodal permutation. There exists a constant $C>0$ such that for any $f_0,f_1\in \U_{\Omegax,\perm}$, and any $w\in W$,
\begin{equation}
|\oo{w}{\mt_{f_0}}-\oo{w}{\mt_{f_{1}}}|_{\Omegax}\leq C|f_0-f_1|_{\Omegax}
\end{equation}
\end{prop}
The following corollary follows directly from the above Proposition~~\ref{prop:1d-scope-conv} and convergence of renormalisation, Theorem~\ref{thm:1d-exp-convergence}.
\begin{cor}\label{prop:1d-scope-conv2}
There exist constants $C>0$ and $0<\rho<1$ such that the following holds: given any infinitely renormalisable $f\in\U_{adapt}$,
\begin{equation}
|\oo{w}{\mt_{n}}-\oo{w}{\mt_*}|_{\Omegax}\leq C\rho^n.
\end{equation}
 \end{cor}

\subsection{A Reinterpretation of the Operator}\label{sect:1d-construction-revised}
Let us now consider $\H_{\Omega}(0)$, defined to be the space of maps $F\in C^\omega(B,B)$ of the form $F=(f\circ\pi_x,\pi_x)$ where $f\in \U_{\Omegax}$. Let us also consider
the subspace $\H_{\Omega,\perm}(0)$ of maps $F=(f\circ\pi_x,\pi_x)$ where $f\in \U_{\Omegax,\perm}$. These will be called the space of \emph{degenerate H\'enon-like maps} and
the space of \emph{renormalisable degenerate H\'enon-like maps of type $\perm$} respectively. The reasons for this will become apparent in Section~\ref{sect:henonlike} when we introduce
non-degenerate H\'enon-like maps.
Observe there is an imbedding $\uline{\i}\colon \U_{\Omegax}\to \H_{\Omega}(0)$, given by $\uline{\i}(f)=(f\circ\pi_x,\pi_x)$, which restricts to an imbedding $\uline{\i}\colon
\U_{\Omegax,\perm}\to \H_{\Omega,\perm}(0)$. We will construct an operator $\RH{}$, defined on $\H_{\Omega}(0)$, such that the
following diagram commutes.
\begin{equation}
\xymatrix{
\U_{\Omegax,\perm}\ar[d]_{\uline{\i}}  \ar[r]^\RU                & \U_{\Omegax}\ar[d]^{\uline{\i}} & \\
\H_{\Omega,\perm}(0)         \ar[r]_\RH & \H_{\Omega}(0)
}
\end{equation}
Let $f\in\U_{\Omegax,\perm}$, let $\{\oo{w}{J}\}_{w\in W}$ be its renormalisation cycle and let $\{\oo{w}{J'}\}_{w\in W}$ be the set of corresponding maximal
extensions.
Let $F=\uline{\i}(f)$ be the corresponding degenerate H\'enon-like map, 
let
\begin{equation}
\oo{w}{B}=\oo{w+1}{J}\times \oo{w}{J}, \quad \oo{w}{B'}=\oo{w+1}{J'}\times\oo{w}{J'}
\end{equation}
and let
\begin{equation}
\oo{w}{B_\diag}=\oo{w}{J}\times\oo{w}{J}, \quad \oo{w}{B_\diag'}=\oo{w}{J'}\times\oo{w}{J'},
\end{equation}
where $w\in W$ is taken modulo $p$.
Observe $\oo{w}{B}$ is invariant under $\o{p}{F}$ for each $w\in W$.

Consider the map $H\colon B\to B$ defined by $H=(\o{p-1}{f},\pi_y)$.
Since $H$ preserves horizontal lines, if it is not injective Rolles' theorem implies between any two points with the same image there exists a solution to $(\o{p-1}{f})'=0$. By the Inverse Function Theorem this will be locally invertible on any connected open set bounded away from set of points for which $DH$ is singular. This set coincides the \emph{critical locus} $\oo{p-1}{\Curve}=\{(x,y):(\o{p-1}{f})'(x)=0\}$. Hence $H$ is a diffeomorphism onto its image on any connected open set bounded away from $\oo{p-1}{\Curve}$.
In particular, since the box $\oo{0}{B}$ is bounded away from $\oo{p-1}{\Curve}$ whenever the maximal extensions are proper extensions, the map $H$ will be a
diffeomorphism there. We call $B^0$ the \emph{central box}. We call the map $H$ the \emph{horizontal diffeomorphism}. Observe that $\oo{0}{B_\diag}=H(\oo{0}{B})$. Define $\bar H\colon \oo{0}{B_\diag}\to
\oo{0}{B}$ to be the inverse of $H$ restricted to $\oo{0}{B_\diag}$.
\begin{rmk}
More generally given any map $T$ which we do not iterate but which is related to the dynamics of a map $F\colon B\to B$, if it is invertible we denote its inverse by $\bar T$.
We use this unconventional notation for following reason: later we consider maps of the form $F=(\phi,\pi_x)\colon B\to B$ and we define $\oo{w}{\phi}\colon B\to J$ by $\o{w}{F}(x,y)=(\oo{w}{\phi}(x,y),\oo{w-1}{\phi}(x,y))$ for all $w>0$ and for $w<0$ whenever they exist. Plain superscripts index objects related to the number of iterates of $F$ and plain subscripts index objects related to the number of renormalisations performed. However, we define $H=(\oo{p-1}{\phi},\pi_y)$, whose inverse is related to the $p-1$-st preimage of $F$, not the first preimage.
\end{rmk}
Since $\bar H$ is well-defined on $\oo{0}{B_\diag}$ the map
\begin{equation}
G=H\circ \o{p}{F}\circ \bar H\colon \oo{0}{B_\diag}\to \oo{0}{B_\diag},
\end{equation}
is also well defined.
We call this map the \emph{pre-renormalisation} of $F$ around $\oo{0}{B}$. 
Let $\III$ denote the affine bijection from $\oo{0}{B_\diag}$ onto $B$ such that the map 
\begin{equation}
\RH F=\III\circ G\circ \bar{\III}\colon B\to B
\end{equation}
is again a degenerate H\'enon-like map where $\bar\III$ denotes the inverse of $\III$. Then $\RH F$ is called the \emph{H\'enon renormalisation} of $F$ around $\oo{0}{B}$ and the operator $\RH$ is called the
\emph{renormalisation operator} on $\H_{\Omegax,\perm}(0)$. Observe that $\RH F=(\RU f\circ\pi_x,\pi_x)$ .
\begin{rmk}
By the same argument as above, $H$ will be a diffeomorphism onto its image when restricted to any $\oo{w}{B}$. 
Since $\oo{w}{B_\diag}=H(\oo{w}{B})$ and $\oo{w}{B}$ is invariant under $\o{p}{F}$ by construction, the maps
\begin{equation}
\oo{w}{G}=H\circ \o{p}{F}\circ \bar{H}\colon \oo{w}{B_\diag}\to \oo{w}{B_\diag}.
\end{equation}
are well defined.
We call $\oo{w}{G}$ the $w$-th \emph{pre-renormalisation}. There are affine bijections $\oo{w}{\III}$ from $\oo{w}{B_\diag}$ to $B$ such that
\begin{equation}
\RHgen{w} F= \oo{w}{\III}\circ \oo{w}{G}\circ\oo{w}{\bar{\III}}\colon B\to B
\end{equation}
is again a degenerate H\'enon-like map (where, as above, $\oo{w}{\bar{\III}}$ denotes the inverse of $\oo{w}{\III}$). Then the map $\RHgen{w} F$ is called the \emph{H\'enon
renormalisation} of $F$ around $\oo{w}{B}$ and the operator $\RHgen{w}$ is called the $w$-th \emph{renormalisation operator} on $\H_{\Omegax,\perm}(0)$. 
Observe that $\RHgen{w} F=(\RUgen{w} f\circ\pi_x,\pi_x)$, where $\RUgen{w}$ denotes the renormalisation around $\oo{w}{J}$.
\end{rmk}
\begin{landscape}
\begin{figure}[htp]
\centering

\psfrag{j0}{$\oo{0}{J}$}
\psfrag{j1}{$\oo{1}{J}$}
\psfrag{j2}{$\oo{2}{J}$}
\psfrag{b0}{$\oo{0}{B}$}
\psfrag{b1}{$\oo{1}{B}$}
\psfrag{b2}{$\oo{2}{B}$}
\psfrag{b0diag}{$\oo{0}{B}_\diag$}
\psfrag{b1diag}{$\oo{1}{B}_\diag$}
\psfrag{b2diag}{$\oo{2}{B}_\diag$}
\psfrag{H}{$H$}
\psfrag{Hi}{$\bar H$}

\psfrag{I}{$I$}
\psfrag{line0}{$Im(F) \ \mbox{or} \ Im(\RH F)$}
\psfrag{line1}{$Im(G)$}
\psfrag{line2}{the strip $S$}
\psfrag{line3}{the renormalisation cycle}

\includegraphics[scale=0.37]{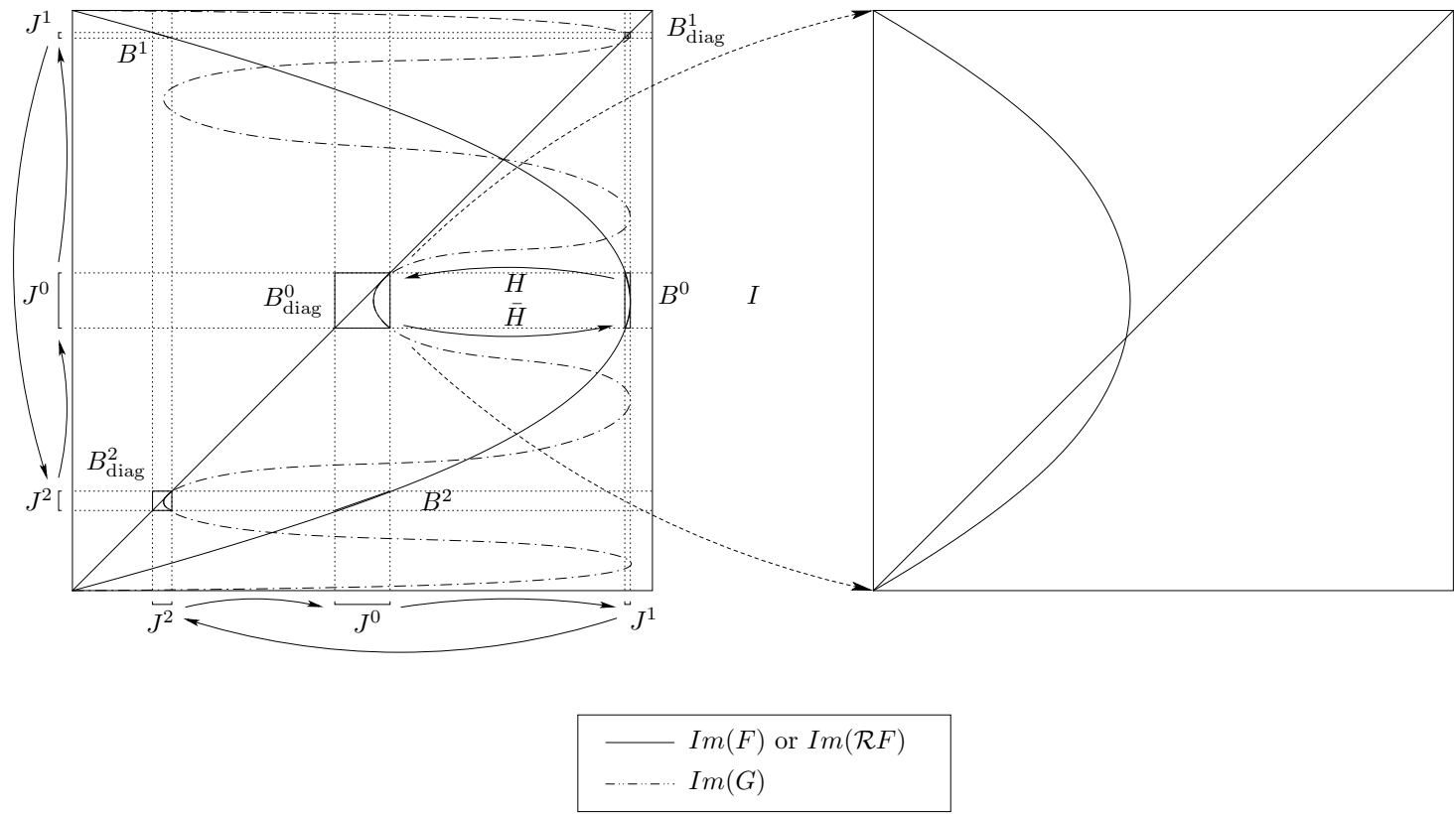}
\caption{A period-three renormalisable unimodal map considered as a degenerate H\'enon-like map. In this case the period is three. Observe that the image of the pre-renormalisation lies on the smooth curve $(\o{3}{f}(x),x)$.}\label{fig:renorm-degnenerate-henon}
\end{figure}
\end{landscape}

\begin{rmk}
The affine bijections $\oo{w}{\III}$ in the remark above map squares to squares. Hence the linear part of $\oo{w}{\III}$ has the form
\begin{equation}
\pm\iibyii{\oo{[w]}{\sigma}}{0}{0}{\pm \oo{[w]}{\sigma}}
\end{equation}
for some $\oo{[w]}{\sigma}>0$. Here the sign depends upon the combinatorial type of $\perm$ only. We call the quantity $\oo{[w]}{\sigma}$ the $w$-th \emph{scaling ratio} of $F$.   
\end{rmk}
\begin{rmk}\label{rmk:degenerate-henon-fixpoint}
Since $\uline{\ii}$ is an imbedding preserving the actions of $\RU$ and $\RH$ it is clear that $\RH$ also has a unique fixed point $F_*$. It must have the form
$F_*=(f_*\circ\pi_x,\pi_x)$ where $f_*$ is the fixed point of $\RU$. Then $F_*$ also has a codimension one stable manifold and dimension one local unstable manifold in $\H_{\Omegax}(0)$.
\end{rmk}
Given $F=\uline{\i}(f)\in\H_{\Omega,\perm}(0)$ let $\MT=\bar{H}\circ\bar{\III}\colon B\to \oo{0}{B}$ and $\oo{w}{\MT}=\o{w}{F}\circ\MT\colon B\to\oo{w}{B}$. Then
$\oo{w}{\MT}$ is called the \emph{$w$-th scope map} of $F$. The reason for this terminology is given by the following Proposition.
\begin{prop}
Let $F=\uline{\i}(f)\in\H_{\Omega,\perm}(0)$. Then
\begin{equation}
\oo{w}{\MT_{F}}(x,y)=
\left\{\begin{array}{ll}
(\oo{0}{\mt}_f(x),\oo{p-1}{\mt}) & w=p-1 \\
(\oo{w+1}{\mt}_f(x),\oo{w}{\mt}_f(x)) & 0<w<p-1 \\ 
(\oo{1}{\mt}_f(x),\oo{0}{\mt}_f(y)) & w=0\end{array}\right.,
\end{equation}
where $\oo{w}{\mt_f}$ denotes the $w$-th scope function for $f$.
\end{prop}
\begin{proof}
Observe that $\bar{H}(x,y)=(\o{-p+1}{f}(x),y)$ and 
\begin{equation}
\o{w}{F}(x,y)=\left\{\begin{array}{ll}(\o{w}{f}(x),\o{w-1}{f}(x))& w>0 \\ (x,y)& w=0\end{array}\right.
\end{equation}
which implies
\begin{align}
\o{w}{F}\circ \bar{H}(x,y)
&=\left\{\begin{array}{ll}(\o{-p+w+1}{f}(x),\o{-p+w}{f}(x))& w>0 \\ (\o{-p+1}{f}(x),y)& w=0\end{array}\right.
\end{align}
where appropriate branches of $\o{-p+w+1}{f}$ and $\o{-p+w}{f}$ are chosen. Also observe $\bar{\III}(x,y)=(\ii_{J\to\oo{0}{J}}(x),\ii_{J\to\oo{0}{J}}(y))$. Composing these gives
us the result.
\end{proof}
\begin{rmk}
Only the zero-th scope map $\MT=\oo{0}{\MT}$ is a diffeomorphism onto its image.
\end{rmk}

Now assume $F\in \H_{\Omega,\perm}$ is $n$-times renormalisable and denote its $n$-th renormalisation $\RH^n F$ by $F_n$. For each $F_n$ we can construct the $w$-th scope
map $\oo{w}{\MT_n}=\oo{w}{\MT}(F_n)\colon \Dom(F_{n+1})\to\Dom(F_n)$, where $\Dom(F_n)=B$ denotes the domain of $F_n$. For $\word{w}{}=w_0\ldots w_n\in W^*$ the
function
\begin{equation}
\oo{\word{w}{}}{\MT}=\MT_0^{w_0}\circ\ldots\circ \MT_n^{w_n}\colon \Dom(F_{n+1})\to \Dom(F_0)
\end{equation}
is called the \emph{$\word{w}{}$-scope map}. Let $\uline\MT=\{\MT^{\word{w}{}}\}$ denote the collection of all scope maps.
The following Corollary is an immediate consequence of the above Proposition.
\begin{cor}
Let $F=\uline{\i}(f)\in\H_{\Omega,\perm}(0)$ be an $n$-times renormalisable degenerate H\'enon-like map. Then given a word $\word{w}{}=w_0\ldots,w_{n-1}\in W^n$,
\begin{equation}
\oo{\word{w}{}}{\MT_{f}}(x,y)=\left\{\begin{array}{ll}(\oo{\word{w}{}+1^n}{\mt}_f(x),\oo{\word{w}{}}{\mt}_f(x))& \word{w}{}\neq 0 \\
(\oo{\word{w}{}+1^n}{\mt}_f(x),\oo{\word{w}{}}{\mt}_f(y)) &
\word{w}{}=0\end{array}\right.,
\end{equation}
where $\oo{\word{w}{}}{\mt_f}$ denotes the $\word{w}{}$-th scope map for $f$.
\end{cor}
In particular we may do this for $F_*$, the renormalisation fixed point. This gives 
\begin{align}
\oo{w}{\MT_*}(x,y)
&=\left\{\begin{array}{ll}(\oo{w+1}{\mt}_*(x),\oo{w}{\mt}_*(x))& w>0 \\ (\oo{w+1}{\mt}_*(x),\oo{w}{\mt}_*(y)) & w=0\end{array}\right.,
\end{align}
where $\oo{w}{\mt}_*$ are the branches of the presentation function.
We will denote the family of scope maps for $F_*$ by $\uline\MT_*=\{\oo{\word{w}{}}{\MT_*}\}_{\word{w}{}\in W^*}$ where $\oo{\word{w}{}}{\MT_*}\colon B\to \oo{\word{w}{}}{B_*}$ is constructed as above.

\section{H\'enon-like Maps}\label{sect:henonlike}
In this section we construct a space of H\'enon-like maps and a Renormalisation operator acting on it which coincides with the renormalisation operator on degenerate maps. We show that the standard unimodal renormalisation picture can be extended to the space of such maps if the H\'enon-like maps are sufficiently dissipative. We then examine the dynamics of infinitely renormalisable maps  in more detail.
\subsection{The Space of H\'enon-like Maps}
Let $\bar\e>0$. Let $\T_{\Omega}(\bar\e)$ denote the space of maps $\e\in C^\omega(B,\RR)$, which satisfy 
\begin{enumerate}
\item $\e(x,0)=0$ for all $x\in J$;
\item $\e(x,y)\geq 0$ for all $(x,y)\in B$;
\item $\e$ admits a holomorphic extension to $\Omega$;
\item $|\e|_{\Omega}\leq \bar\e$, where $|\!-\!|_{\Omega}$ denotes the sup-norm on $\Omega$. 
\end{enumerate}
Such maps will be called \emph{thickenings} or \emph{$\bar\e$-thickenings}. Let $B'=J'\times J'\subset \RR^2$ for some closed interval $J'\subset \RR$. 
Given $\e'\in C^\omega(B',\RR)$ let $E'(x,y)=(x,\e'(x,y))$. If there is an affine bijection $\III\colon B'\to B$ such that $E(x,y)=\III\circ E'\circ \bar{\III}(x,y)=(x,\e(x,y))$ where $\e$ is a thickening, then we say $\e'$ is a \emph{thickening on $B'$}.

Let $f\in\U_{\Omegax}$ be a unimodal map and let $\bar\e>0$ be a constant. Let
\begin{equation}
\H_{\Omega}^{-}(f,\bar\e)=\{F\in\Emb^\omega(B,\RR^2): F(x,y)=(f(x)+\e(x,y),x), \e\in\T_{\Omega}(\bar\e)\}
\end{equation}
and
\begin{equation}
\H_{\Omega}^{+}(f,\bar\e)=\{F\in\Emb^\omega(B,\RR^2): F(x,y)=(f(x)-\e(x,y),x), \e\in\T_{\Omega}(\bar\e)\}.
\end{equation}
Note that such maps to not necessarily leave $B$ invariant. Typically we take extensions of these $F$ to some rectangular domain $B'\subset \Omega$ containing $B$ on which $F$ is invariant. Also note that $\H_{\Omega}^{-}(f,\bar\e)$ consists of orientation reversing maps and $\H_{\Omega}^{+}(f,\bar\e)$ consists of orientation preserving maps.

Given these spaces, let
\begin{equation}
\H_{\Omega}^{\pm}(\bar\e)=\bigcup_{f\in\U_{\Omegax}}\H_{\Omega}^{\pm}(f,\bar\e)
\end{equation}
and
\begin{equation}
\H_{\Omega}^{\pm}=\bigcup_{\bar\e>0}\H_{\Omega}^{\pm}(\bar\e).
\end{equation}
and finally set $\H_{\Omega}=\H_{\Omega}^{+}\cup\H_{\Omega}^{-}$.
Observe that the condition $\e(x,0)=0$ ensures that each H\'enon-like map $F$ has a unique representation as $F=(f-\e,\pi_x)$. We will call this representation the \emph{parametrisation} of $F$. We will write $F=(\phi,\pi_x)$ when the
parametrisation is not explicit. Observe that the degenerate H\'enon-like maps considered in
Section~\ref{sect:unimodal} lie in a subset of the boundary of $\H_{\Omega}(\bar\e)$ for all $\bar\e>0$. Given a square $B'\subset \RR^2$ a map $F\in\Emb^\omega(B',B')$ is 
\emph{H\'enon-like on $B'$} if there exists an affine bijection $\III\colon B'\to B$ such that $\III\circ F\circ\bar{\III}\colon B\to B$ is a H\'enon-like map.

Given a H\'enon-like map $F=(\phi,\pi_x)\colon B\to F(B)$ its inverse will have the form $\o{-1}{F}=(\pi_y,\oo{-1}{\phi})\colon F(B)\to B$ where $\oo{-1}{\phi}\colon F(B)\to J$ satisfies
\begin{equation}
\pi_y=\oo{-1}{\phi}(\phi,\pi_x);\quad \pi_x=\phi(\pi_y,\oo{-1}{\phi}).
\end{equation}
More generally, given an integer $w>0$ let us denote the $w$-th iterate of $F$ by $\o{w}{F}\colon B\to B$, and the $w$-th preimage by $\o{-w}{F}\colon \o{w}{F}(B)\to B$. Observe that  they have the respective forms $\o{w}{F}=(\oo{w}{\phi},\oo{w-1}{\phi})$  and $\o{-w}{F}=(\oo{-w+1}{\phi},\oo{-w}{\phi})$ for some functions $\oo{w}{\phi}\colon B\to J$ and $\oo{-w}{\phi}\colon \o{w}{F}(B)\to J$.
We then define the \emph{$w$-th critical curve} or \emph{critical locus} to be the set $\oo{w}{\Curve}=\oo{w}{\Curve}(F)=\{\del_x\oo{w}{\phi}(x,y)=0\}$.

\subsection{The Renormalisation Operator}
Let us consider the operators $\RU$ and $\RH$ from Section~\ref{sect:1d-construction-revised}.
Observe that $\RU$ is constructed as some iterate under an affine coordinate change whereas $\RH$ uses non-affine coordinate changes
. That they are equivalent is a coincidence that we now exploit.

Our starting point is that non-trivial iterates of non-degenerate $F\in \H_\Omega$ will most likely not have the form
$(f\circ\pi_x\pm\e,\pi_x)$ after affine rescaling. Therefore, unlike the one dimensional case, we will need to perform a `straightening' via a non-affine change of coordinates. 

\begin{defn}
Let $p>1$ be an integer. A map $F\in \H_{\Omega}$ is \emph{pre-renormalisable with period $p$} if there exists a closed topological disk $\oo{0}{B}\subset B$ with $\o{p}{F}(\oo{0}{B})\subset \oo{0}{B}$.
The domain $\oo{0}{B}$ is called the \emph{central box}. The topological disks $\oo{w}{B}=\o{w}{F}(\oo{0}{B})$, $w\in W$, will be called the \emph{boxes} and the collection $\uline B=\{\oo{w}{B}\}_{w\in W}$ will be called the \emph{cycle}.
\end{defn}
\begin{defn}\label{def:henon-renormalisation}
Let $p>1$ be an integer. A map $F\in\H_{\Omega}$ is \emph{renormalisable with period $p$} if the following properties hold,
\begin{enumerate}
\item\label{property:pre-renorm}
$F$ is pre-renormalisable with period $p$;
\item\label{property:conjugacy}
there exists diffeomorphism onto its image $H\colon \oo{0}{B}\to \oo{0}{B_\diag}$, where $\oo{0}{B}$ is the pre-renormalisation domain of $F$ and $\oo{0}{B_\diag}$ is a square symmetric about the diagonal, such that
\begin{equation}
G=H\circ \o{p}{F}\circ\bar H\colon \oo{0}{B_\diag}\to\oo{0}{B_\diag}
\end{equation}
is H\'enon-like on $\oo{0}{B_\diag}$.
\end{enumerate}
The map $G$ is called the \emph{pre-renormalisation} of $F$ with respect to $H$.
By definition, there exists an affine map $\III\colon \oo{0}{B_\diag}\to B$ such that
\begin{equation}
\RH F=\III\circ G\circ \bar{\III}\colon B\to B
\end{equation}
is an element of $\H_\Omega$, where $G$ denotes the pre-renormalisation of $F$.
The map $\RH F$ is called the \emph{H\'enon-renormalisation} of $F$. We denote the space of all renormalisable maps by $\H_{\Omega,p}$.  The operator $\RH\colon
\H_{\Omega,p}\to\H_\Omega$ is called the \emph{H\'enon-renormalisation operator} or simply the \emph{renormalisation operator} on $\H_\Omega$. The
absolute value of the eigenvalues
of the linear part of $\bar\III$ (which coincide as it maps a square box to a square box) is called the \emph{scaling ratio} of $F$.
\end{defn}
\begin{rmk}
We denote the subspace of $\H_{\Omega,p}$ consisting of renormalisable maps expressible as $F=(f\pm\e,\pi_x)$, where $|\e|_\Omega<\bar\e$, by $\H_{\Omega,p}(f,\bar\e)$ and let $\H_{\Omega,p}(\bar\e)=\bigcup_{f\in\U_{\Omegax}}\H_{\Omega,p}(f,\bar\e)$ denote their union.
\end{rmk}
There are, a priori, many coordinate changes $H$ satisfying these properties. However, we now choose one canonically which has sufficient dynamical meaning.  
By analogy with the degenerate case, consider the map $H=(\oo{p-1}{\phi},\pi_y)$. As $H$ preserves horizontal lines between two preimages of the same point there lies a solution to $\del_x\oo{p-1}{\phi}=0$. The Inverse Function Theorem then tells us this is a local diffeomorphism away from the
\emph{critical locus} $\oo{p-1}{\Curve}=\{(x,y)\in B : \del_x\oo{p-1}{\phi}(x,y)=0\}$. Also, since it maps horizontal lines to horizontal lines, it must be a diffeomorphism onto its image. Hence, abusing terminology slightly, we will call this map the \emph{horizontal diffeomorphism}
associated to $F$.

Also consider the map $V=\o{p-1}{F}\circ \bar{H}\colon H(\oo{0}{B})\to \oo{p-1}{B}$. Since $\o{p-1}{F}$ is a diffeomorphism onto its image everywhere and $H$
is a diffeomorphism onto its image when restricted to $\oo{0}{B}$ we find that $V$ is also a diffeomorphism onto its image. We will call $V$ the \emph{vertical diffeomorphism}.
The reason for considering the maps $H$ and $V$ is given by the following Proposition. 
\begin{prop}\label{prop:horizdiffeo}
Let $F=(\phi,\pi_x)\in \H_{\Omega}$. Assume that, for some integer $p>1$, the following properties hold,
\begin{enumerate}
\item $\oo{0}{B}\subset B$ is a subdomain on which $\o{p}{F}$ is invariant;
\item the horizontal diffeomorphism $H=(\oo{p-1}{\phi},\pi_y)$ is a
diffeomorphism onto its image when restricted to $\oo{0}{B}$.
\end{enumerate}
Then $H\circ\o{p}{F}\circ\bar H\colon H(\oo{0}{B})\to H(\oo{0}{B})$ has the form
\begin{equation}
H\circ\o{p}{F}\circ\bar H(x,y)=(\oo{p}{\phi}\circ V(x,y),x)
\end{equation}
where $V$ is the vertical diffeomorphism described above. Moreover, the vertical diffeomorphism has the form $V(x,y)=(x,v(x,y))$ for some $v\in C^\omega (B,J)$. 
\end{prop}
\begin{proof}
Observe $\bar H$ has the form $\bar H=(\oo{p-1}{\bar\phi},\pi_y)$ for some $\oo{p-1}{\bar\phi}\colon H(\oo{0}{B})\to \RR$. 
Equating $\o{p-1}{F}\circ F$ with $\o{p}{F}$ implies $\oo{p-1}{\phi}(\phi,\pi_x)=\oo{p}{\phi}$. Equating $H\circ\bar H$ and $\bar H\circ H$ with the identity implies
\begin{equation}
\pi_x=\oo{p-1}{\phi}(\oo{p-1}{\bar{\phi}},\pi_y)=\oo{p-1}{\bar{\phi}}(\oo{p-1}{\phi},\pi_y).
\end{equation}
Hence, by definition of $H$ and $V$,
\begin{equation}
H\circ F=(\oo{p-1}{\phi}(\phi,\pi_x),\pi_x) =(\oo{p}{\phi},\pi_x)
\end{equation}
and
\begin{equation}
V=\o{p-1}{F}\circ\bar H=(\oo{p-1}{\phi}(\oo{p-1}{\bar\phi},\pi_y),\oo{p-2}{\phi}(\oo{p-1}{\bar\phi},\pi_y))=(\pi_x,\oo{p-2}{\phi}(\oo{p-1}{\bar\phi},\pi_y)). 
\end{equation}
Therefore, if we set $v(x,y)=\oo{p-2}{\phi}(\oo{p-1}{\bar\phi},\pi_y)$ the result is shown.
\end{proof}
We now show that maps satisfying the hypotheses of the above Proposition exist, are numerous and in fact renormalisable in the sense described above.
More precisely, we show that $\RH$ is defined on a tubular neighbourhood of $\H_{\Omega,\upsilon}(0)$ in the closure of $\H_{\Omega}$.
This is essentially a perturbative result. To do this we need the following.
\begin{prop}[variational formula of the first order]\label{prop:var-formula} 
Let $F\in\H_\Omega$ be expressible as $F=(\phi,\pi_x)$ where $\phi(x,y)=f(x)+\e(x,y)$. 
Then, for all $w\in W$,
\begin{align}
\oo{w}{\phi}(x,y)
&=\o{w}{f}(x)+\oo{w}{L}(x)+\e(x,y)(\o{w}{f})'(x)+\bigo(\bar\e^2)
\end{align}
where
\begin{align}
\oo{w}{L}(x)=
&\e(\o{w-1}{f}(x),\o{w-2}{f}(x))+\e(\o{w-2}{f}(x),\o{w-3}{f}(x))f'(\o{w-1}{f}(x)) \notag \\
&+\ldots + \e(f(x),x)\prod_{i=1}^{w-1}f'(\o{i}{f}(x))
\end{align}
\end{prop}
\begin{proof}
We proceed by induction. Assume this holds for all integers $0<i<w$ and let $\oo{w}{L}(x)$ be as above.
Then
\begin{align}
\oo{w}{\phi}(x,y)
&=\phi(\oo{w-1}{\phi}(x,y),\oo{w-2}{\phi}(x,y)) \notag \\
&=f(\oo{w-1}{\phi}(x,y))+\e(\oo{w-1}{\phi}(x,y),\oo{w-2}{\phi}(x,y))
\end{align}
but observe, by Taylors' Theorem,
\begin{align}
f(\oo{w-1}{\phi}(x,y))
&=f(\o{w-1}{f}(x)+\oo{w-1}{L}(x)+\e(x,y)(\o{w-1}{f})'(x)+\bigo(\bar\e^2)) \notag \\
&=\o{w}{f}(x)+f'(\o{w-1}{f}(x))\oo{w-1}{L}(x)+\e(x,y)(\o{w}{f})'(x)+\bigo(\bar\e^2)
\end{align}
\begin{landscape}
\begin{figure}[htp]
\centering
\psfrag{j0}{$\oo{0}{J}$}
\psfrag{j1}{$\oo{1}{J}$}
\psfrag{j2}{$\oo{2}{J}$}
\psfrag{b0}{$\oo{0}{B}$}
\psfrag{b1}{$\oo{1}{B}$}
\psfrag{b2}{$\oo{2}{B}$}
\psfrag{b0diag}{$\oo{0}{B}_\diag$}
\psfrag{b1diag}{$\oo{1}{B}_\diag$}
\psfrag{b2diag}{$\oo{2}{B}_\diag$}
\psfrag{H}{$H$}
\psfrag{Hi}{$\bar H$}

\psfrag{I}{$I$}
\psfrag{line0}{$Im(F) \ \mbox{or} \ Im(\RH F)$}
\psfrag{line1}{$Im(G)$}
\psfrag{line2}{the strip $S$}
\psfrag{line3}{the renormalisation cycle}

\includegraphics[scale=0.385]{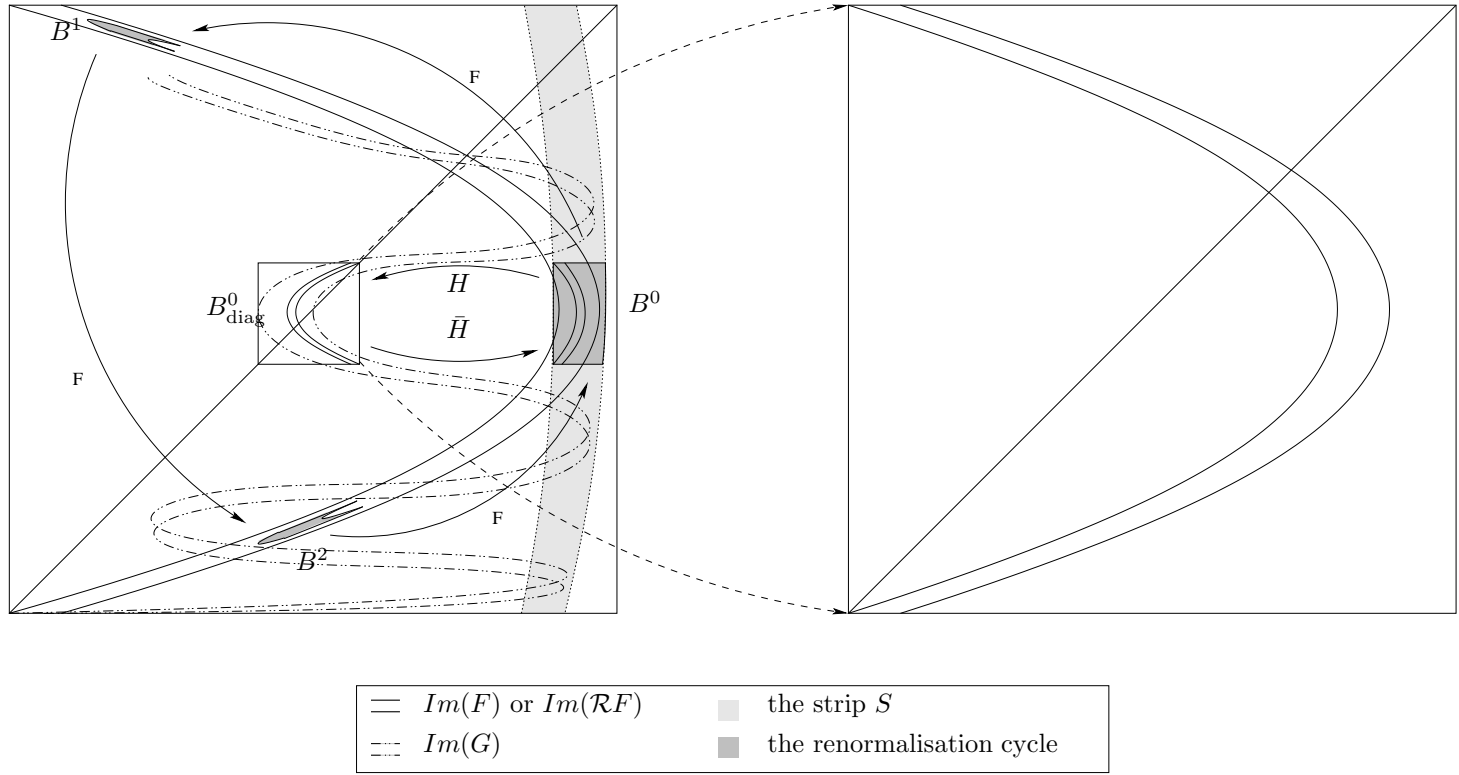}
\vspace*{0.1cm}
\caption{A renormalisable H\'enon-like map which is a small perturbation of a degenerate H\'enon-like map. In this case the combinatorial type is period tripling. Here the lightly shaded region is the preimage of the vertical strip through $\oo{0}{B}_\diag$. The dashed lines represent the image of the square $B$ under the pre-renormalisation $G$. If the order of all the critical points of $\o{2}{f}$ is the same it can be shown that $G$ can be extended to an embedding on the whole of $B$, giving the picture above.}\label{fig:renorm-henon}
\end{figure}
\end{landscape}

and
\begin{align}
\e(\oo{w-1}{\phi}(x,y),\oo{w-2}{\phi}(x,y))
&=\e((\o{w-1}{f}(x),\o{w-2}{f}(x))+\bigo(\bar\e)) \notag \\
&=\e(\o{w-1}{f}(x),\o{w-2}{f}(x))+\bigo(\bar\e^2),
\end{align}
where we have used, since $\e$ is analytic, that all derivatives of $\e$ are of the order $\bar\e$. Combining these gives us the result.
\end{proof}
\begin{prop}\label{prop:henon-perturb-invariantdomain}
Let $p>1$ be an integer. Let $F\in\H_{\Omega}$, let $\oo{0}{B}\subset B$ be a pre-renormalisation domain of period $p$ and let $G$ be its pre-renormalisation. Assume
\begin{itemize}
\item $\pi_xG(\oo{0}{B}_\diag)\subsetneq\pi_x(\oo{0}{B}_\diag)$;
\item $G$ is H\'enon-like on $\oo{0}{B}_\diag$.
\end{itemize} 
Then there exists a neighbourhood $\mathcal N\subset \H_{\Omega}$ of $F$ such that $\tilde F\in \mathcal N$ implies 
\begin{enumerate}
\item $\tilde F$ has a pre-renormalisation domain with the same properties; 
\item there exists a constant $C>0$, depending upon $f$ only, such that
\begin{equation}
\dist_{Haus}(\oo{0}{B}_\diag,\oo{0}{\tilde B}_\diag)<C|F-\tilde F|_{\Omega};
\end{equation}
and
\begin{equation}
\dist_{Haus}(\oo{0}{\Omega}_\diag,\oo{0}{\tilde \Omega}_\diag)<C|F-\tilde F|_{\Omega};
\end{equation}
\end{enumerate}
\end{prop}
\begin{proof}
Given $F=(\phi,\pi_x)$ satisfying our hypotheses let $H$ denote its horizontal diffeomorphism and $V$ its vertical diffeomorphism.
Let $G=(\varphi,\pi_x)$ denote its pre-renormalisation. Let $\oo{0}{B_\diag}=\oo{0}{J}\times\oo{0}{J}$.
Let $g_{\pm}(x)=\varphi(x,\del^{\pm}\oo{0}{J})$ be the two bounding curves of the image of $G$.

Similarly, given $\tilde F=(\tilde\phi,\pi_x)\in \H_{\Omega}$ let $\tilde H$ denote its horizontal diffeomorphism and $\tilde V$  its vertical diffeomorphism.
Let $\tilde G=(\tilde\varphi,\pi_x)$ denote its pre-renormalisation. These all vary continuously with $\tilde F$.

Observe that $G$ has a fixed point $(\alpha,\alpha)\in\del\oo{0}{B}_\diag$. 
Observe also that $\alpha\in\del\oo{0}{J}$ is a fixed point for $g_{-}$ which, by assumption, is expanding.
Let $\beta\in\del\oo{0}{J}$ be the other boundary component. Then $\beta$ is a preimage of $\alpha$ under $g_{-}$ and has non-zero derivative.  
The image of the horizontal line through $(\alpha,\alpha)$ intersects the diagonal $\{x=y\}$ tranversely at $(\alpha,\alpha)$. 
These properties are all open conditions.
Hence there exists a neighbourhood $\mathcal N_0\subset \H_\Omega$ of $F$ such that $\tilde F\in \mathcal N_0$ implies $\tilde F$ also has these properties once we set $\tilde
g_{-}(x)=\tilde\varphi(x,\tilde\alpha)$. If we let $\oo{0}{\tilde J}=[\tilde\alpha,\tilde\beta]$ then it is clear $\tilde{g}_{-}$ is unimodal on $\oo{0}{\tilde J}$.

Now let $\oo{0}{\tilde B}_\diag=\oo{0}{\tilde J}\times\oo{0}{\tilde J}$ and $g_{+}(x)=\tilde\varphi(x,\tilde\beta)$. Since $\pi_x(G(\oo{0}{B}_\diag))\subsetneq
\pi_x(\oo{0}{B}_\diag)$, the critical value of $g_{+}$ lies in $\interior(\oo{0}{J})$. Since the critical value of $g_{+}$ and $\del\oo{0}{J}$ depend continuously
on $F$, there exists a neighbourhood $\mathcal N_1\subset \mathcal N_0$ such that $\tilde F\in \mathcal N_1$ implies the critical value of $\tilde g_{+}$ lies in $\interior(\oo{0}{J})$. Hence
$\oo{0}{\tilde B}_\diag$ is $\tilde G$-invariant and $\pi_x(\tilde G(\oo{0}{\tilde B}_\diag))\subset\pi_x(\oo{0}{\tilde B}_\diag)$.

For the second statement observe the horizontal diffeomorphisms $H$ and $\tilde H$ will map diffeomorphically onto $\oo{0}{B}_\diag$ and $\oo{0}{\tilde B}_\diag$. Hence the Hausdorff distance will only depend on the distance between $\oo{0}{B}$ and $\oo{0}{\tilde B}$ and on the distance between $H$ and $\tilde H$. Both of these, in turn depend on $|F-\tilde F|_\Omega$.  Finally, the existence of the affine bijection $\tilde I\colon \oo{0}{\tilde B}_\diag\to B$ is clear.  
\end{proof}
\begin{prop}\label{prop:henon-perturb-renormalisation}
Let $p>1$ be an integer.
Let $F\in\H_{\Omega,p}$ be renormalisable of combinatorial type $p$. Let $\oo{0}{B}\subset B$ be the pre-renormalisation domain of period $p$ and let $G$ be its pre-renormalisation. Assume
\begin{itemize}
\item $\pi_xG(\oo{0}{B}_\diag)\subsetneq\pi_x(\oo{0}{B}_\diag)$;
\item $G$ is H\'enon-like on $\oo{0}{B}_\diag$.
\end{itemize}
Then there exists a neighbourhood $\mathcal N\subset \H_\Omega$ of $F$ and a constant $C>0$, depending upon $F$ only, such that $F\in \mathcal N$ implies
\begin{enumerate}
\item
$\tilde F$ is $p$-renormalisable with the same properties;
\item
there exists a constant $C>0$, depending upon $f$ only, such that
\begin{equation}
|\RH F-\RH\tilde F|_{\Omega}<C|F-\tilde F|_{\Omega}.
\end{equation}
\end{enumerate}
\end{prop}
\begin{proof}
Given $F\in \H_{\Omega,p}$ let $\oo{0}{B}$ denote its pre-renormalisation domain and $H$ denote its horizontal diffeomorphism. For each $w\in W$ let $\oo{w}{B}=\o{w}{F}(\oo{0}{B})$. 
Let $\mathcal N_0$ denote the neighbourhood of $F$ from Proposition~\ref{prop:henon-perturb-invariantdomain}. 
Given $\tilde F\in \mathcal N_0$  let $\oo{0}{\tilde B}$ denote its pre-renormalisation domain, let $\tilde H$ denote the horizontal diffeomorphism and for each $w\in W$ let $\oo{w}{\tilde B}=\o{w}{\tilde F}(\tilde H(\oo{0}{\tilde B}))$. 
Let $\tilde\Curve^{p-1}$ denote its critical curve.

For the first assertion observe the set $\tilde\Curve^{p-1}$ and the domain $\oo{0}{\tilde B}$ vary continuously with $\tilde F$. As $\Curve^{p-1}$ and the domain 
$\oo{0}{B}_\diag$ are separated by some distance $\gamma$, there is a neighbourhood $\mathcal N_1\subset \mathcal N_0$ of $F$ such that $\tilde F\in \mathcal N_1$ implies $\tilde\Curve^{p-1}$ and $\oo{0}{\tilde B}_\diag$
are separated by a distance of $\gamma/2$ or greater.

For the second assertion observe that $\tilde H$ and $\oo{0}{\tilde B}$ vary continuously with $\tilde F$. Hence $\oo{0}{\tilde B}_\diag$, and therefore $\tilde I$, will also vary continuously with $\tilde F$. As $\oo{0}{B}$ is bounded away from $\tilde\Curve^{p-1}$ for $\bar\e>0$ sufficiently small the result follows.
\end{proof}
The previous two results are quite general, in that they deal with perturbations of any renormalisable H\'enon-like map, not just perturbations of the degenerate maps. However, now we turn our attention to this particular case.
\begin{cor}\label{cor:henon-perturbedunimodal}
Let $\perm$ be a unimodal permutation of length $p>1$.
Let $F=\uline{\i}(f)\in \H_{\Omegax,\perm}(0)$.
Then there exist constants $C, \bar\e_f>0$ and a domain $\Omega'\subset\CC^2$ such that  for any $0<\bar\e<\bar\e_f$,  $\tilde F\in \H_{\Omega}(f,\bar\e)$ implies: 
\begin{enumerate}
\item  $\tilde F\in \H_{\Omega,p}(f,\bar\e)$;
\item $\RH \tilde F\in \H_{\Omega'}(C\bar\e^p)$.
\end{enumerate}
\end{cor}
\begin{notn}
Given a unimodal permutation $\perm$ of length $p>1$ let 
\begin{equation}
\H_{\Omega,\perm}=\bigcup_{f\in \U_{\Omegax,\perm}}\H_{\Omega,p}(f,\bar\e_f).
\end{equation}  
When restricting to $\bar\e$-thickenings we will use the notation $\H_{\Omega,\perm}(\bar\e)$.
\end{notn}
\begin{proof}
The first property follows from Proposition~\ref{prop:henon-perturb-renormalisation}. We now show the second property. Let $F\in \H_{\Omega,p}(0)$ be as above and let $\tilde F\H_{\Omega,\perm}$ denote the thickening of $F$ by $\e\in\T_\Omega(\bar\e)$. Let $H$ and $\tilde H$ denote their respective horizontal diffeomorphisms. Let 
\begin{equation}
G(x,y)=(g(x)\pm \delta(x,y),x), \quad \tilde G(x,y)=(\tilde g(x)\pm\tilde\delta(x,y),x)
\end{equation}
denote their respective pre-renormalisations. To simplify notation we only consider the orientation preserving case. The orientation reversing case is the same. Observe
\begin{equation}
\del_y\delta(x,y)=\jac{G}{(x,y)}=\jac{\o{p}{F}}{\bar H(x,y)}\frac{\jac{H}{\o{p}{F}(\bar H(x,y))}}{\jac{H}{\bar H(x,y)}}
\end{equation}
and
\begin{equation}
\del_y\tilde\delta(x,y)=\jac{\tilde G}{(x,y)}=\jac{\o{p}{\tilde F}}{\bar{\tilde H}(x,y)}\frac{\jac{\tilde H}{\o{p}{\tilde F}(\bar{\tilde H}(x,y))}}{\jac{\tilde H}{\bar{\tilde H}(x,y)}}.
\end{equation}
Now observe that $|\jac{\o{p}{F}}{\bar H(x,y)}|_\Omega=0$ and $|\jac{\o{p}{\tilde F}}{\bar \tilde H(x,y)}|_\Omega\leq |\e|_{\Omega}^p$. Next recall $\jac{H}{(x,y)}=\del_x\oo{p-1}{\phi}(x,y)$, so by the Variational Formula in Proposition~\ref{prop:var-formula} there is a constant $C_0>0$ such that, for $|\e|_\Omega$ sufficiently small,
\begin{equation}
\left|\frac{\jac{H}{\o{p}{F}(\bar H(x,y))}}{\jac{H}{\bar H(x,y)}}-\frac{\jac{\tilde H}{\o{p}{\tilde F}(\bar{\tilde H}(x,y))}}{\jac{\tilde H}{\bar{\tilde H}(x,y)}}\right|\leq C_0|\e|_{\Omega}.
\end{equation}
Since $f$ is renormalisable, $\o{p-1}{f}$ is a diffeomorphism on $\oo{1}{J}$. Therefore
\begin{equation}
\left|\frac{\jac{H}{\o{p}{F}(\bar{H}(x,y))}}{\jac{H}{\bar H(x,y)}}\right|
\leq \exp(\dis{H}{\oo{0}{B}_\diag})
\leq \exp(\dis{\o{p-1}{f}}{\oo{1}{J}})
\end{equation}
is bounded and we find there exists a constant $C_1>0$ such that, for $|\e|_{\Omega}$ sufficiently small,
\begin{equation}
\left|\frac{\jac{\tilde{H}}{\o{p}{\tilde{F}}(\bar{\tilde{H}}(x,y))}}{\jac{\tilde{H}}{\bar{\tilde H}(x,y)}}\right|<C_1.
\end{equation}
Hence $|\del_y\tilde\delta(x,y)|<C_1|\e|_{\Omega}^p$. By construction the renormalisation, $\tilde{F}_1$, of $\tilde F$ has parametrisation $(\tilde f_1,\tilde \e_1)$ which is an affine rescaling of $(\tilde g, \tilde \delta)$. There exists a constant $C_2>0$ such that the affine rescaling has scaling ratio $\sigma+C_2|\e|_\Omega$, where $\sigma$ is the scaling ratio for $F$. This implies there exists a constant $C_3>0$ such that $|\del_y\tilde\e_1|_{\Omega'}\leq C_3|\e|_\Omega^p$. Moreover, $\tilde\e_1$ satisfies $\tilde\e_1(x,0)=0$ by construction. Therefore $|\tilde\e_1|_{\Omega'}\leq C_3|\e|_\Omega^p$ and the result is shown.
\end{proof}

\begin{thm}\label{thm:R-construction}
Let $\perm$ be a unimodal permutation of length $p>1$.
Then there are constants $C,\bar\e_0>0$ and a domain $\Omega'\subset\CC$, depending upon $\upsilon$ and $\Omega$, such that the following holds: 
for any $0<\bar\e<\bar\e_0$ there is a subspace $\H_{\Omega,\upsilon}(\bar\e)\subset\H_{\Omega}(\bar\e)$ containing $\H_{\Omega,\perm}(0)$ and a dynamically defined continuous operator, 
\begin{equation}
\RH\colon \H_{\Omega,\upsilon}(\bar\e)\to \H_{\tilde\Omega}(C\bar\e^p),
\end{equation}
which extends continuously to $\RH$ on $\H_{\Omega,\perm}(0)$. Moreover $\bar\e_0>0$ can be chosen so that 
\begin{equation}
\RH\colon \H_{\Omega,\upsilon}(\bar\e)\to \H_{\tilde\Omega}(\bar\e).
\end{equation}
\end{thm}
\begin{proof}
By Corollary~\ref{cor:henon-perturbedunimodal}, for each $f\in\U_{\Omega,\perm}$ there exists a $\bar\e_f>0$ and a $C_f>0$ such that $\RH$ has an extension $\RH\colon\H_{\Omega}(f,\bar\e_f)\to \H_{\Omega}(C_f\bar\e_f^p)$. By compactness of $\U_{\Omega,\perm}$ these constants can be chosen uniformly, so setting
\begin{equation}
\H_{\Omega,\perm}(\bar\e)=\bigcup_{f\in\U_{\Omega,\perm}}\H_{\Omega}(f,\bar\e_f)
\end{equation}
we find that $\RH\colon \H_{\Omega,\perm}(\bar\e)\to\H_{\Omega}(C\bar\e^p)$. Choosing $\bar\e_0>0$ sufficiently small so that $\bar\e<C\bar\e^p$ for all $0<\bar\e<\bar\e_0$ gives the final claim.
\end{proof}

\subsection{The Fixed Point and Hyperbolicity}\label{sect:fixpoint}

We now consider H\'enon-like maps that are infinitely renormalisable. 
Throughout the rest of this section we fix a unimodal permutation $\perm$ of length $p$.
Denote by $\I_{\Omega,\perm}(\bar\e)$ the subspace of $\H_{\Omega}(\bar\e)$ consisting of infinitely renormalisable H\'enon-like maps, where each renormalisation has the same combinatorial type $\perm$. 
We call $\I_{\Omega,\perm}(\bar\e)$ the space of infinitely renormalisable H\'enon-like maps with \emph{stationary combinatorics} $\perm$.
Given any $F\in \I_{\Omega,\perm}(\bar\e)$ we write $F_n=\RH^n F$. We will use subscripts to denote quantities associated
with the $n$-th H\'enon-renormalisation. (For example, $\phi_n=\phi(F_n)$ will denote the function satisfying $F_n=(\phi_n,\pi_x)$.)

\begin{thm}\label{thm:R-convergence}
Let $\perm$ be a unimodal permutation of length $p>1$.
Let $\Omega=\Omegax\times\Omegay\cc\CC^2$ be a polydisk containing $B$.
Let $f_*\in\U_{\Omegax,\perm}$ denote the unimodal renormalisation fixed point of type $\perm$, 
and let $F_*=(f_*\circ\pi_x,\pi_x)$ denote the associated degenerate H\'enon-like map. 

Then $F_*$ is a hyperbolic fixed point of $\RH\colon \H_{\Omegax,\perm}\to\H_{\Omegax}$ with codimension-one stable manifold and dimension-one unstable manifold.
\end{thm}
\begin{proof}
As was noted in Remark~\ref{rmk:degenerate-henon-fixpoint}  we can canonically embed $\U_{\Omegax,\perm}$ into $\H_{\Omegax,\perm}(0)$. By definition, $\RH$ restricted to
$\H_{\Omegax,\perm}(0)$ is given by $\RH (f\circ\pi_x,\pi_x)=(\RU f\circ\pi_x,\pi_x)$. Therefore it is clear that the fixed point of $\RU$ induces a fixed point of $\RH$. That is, if $f_*$ denotes the fixed point of $\RU$ then the
point $F_*=(f_*\circ\pi_x,\pi_x)$ in $\H_{\Omegax,\perm}(0)$ is a fixed point of $\RH$. It is also clear that, when restricted to $\H_{\Omega}(0)$, the fixed point is unique and hyperbolic, with codimension one stable manifold and dimension-one unstable manifold.

Next observe that, by Theorem~\ref{thm:R-construction}, $\RH\colon \H_{\Omegax,\perm}\cap B_{\bar\e}(F_*)\to\H_{\Omegax}(C\bar\e^p)$ is super-exponentially contracting in the $\e$-direction if $\bar\e=\bar\e(\Omegax,\perm)>0$ is sufficiently small. This implies $D_{F_*}\RH$ is hyperbolic in the $\e$-direction, and moreover has zero spectrum.

Combining the conclusions of each of the two preceding paragraphs we find $F_*$ is hyperbolic, with a unique expanding direction. Therefore the Stable Manifold Theorem in~\cite{dMP} implies there exists a codimension-one stable manifold and dimension-one unstable manifold. (The Stable Manifold Theorem in~\cite{dMP} is stated for diffeomorphisms but the argument holds for endomorphisms as well.)
\end{proof}

\subsection{Scope Maps and The Renormalisation Cantor Set}\label{sect:scope}
We now recast the renormalisation theory we have just developed for H\'enon-like maps in terms of scope maps and presentation functions (defined below) in a way analogous to that in Section~\ref{sect:unimodal-scopemaps}. We show using these that, similar to the unimodal case, infinitely renormalisable H\'enon-like maps also possess an invariant Cantor set on which the H\'enon-like map acts as the adding machine.

Throughout this section, $\perm$ will be a fixed unimodal permutation of length $p>1$, 
$\bar\e_0>0$ will be a constant and $\Omega\subset \CC^2$ will be a complex polydisk containing the square $B$ in its interior such that $\I_{\Omega,\perm}(\bar\e)$ is invariant under renormalisation for all $0<\bar\e<\bar\e_0$.

If $F\in\H_{\Omega,\upsilon}(\bar\e)$ let $\{\oo{w}{B}\}_{w\in W}$ denote its cycle. 
Let $H\colon \oo{0}{B}\to \oo{0}{B_\diag}$ denote its horizontal diffeomorphism and 
$G\colon\oo{0}{B_\diag}\to \oo{0}{B_\diag}$ its pre-renormalisation. 
Let $\III\colon\oo{0}{B_\diag}\to B$ denote the affine rescaling such that $\RH F=\III G\bar\III$. 
Then we call the coordinate change $\MT=\MT(F)\colon B\to \oo{0}{B}$, given by $\MT=\bar{H}\circ\bar\III$, the \emph{scope map} of $F$. 
More generally, for $w\in W$ we will call the map $\oo{w}{\MT}=\o{w}{F}\circ \MT\colon B\to \oo{w}{B}$ the \emph{$w$-scope map} of $F$.

If $F$ is $n$-times renormalisable we denote the $n$-th renormalisation $\RH^n F$ by $F_n$. 
For $w\in W$ let $\oo{w}{\MT}_n=\oo{w}{\MT}(F_n)\colon \Dom(F_{n+1})\to \Dom(F_{n})$ be the $w$-scope map for $F_n$. 
Then, if $\word{w}{}=w_0\ldots w_n\in W^*$, the function
\begin{equation}
\MT^{\word{w}{}}=\MT_0^{w_0}\circ\ldots\circ \MT_n^{w_n}\colon \Dom(F_{n+1})\to \Dom(F_0)
\end{equation}
is called the \emph{$\word{w}{}$-scope map} for $F$. Let $\uline\MT=\{\MT^{\word{w}{}}\}_{\word{w}{}\in W^n}$ denote the collection of all scope functions for $F$.

Similarly if $F\in\I_{\Omega,\perm}(\bar\e)$, let $\uline{\MT}=\{\oo{\word{w}{}}{\MT}\}_{w\in W^*}$ denote the family of scope maps associated to $F$. 
Let $\uline{\MT_n}=\{\oo{\word{w}{}}{\MT_n}\}_{\word{w}{}\in W^*}$ denote the family of scope maps associated with $F_n$. 
For any $n\geq 0$, let $\oo{\word{w}{}}{B_n}=\oo{\word{w}{}}{\MT_n}(B)$. 
These are closed simply-connected domains which we call the \emph{pieces} for $F_n$. 
Finally let $\oo{\word{w}{}}{B_*}=\oo{\word{w}{}}{\MT_*}(B)$.
\begin{landscape}
\vspace*{2cm}
\begin{figure}[tp]
\centering

\psfrag{...}{$\cdots$}

\psfrag{a}{$0$}
\psfrag{amt}{$\oo{0}{\MT_0}$}
\psfrag{ab0}{$\oo{0}{B_0}$}
\psfrag{ab1}{$\oo{1}{B_0}$}
\psfrag{ab2}{$\oo{2}{B_0}$}
\psfrag{af0}{$F_0$}
\psfrag{af1}{$F_0$}
\psfrag{af2}{$F_0$}

\psfrag{b}{$1$}
\psfrag{bmt}{$\oo{0}{\MT_1}$}
\psfrag{bb0}{$\oo{0}{B_1}$}
\psfrag{bb1}{$\oo{1}{B_1}$}
\psfrag{bb2}{$\oo{2}{B_1}$}
\psfrag{bf0}{$F_1$}
\psfrag{bf1}{$F_1$}
\psfrag{bf2}{$F_1$}

\psfrag{c}{$2$}
\psfrag{cmt}{$\oo{0}{\MT_2}$}
\psfrag{cb0}{$\oo{0}{B_2}$}
\psfrag{cb1}{$\oo{1}{B_2}$}
\psfrag{cb2}{$\oo{2}{B_2}$}
\psfrag{cf0}{$F_2$}
\psfrag{cf1}{$F_2$}
\psfrag{cf2}{$F_2$}

\includegraphics[scale=0.21]{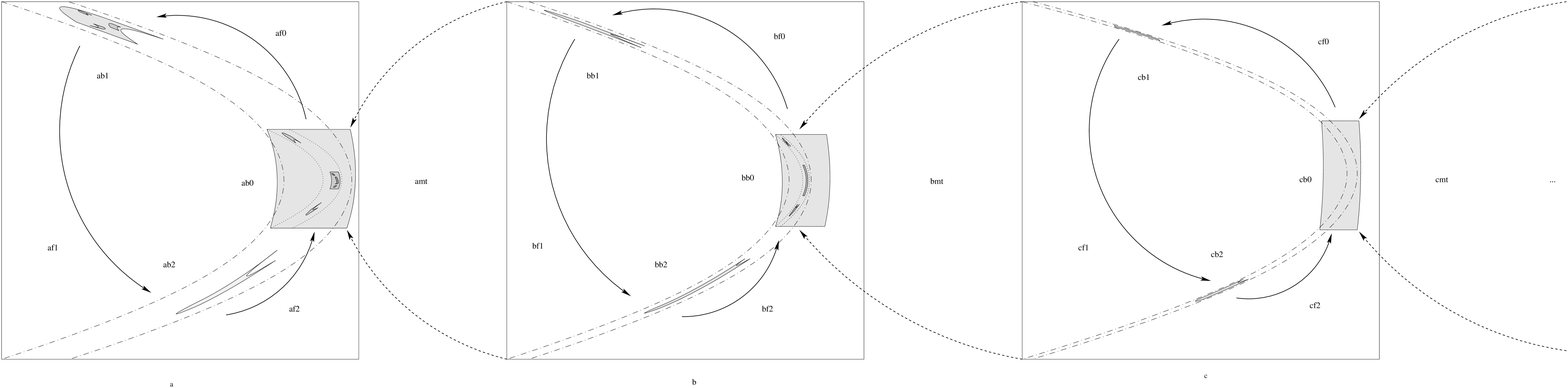}
\vspace*{0.2cm}
\caption{The sequence of scope maps for a period-three infinitely renormalisable H\'enon-like map. In this case the maps has stationary combinatorics of period-tripling type. Here the dashed line represents the bounding arcs of the image of the square $B$ under consecutive renormalisations $F_n$.}\label{fig:scopemaps-henon}
\end{figure}
\end{landscape}

\begin{prop}\label{prop:unimodal-cantorset-stable}
Let $f_n\in \U_{\Omegax,\perm}$ be a sequence of renormalisable unimodal maps and let $\pf_n=\{\oo{w}{\mt_n}\}_{w\in W}$ denote the presentation function of $f_n$. Assume
\begin{enumerate}
\item the central cycle $\{\oo{w}{J_n}\}_{w\in W}$ has uniformly bounded geometry for all $n>0$;
\item $\dis{\oo{w}{\mt_n}}{z}$ is uniformly bounded;
\item there exists an integer $N>0$ such that the Schwarzian derivative $\S_{\oo{w}{\mt}_n}>0$ for all integers $n>N$ and $w\in W$. 
\end{enumerate}
Then
\begin{equation}
\Cantor=\bigcap_{n\geq 0}\bigcup_{\word{w}{}\in W^n}\oo{\word{w}{}}{\mt}(J)
\end{equation}
is a Cantor set.
\end{prop}
\begin{proof}
Given closed intervals $J\subset T$, with $J$ properly contained in $T$, consider their cross-ratio,
\begin{equation}
D(J,T)=\frac{|J||T|}{|L||R|}
\end{equation}
where $L$ and $R$ are the left and right connected components of $T\setminus J$ respectively. We recall the following properties:
\begin{enumerate}
\item maps with positive Schwarzian derivative contract the cross-ratio;
\item for all $K>0$ there exists a $0<K'<1$ such that $D(J,T)<K$ implies $\frac{|J|}{|T|}<K'$.
\end{enumerate}
The first assumption implies there is some $K>0$ such that $D(\oo{w}{J}_n,J)<K$ for all $w\in W$, and $n\in\NN$. The third assumption implies the intervals $\oo{w_N\ldots w_n}{J}_N=\oo{w_N}{\mt_N}\circ\cdots\circ\oo{\mt_n}{w_n}(J)$ are images of $\oo{w}{J_n}$ under positive Schwarzian maps for all $n>N$ . Hence the first property of the cross-ratio implies  $D(\oo{w_N\ldots w_n}{J}_N,\oo{w_N\ldots w_{n-1}}{J}_N)<K$ for all $n>N$. By the second property of the cross ratio this implies $\frac{|\oo{w_N\ldots w_n}{J}_N|}{|\oo{w_N\ldots w_{n-1}}{J}_N|}<K'<1$. The same argument applies to the images of the gaps between the $\oo{w}{J}_n$. Therefore
\begin{equation}
\Cantor_N=\bigcap_{n\geq N}\bigcup_{\word{w}{}\in W^n}\oo{w_N\ldots w_n}{\mt}_N(J)
\end{equation}
is a Cantor set. By the second assumption $\oo{w_0\ldots w_{N-1}}{\mt}_0$ has bounded distortion for all $w_0\ldots w_{N-1}\in W^N$. The image of a Cantor set under a map with bounded distortion is still a Cantor set. Hence
\begin{equation}
\Cantor=\bigcap_{n\geq 0}\bigcup_{\word{w}{}\in W^n}\oo{\word{w}{}}{\mt}(J)
\end{equation}
is a Cantor set and the result is shown.
\end{proof}
We rephrase the above Proposition in terms of scope maps for degenerate H\'enon-like maps in the following.
\begin{cor}\label{cor:deghenon-cantorset-stable}
Let $F_n=\uline{\i}(f_n)\in \H_{\Omega,\perm}$ be a sequence of renormalisable degenerate H\'enon-like maps, let $\PF_n=\{\oo{w}{\MT_n}\}_{w\in W}$ denote the scope function of $F_n$ and let $\pf_n=\{\oo{w}{\mt_n}\}_{w\in W}$ denote the scope function for $f_n$. Assume
\begin{enumerate}
\item the central cycle $\{\oo{w}{B_n}\}_{w\in W}$ has uniformly bounded geometry;
\item $\dis{\oo{w}{\MT_n}}{z}$ is uniformly bounded;
\item there exists an integer $N>0$ such that $\S_{\oo{w}{\mt}_n}>0$ for all $n>N, w\in W$. 
\end{enumerate}
Then
\begin{equation}
\Cantor=\bigcap_{n\geq 0}\bigcup_{\word{w}{}\in W^n}\oo{\word{w}{}}{\MT}(B)
\end{equation}
is a Cantor set.
\end{cor}

The following is the main result of this section. It states that, under suitable conditions, a perturbation of a family of scope maps
whose limit set is a Cantor set will also have a limit set which is a Cantor set.
\begin{prop}\label{prop:henon-cantorset-stable}
Let $F_n\in \H_{\Omega,\perm}$ be a sequence of renormalisable H\'enon-like maps and let $\PF_n=\{\oo{w}{\MT_n}\}_{w\in W}$ denote the presentation function of $F_n$. Assume
\begin{enumerate}
\item the set $\Cantor=\bigcap_{n\geq 0}\bigcup_{\word{w}{}\in W^n}\oo{\word{w}{}}{\MT}(B)$ is a Cantor set;
\item for $\word{w}{}=w_0w_1\ldots\in W^*$ the cylinder sets $\oo{w_0,\ldots w_n}{\MT}(B)$ `nest down exponentially': there exists a constant $0<\delta<1$ such that $\diam(\oo{w_0,\ldots w_n}{\MT}(B))<\delta\diam(\oo{w_0,\ldots w_{n-1}}{\MT}(B))$ for all $n>0$;
\item $\|\D{\oo{w}{\MT}_n}{z}\|<K$ for all $z\in\Omega, w\in W$ and $n>0$.
\end{enumerate}
Then there exists an $\bar\e>0$ such that for any sequence $\tilde F_n\in\H_{\Omega,\perm}$ of renormalisable H\'enon-like maps with presentation functions $\PF_n=\{\oo{w}{\MT_n}\}_{w\in W}$ satisfying $|F_n-\tilde F_n|_{\Omega}<C\bar\e^{p^n}$ the set
\begin{equation}
\tilde\Cantor=\bigcap_{n\geq 0}\bigcup_{\word{w}{}\in W^n}\oo{\word{w}{}}{\tilde\MT}(B)
\end{equation}
is also a Cantor set.
\end{prop}
\begin{proof}
It is clear that $\tilde\Cantor$ is closed and non-empty, hence we are just required to show it is totally disconnected and contains no isolated points.
Let us first introduce some notation.
Define functions $E_n$ by $\tilde F_n=F_n+E_n$. By hypothesis $|E_n|_{\Omega}\leq C_0\bar\e^{p^n}$.
This implies we can write $\oo{w}{\tilde\MT_n}=\oo{w}{\MT_n}+\oo{w}{\Lambda_n}$ where $|\oo{w}{\Lambda_n}|_{\Omega}\leq C_1\bar\e^{p^n}$ for some constant $C_1>0$.
For $\word{w}{}=w_0\ldots w_n\in W^*$ let 
\begin{equation}
\oo{w_0\ldots w_n}{\MT}=\oo{w_0}{\MT}_0\circ\cdots\circ\oo{w_n}{\MT}_n, \qquad
\oo{w_0\ldots w_n}{\tilde\MT}=\oo{w_0}{\tilde\MT}_0\circ\cdots\circ\oo{w_n}{\tilde\MT}_n.
\end{equation}
Then define $\oo{w_0\ldots w_n}{\Lambda}$ to be the function satisfying $\oo{w_0\ldots w_n}{\tilde\MT}=\oo{w_0\ldots w_n}{\MT}+\oo{w_0\ldots w_n}{\Lambda}$. From Proposition~\ref{prop:variation-composition} we find for each $z\in B$, if we set $z_i=\oo{w_i\ldots w_n}{\MT}(z)$
and $\oo{\emptyset}{\tilde\MT}=\id$, that
\begin{equation}\label{eqn:lambda-variation}
\oo{w_0\ldots w_n}{\Lambda}(z)
=\sum_{i\geq 1}\D{\oo{w_0\ldots w_{i-1}}{\MT}}{z_i}(\oo{w_i}{\Lambda}(z_{i+1}))
+\bigo(|\D{\oo{w_i}{\Lambda}}{}||\oo{w_j}{\Lambda}|,|\oo{w_i}{\Lambda}|^2).
\end{equation}
Now fix $\word{w}{}=w_0w_1\ldots\in \bar W$ and let $z, z'\in B$ be any distinct fixed pair of points.
Observe that for any $0<m<n$,
\begin{equation}
|\oo{w_0\ldots w_n}{\tilde\MT}(z)-\oo{w_0\ldots w_n}{\tilde\MT}(z')|\leq \sup_{\xi\in B}\|\D{\oo{w_0\ldots w_{m-1}}{\tilde\MT}}{\xi}\| |\oo{w_m\ldots w_n}{\tilde\MT}(z)-\oo{w_m\ldots w_n}{\tilde\MT}(z')|.
\end{equation}
Since the derivatives of $\oo{w}{\MT}_n$ are uniformly bounded and $|\tilde F_n-F_n|_\Omega$ decreases super-exponentially there is a constant $\bar\e>0$ such that $\|D_\xi \oo{w}{\MT}_n\|< 2K$ for all $\xi\in B$. Hence
\begin{equation}
\|D_\xi\oo{w_0\ldots w_m}{\tilde \MT}\|\leq \prod_0^{m-1}\|D_{\xi_i}\oo{w_i}{\tilde \MT}_i\|\leq (2K)^m.
\end{equation}  
Now we consider the splitting
\begin{align}
&|\oo{w_m\ldots w_n}{\tilde\MT}(z)-\oo{w_m\ldots w_n}{\tilde\MT}(z')| \notag \\
&\leq  |\oo{w_m\ldots w_n}{\MT}(z)-\oo{w_m\ldots w_n}{\MT}(z')|+|\oo{w_m\ldots w_n}{\Lambda}(z)-\oo{w_m\ldots w_n}{\Lambda}(z')|. 
\end{align}
By hypothesis $|\oo{w_m\ldots w_n}{\MT}(z)-\oo{w_m\ldots w_n}{\MT}(z')|\leq \delta^{n-m}$ while Proposition~\ref{prop:variation-composition} in the appendix implies there exists a constant $C_3>0$ such that
\begin{equation}
|\oo{w_m\ldots w_n}{\Lambda}(z)-\oo{w_m\ldots w_n}{\Lambda}(z')|\leq C(2K)^{n-m}\bar\e^{p^m}.
\end{equation} 
Therefore
\begin{equation}
|\oo{w_0\ldots w_n}{\tilde\MT}(z)-\oo{w_0\ldots w_n}{\tilde\MT}(z')|\leq (2K)^m\left(\delta^{n-m}+C(2K)^{n-m}\bar\e^{p^m}\right).
\end{equation}
By Proposition~\ref{prop:exp-vs-superexp} this can be made arbitrarily small by taking $m,n/m>0$ sufficiently large.

Next we show that $\Cantor$ does not have any isolated points.
Assume there is a word $\word{w}{}=w_0w_1\ldots\in \bar W$ for which the associated cylinder set $\oo{\word{w}{}}{B}$ is isolated. Then for any other word $\word{\tilde w}{}\in\bar W$ we must have $\dist(\oo{\word{w}{}}{B},\oo{\word{\tilde w}{}}{B})>\rho$ for some $\rho>0$ which we may assume satisfies $\rho<1$. We know that for any $0<\rho<1$ there is an integer $N>0$ such that for all $n>N$, $\diam(\oo{w_0\ldots w_n}{B})<\rho$. In particular $\dist(\oo{w_0\ldots w_nw_{n+1}}{B},\oo{w_0\ldots w_n\tilde{w}}{B})<\rho$ for any $\tilde w\in W$, which is a contradiction. Hence $\Cantor$ does not have any isolated points.
\end{proof}
Using these last two results we can now prove the following.
\begin{prop}\label{prop:henon-cantorset-existence}
Let $\perm$ be a unimodal permutation of length $p>1$. There exists a constant $\bar\e_0>0$, depending upon $\perm$, for which the following holds:
given any $F\in\I_{\Omega,\perm}(\bar\e_0)$ let $\uline\MT=\{\oo{\word{w}{}}{\MT}\}_{\word{w}{}\in W^*}$ denote its family of scope maps.
Then the set
\begin{equation}
\Cantor=\bigcap_{n\geq 0}\bigcup_{\word{w}{}\in W_p^n}\MT^{\word{w}{}}(B),
\end{equation}
has the following properties:
\begin{enumerate}
\item it is an $F$-invariant Cantor set;
\item $F$ acts as the adding machine upon $\Cantor$, i.e. there exists a map $h\colon \oline{W}_p\to \Cantor$ such that the following diagram commutes:
\begin{equation}
\xymatrix{
\oline{W}_p\ar[d]_{h}  \ar[r]^{\word{w}{}\mapsto 1+\word{w}{}}                & \oline{W}_p\ar[d]^{h} & \\
\Cantor         \ar[r]_F & \Cantor
}
\end{equation}
\item there is a unique $F$-invariant measure, $\mu$, with support on $\Cantor$.
\end{enumerate}
\end{prop}
The set $\Cantor$ will be called the \emph{renormalisation Cantor set} for $F$, or simply the \emph{Cantor set} for $F$ when there is no risk of confusion. 
\begin{proof}
The only thing that needs to be shown is that the limit set is actually a Cantor set. The rest follow by standard arguments.

Let $F\in\I_{\Omega,\perm}$ and let $F_n$ denote it's $n$-th renormalisation. Then for $n>0$ sufficiently large the unimodal part $f_n$ of $F_n$ will be renormalisable and they will converge exponentially to $f_*$. Therefore the corresponding degenerate maps $\i(f_n)$ will satisfy the conditions of Corollary~\ref{cor:deghenon-cantorset-stable} and hence the limit set of their scope maps will be a Cantor set that nests down exponentially. Applying the Proposition above shows the limit set of the scope maps for the $F_n$ is also a Cantor set.
\end{proof}
\begin{rmk}
Let us denote the cylinder sets of $\Cantor$ under the action of $F$ by $\Cantor^{\word{w}{}}$. That is, $\Cantor^{\word{w}{}}=\Cantor\cap \MT^{\word{w}{}}(B)$. Then the collection $\uline\Cantor=\{\Cantor^{\word{w}{}}\}_{\word{w}{}\in W^*}$ has the
following structure
\begin{enumerate}
\item $F(\Cantor^{\word{w}{}})=\Cantor^{1+\word{w}{}}$ for all $\word{w}{}\in W^*$;
\item $\Cantor^{\word{w}{}}$ and $\Cantor^{\word{\tilde w}{}}$ are disjoint for all $\word{w}{}\neq \word{\tilde w}{}$ of the same length;
\item the disjoint union of the $\Cantor^{\word{w}{w}}$ is equal to $\Cantor^{\word{w}{}}$, for all $\word{w}{}\in W^*, w\in W$;
\item $\Cantor =\bigcup_{\word{w}{}\in W^n}\Cantor^{\word{w}{}}$ for each $n\geq 1$.
\end{enumerate} 
This will play an important role in studying the geometry of the Cantor set $\Cantor$.
\end{rmk}

\begin{rmk}
For any integer $n>0$ we can construct the functions $\MT^{\word{w}{}}_n=\MT^{\word{w}{}}(F_n)$ and the sets $\Cantor_n^{\word{w}{}}=\Cantor^{\word{w}{}}(F_n)$ in the same way
as above. Let $\uline\MT_n=\{\MT^{\word{w}{}}_n\}_{\word{w}{}\in W^*}$ and $\uline\Cantor_n=\{\Cantor_n^{\word{w}{}}\}_{\word{w}{}\in W^*}$.

The number $n$ is called the \emph{height} of $\MT^{\word{w}{}}_n$ and $\Cantor_n^{\word{w}{}}$ and the length of $\word{w}{}$ is called the \emph{depth}.
We use the terms height and depth to reflect a kind of duality in our construction.
We will also use these adjectives for all associated objects.
\end{rmk}
\begin{cor}\label{cor:cantor-converge}
Let $\perm$ be a unimodal permutation of length $p>1$. There exist constants $C>0$ and $0<\rho<1$ such that the following holds:
Let $F\in\I_{\Omega,\perm}(\bar\e)$ and let $\word{w}{}\in\bar W$ be an arbitrary infinite word. Then the points $\oo{\word{w}{}}{\Cantor_n}$ and $\oo{\word{w}{}}{\Cantor_*}$
satisfy
\begin{equation}
|\oo{\word{w}{}}{\Cantor_n}-\oo{\word{w}{}}{\Cantor_*}|<C\rho^n.
\end{equation}
\end{cor}

\begin{defn}
Let $F\in \I_{\Omega\perm}(\bar\e)$ be an infinitely renormalisable H\'enon-like map with renormalisation Cantor set $\Cantor$ with $F$-invariant measure $\mu$. Then the \emph{Average Jacobian} of $F$ is defined by
\begin{equation}
b_F=\exp \int_\Cantor \log\left|\jac{F}{z}\right| d\mu(z).
\end{equation}
\end{defn}

The remainder of this work can be considered as a study of this quantity. The following result was given in~\cite{dCML} for period doubling, but the proof is valid for any $\perm$. We state it here without proof.
\begin{lem}[Distortion Lemma]\label{lem:distortion}
Let $\perm$ be a unimodal permutation of length $p>1$. Then there exist constants $C>0$, and $0\leq \rho<1$ such that the following holds:
Let $F\in\I_{\Omega,\perm}(\bar\e)$ and let $\oo{\word{w}{}}{B}$ denote the piece associated to the word $\word{w}{}\in W^*$. Then for any $\oo{\word{w}{}}{B}$, where $\word{w}{}\in W^n$, and any $z_0,z_1\in \oo{\word{w}{}}{B}$,
\begin{equation}
\log \left| \frac{\jac{\o{m}{F}}{z_0}}{\jac{\o{m}{F}}{z_1}}\right|\leq C\rho^n
\end{equation}
for all $m=1,p,\ldots,p^n$.
\end{lem}
\begin{cor}
Let $\perm$ be a unimodal permutation of length $p>1$. Then there exists a universal constant $0<\rho<1$ such that the following holds:
given $0<\bar\e<\bar\e_0$, let $F\in\I_{\Omega,\perm}(\bar\e)$. Then for any integer $n\geq 0$, any $\word{w}{}\in W^n$, and any $z\in \oo{\word{w}{}}{B}$,
\begin{equation}
\jac{\o{p^n}{F}}{z}=b_F^{p^n}(1+\bigo(\rho^n)).
\end{equation}
\end{cor}
\begin{rmk}
The constant $\rho$ may be taken as the universal constant from Theorem~\ref{thm:R-convergence}.
\end{rmk}
\begin{proof}
Observe that, as $\mu$ has support on $\Cantor$,
\begin{align}
\int_{\oo{\word{w}{}}{B}}\log|\jac{\o{p^n}{F}}{z}|d\mu(z)
&=\int_{\oo{\word{w}{}}{\Cantor}}\log|\jac{\o{p^n}{F}}{z}|d\mu(z) \notag \\
&=\int_\Cantor\log|\jac{F}{z}|d\mu(z) \notag \\
&=\log b_F.
\end{align}
Therefore, there is a $\xi\in \oo{\word{w}{}}{B}$ such that
\begin{equation}
\log|\jac{\o{p^n}{F}}{\xi}|=\frac{\log b_F}{\mu(\oo{\word{w}{}}{B})}=p^n\log b_F
\end{equation}
so the result follows from the Lemma~\ref{lem:distortion}. 
\end{proof}
\begin{prop}[Monotonicity]\label{monotonicity}
Let $\perm$ be a unimodal permutation of length $p>1$. Let $F_t\in\I_{\Omega,\perm}(\bar\e_0)$ be a one parameter family of infinitely renormalisable H\'enon-like maps such that the average Jacobian $b_t=b(F_t)$ depends strictly
monotonically on $t$. Let $\tilde F_t=\RH F_t$ and let $\tilde b_t=b(\tilde F_t)$. Then $\tilde b_t$ is also strictly monotone in $t$.
\end{prop}
\begin{proof}
Let $\tilde F_t=\RH F_t$, $\tilde\Cantor_t=\Cantor(\tilde F_t)$, and $\tilde\mu_t=\mu(\tilde F_t)$. Recall that, by definition,
\begin{equation}
\log b_t=\int_{\Cantor_t}\log|\jac{F_t}{z}|d\mu_t (z), \quad \log\tilde b_t=\int_{\tilde\Cantor_t}\log|\jac{\tilde F_t}{z}|d\tilde\mu_t(z).
\end{equation}
Then by construction $\tilde F_t=\MT_t^{-1}\o{p}{F_t}\MT_t$ and $\tilde\Cantor_t=\bar\MT_t(\Cantor_t^{0})$, where $\Cantor_t^0=\Cantor_t\cap B_t^0$. Since $\mu_t$ and $\tilde\mu_t$ are determined by the adding machine actions on $\Cantor_t$ and $\tilde\Cantor_t$ respectively we also have $\tilde\mu_t=p\mu_t\circ \MT_t$.
Therefore
\begin{align}
& \int_{\tilde\Cantor_t}\log|\jac{\tilde F_t}{z}|d\tilde{\mu}_t(z) \notag \\
&=p\int_{\bar\MT_t(\Cantor_t^0)}\log\left(\left|\jac{\o{p}{F_t}}{\MT_t(z)}\right|\left|\frac{\jac{\MT_t}{z}}{\jac{\MT_t}{\tilde F_t(z)}}\right|\right)d(\mu_t\circ\MT_t)(z)
\end{align}
hence\footnote{here we use the integral substitution fomula, namely if $(X,\mathcal B),(X',\mathcal B')$ are measurable spaces, $\mu$ is a measure on $X$, $T\colon X\to Y$ surjective then for all $\mu\circ T^{-1}$-measurable $\phi$ on $Y$,
\[\int_X\phi\circ Td\mu=\int_Y\phi d (\mu\circ T^{-1})\] }
\begin{align}
&\int_{\tilde\Cantor_t}\log|\jac{\tilde F_t}{z}|d\tilde{\mu}_t(z) \\
&=p\int_{\Cantor_t^0}\log\left(\left|\jac{\o{p}{F_t}}{z}\right|\left|\frac{\jac{\MT_t}{\bar\MT_t(z)}}{\jac{\MT_t}{\bar\MT_t\o{p}{F_t}(z)}}\right|\right)d\mu_t(z) \notag \\
&=p\int_{\Cantor_t^0}\sum_{i=0}^{p-1}\log|\jac{F_t}{\o{i}{F_t}(z)}|d\mu_t(z)+p\int_{\Cantor_t^0}\log\left(\frac{\jac{\MT_t}{\bar\MT_t(z)}}{\jac{\MT_t}{\bar\MT\o{p}{F_t}(z)}}\right)d\mu_t(z). \notag
\end{align}
Now observe, by definition of $\mu_t$,
\begin{equation}
\int_{\Cantor_t^0}\sum_{i=0}^{p-1}\log|\jac{F_t}{\o{i}{F_t}(z)}|d\mu_t(z)
=\int_{\Cantor_t}\log|\jac{F_t}{z}|d\mu_t(z)
\end{equation}
and
\begin{equation}
\int_{\Cantor_t^0}\log|\jac{\MT_t}{\bar\MT\o{p}{F_t}(z)}|d\mu_t(z)
=\int_{\Cantor_t^0}\log|\jac{\MT_t}{\bar\MT(z)}|d\mu_t(z).
\end{equation}
Together these imply
\begin{equation}
\log \tilde b_t=\int_{\tilde\Cantor_t}\log|\jac{\tilde F_t}{z}|d\tilde{\mu}_t(z) =p\int_{\Cantor_t}\log|\jac{F_t}{z}|d\mu_t =p\log b_t
\end{equation}
which depends monotonically on $t$ if $\log b_t$ depends monotonically. Since the logarithm function is monotone the proof is complete. 
\end{proof}

\section{Asymptotics of Scope Functions}\label{sect:asymptotics}
We study affine rescaling of scope functions and their compositions. We only consider the case when $w_i=0$ for all $i>0$ as this is the simplest to deal with and the most relevant in the next sections. However, we believe a large portion of the results below can be extended to the more general case. As before, unless otherwise stated, 
throughout this section $\perm$ will be a fixed unimodal permutation of length $p>1$ and $\bar\e_0>0$ will be a constant and $\Omega\subset \CC^2$ will be a complex polydisk containing the square $B$ in its interior such that $\I_{\Omega,\perm}(\bar\e)$ is invariant under renormalisation for all $0<\bar\e<\bar\e_0$.
 
\begin{prop}
Let $\perm$ be a unimodal permutation of length $p>1$. Then there exists a constant $\bar\e_0>0$ such that the following holds:
given $0\leq \bar\e<\bar\e_0$, for any $F\in\I_{\Omega,\perm}(\bar\e)$ let $\MT_n\colon B\to B$ denote its $n$-th scope map. Explicitly, for any $(x,y)\in B$, let
\begin{equation}
F(x,y)=(\phi_n(x,y),x); \quad \MT_n(x,y)=(\oo{1}{\mt}_n,\oo{0}{\mt}_n).
\end{equation}
Then there is a constant $C>0$, depending upon $F$ only, such that
\begin{equation}
|\del_{x^i}\oo{1}{\mt}_n(x,y)|<C, \quad |\del_{x^iy^j}\oo{1}{\mt}_n(x,y)|<C\bar\e^{p^n}
\end{equation}
for any $(x,y)\in B$ and any integers $i,j\geq 1$.
\end{prop}
\begin{proof}
By Theorem~\ref{thm:R-convergence} we know there exists a constant $C_0>0$ and, for each integer $n>0$, a degenerate $\tilde F_n\in\H_{\Omega,\perm}(0)$ such that $|F_n-\tilde
F_n|_{\Omega}\leq C_0\bar\e^{p^n}$ and $F_n$
converges exponentially to $F_*$. Let $\tilde\MT_{n}$ denote the scope function for $F_n$. Then this implies there exists a constant $C_1>0$ such that $|\MT_{n}-\tilde
\MT_{n}|_{\Omega}\leq C_1\bar\e^{p^n}$ and $\tilde\MT_n$ converges exponentially to $\MT_*$. Since $\MT_*$ is analytic there exists a constant $C_2>0$ such that
$|\del_{x^i}\oo{1}{\mt}_*|<C_2$ and as $F_*$ is degenerate $\del_{x^iy^j}\oo{1}{\mt}_*=0$ for $j>0$. Hence the result follows.
\end{proof}
The next Lemma is a simple application of Taylor's Theorem.
\begin{lem}\label{lem:scope-decomposition}
For any $F\in \I_{\Omega,\perm}(\bar\e)$ let $\MT\colon B\to \oo{0}{B}$ denote its $n$-th scope map. Explicitly, for $(x,y)\in B$ let 
\begin{equation}
F(x,y)=(\phi(x,y),x); \quad \MT(x,y)=(\oo{1}{\mt}(x,y),\oo{0}{\mt}(x,y)).
\end{equation}
Then, for $z_0\in B$ and $z_1\in\RR$ satisfying $z_0+z_1\in B$, $\MT$ can be expressed as
\begin{equation}
\MT(z_0+z_1)=\MT(z_0)+\D{\MT}{z_0}\circ (\id+\Rem{\MT}{z_0})(z_1)
\end{equation}
where $\D{\MT}{z_0}$ denotes the derivative of $\MT$ at $z_0$ and $\Rem{\MT}{z_0}$ is a nonlinear remainder term. The maps $\D{\MT}{z_0}$ and $\Rem{\MT}{z_0}$ take the form
\begin{equation}
\D{\MT}{z_0}=\sigma\iibyii{s(z_0)}{t(z_0)}{0}{1}; \quad \Rem{\MT}{z_0}(z_1)=\ibyii{r(z_0)(z_1)}{0}
\end{equation}
for some functions $s(z)$ and $t(z)$. Here $\sigma$ denotes the scaling ratio of $\MT$.
\end{lem}
\begin{rmk}\label{rmk:sigma}
There are two related quantities that will henceforth play an important role. The first is the scaling ratio of $F_*$, defined to be the unique eigenvalue, of multiplicity two,
of the affine factor of $\MT_*=\oo{0}{\MT}_*$. The second is the derivative of $\oo{0}{\MT}_*$ at the tip $\tau_*$ of $F_*$. By Lemma~\ref{lem:univ-linearisation} the
derivative of $\oo{1}{\mt}_*$ at its fixed point is also this scaling ratio (up to sign, which depends on the combinatorics), but the fixed point of $\oo{1}{\mt}$ is the
critical value, which is the projection onto the $x$ axis of $\tau_*$.
Hence these two quantities coincide and shall be denoted by $\sigma$.
\end{rmk}
\begin{defn}
The functions $s(z)$ and $t(z)$ given by the Lemma~\ref{lem:scope-decomposition} above are called the \emph{squeeze} and the \emph{tilt} of $\MT$ at $z$ respectively.
\end{defn}
\begin{prop}\label{prop:r-estimate}
Let $F\in\H_{\Omega,\perm}(\bar\e)$ and $\tilde F\in\H_{\Omega,\perm}(0)$ satisfy $|F-\tilde F|_{\Omega}<\bar\e$.
Let $\MT=(\oo{1}{\mt},\oo{0}{\mt})$ and $\tilde\MT_f=(\oo{1}{\tilde\mt},\oo{0}{\tilde\mt})$ denote their respective scope maps.
Assume there is a constant $C>1$ such that, for all $i>0$,
\begin{equation}
C^{-1}<\left|\del_x\oo{1}{\mt}\right|_{\Omega} \ \ \mbox{and} \ \ \left|\frac{\del_{x^i}\oo{1}{\mt}}{\del_x\oo{1}{\mt}}\right|_{\Omega}<C
\end{equation}
Then there is a constant $K>0$ such that, if $R(z_0)(z_1)=\Rem{\MT}{z_0}(z_1)$ is defined as above,
\begin{equation}
|\del_x r(z_0)(z_1)|, |\del_{xx} r(z_0)(z_1)|<K(1+|F-F_*|+\bar\e)
\end{equation}
and
\begin{equation}
|\del_y r(z_0)(z_1)|, |\del_{xy} r(z_0)(z_1)|, |\del_{yy} r(z_0)(z_1)|<K\bar\e.
\end{equation}
for any $z_0\in B$ and $z_1\in\RR^2$ satisfying $z_0+z_1\in B$.  
\end{prop}
\begin{proof}
Let $z_i=(x_i,y_i)$ for $i=0,1$.
Expanding $\MT$ in power series around $z_0$ and equating it with the above representation gives
\begin{align}
r(z_0)(z_1)
&=\sum_{i,j\geq 0;i+j\geq 2} \binom{i+j}{j}x_1^iy_1^j\dfrac{\del_{x^iy^j}\oo{1}{\mt}(z_0)}{\del_{x}\oo{1}{\mt}(z_0)}.
\end{align}
and we get a similar expression for $\tilde r(z_0)(z_1)$ and $r_*(z_0)(z_1)$. We may write $r(z_0)(z_1)=A_0(x_1)+y_1A_1(x_1)+y_1^2A_2(x_1,y_1)$ where
\begin{align}
A_0(x_1)&=\sum_{i\geq 2;j=0}x_1^{i}\dfrac{\del_{x^i}\oo{1}{\mt}(z_0)}{\del_{x}\oo{1}{\mt}(z_0)} \\
A_1(x_1)&=\sum_{i\geq 1;j=1}\binom{i+1}{1}x_1^i\dfrac{\del_{x^iy}\oo{1}{\mt}(z_0)}{\del_{x}\oo{1}{\mt}(z_0)} \\
A_2(x_1,y_1)&=\sum_{i\geq 0;j\geq 2}\binom{i+j}{j}x_1^iy_1^j\dfrac{\del_{x^iy^j}\oo{1}{\mt}(z_0)}{\del_{x}\oo{1}{\mt}(z_0)}.
\end{align}
Define $\tilde A_0(x_1),\tilde A_1(x_1)$ and $\tilde A_2(x_1,y_1)$ for $\tilde r(z_0)(z_1)$ and $A_{*,0}(x_1), A_{*,1}(x_1)$ and $A_{*,2}(x_1,y_1)$ for $r_*(z_0)(z_1)$ similarly.
We claim there exists a constant $K_0>0$ such that
\begin{equation}
|A_0'(x_1)|, |A_0''(x_1)|\leq K_0\frac{|x_1|^2}{1-|x_1|}(1+|\tilde F-F_*|_{\Omega}+|\tilde F-F|_{\Omega}).
\end{equation}
First, as $F_*$ is fixed we may assume, without loss of generality, that the constant $C>0$ satisfies $\left|\frac{\del_{x^i}\oo{1}{\mt}_*}{\del_{x}\oo{1}{\mt}_*}\right|<C$. Also
observe that $C^{-1}< |\del_x\oo{1}{\mt}|$ implies $|\del_x\oo{1}{\tilde\mt}|^{-1}<C(1+\kappa\bar\e)$ for some $\kappa>0$.
Therefore, by Lemma~\ref{lem:ratio-perturb}, 
\begin{equation}
\left|\frac{\del_{x^i}\oo{1}{\tilde\mt}}{\del_{x}\oo{1}{\tilde\mt}}-\frac{\del_{x^i}\oo{1}{\mt_*}}{\del_{x}\oo{1}{\mt_*}}\right|_{\Omega}
\leq C(1+\kappa\bar\e)\max\left(1,C\right)\max_i\left(|\del_{x^i}\oo{1}{\tilde\mt}-\del_{x^i}\oo{1}{\mt}_*|_\Omega\right).
\end{equation}
The same argument, this time using the assumption $\left|\frac{\del_{x^i}\oo{1}{\mt}}{\del_{x}\oo{1}{\mt}}\right|<C$, also implies
\begin{equation}
\left|\frac{\del_{x^i}\oo{1}{\tilde\mt}}{\del_{x}\oo{1}{\tilde\mt}}-\frac{\del_{x^i}\oo{1}{\mt}}{\del_{x}\oo{1}{\mt}}\right|_{\Omega}
\leq C(1+\kappa\bar\e)\max\left(1,C\right)\max_i\left(|\del_{x^i}\oo{1}{\tilde\mt}-\del_{x^i}\oo{1}{\mt}|_\Omega\right),
\end{equation}
so analyticity of $F,\tilde F$ and $F_*$ implies there is a constant $C_0>0$ such that
\begin{equation}
\left|\frac{\del_{x^i}\oo{1}{\tilde\mt}}{\del_{x}\oo{1}{\tilde\mt}}-\frac{\del_{x^i}\oo{1}{\mt_*}}{\del_{x}\oo{1}{\mt_*}}\right|_{\Omega}
\leq C_0|\tilde F-F_*|_{\Omega}
\end{equation}
and
\begin{equation}
\left|\frac{\del_{x^i}\oo{1}{\tilde\mt}}{\del_{x}\oo{1}{\tilde\mt}}-\frac{\del_{x^i}\oo{1}{\mt}}{\del_{x}\oo{1}{\mt}}\right|_{\Omega}
\leq C_0|\tilde F-F|_{\Omega}.
\end{equation}
Hence, by the summation formula for a geometric progression,
\begin{align}
|\tilde A_0(x_1)-A_{*,0}(x_1)|
&\leq\sum_{i\geq 2}|x_1|^i\left|\frac{\del_{x^i}\oo{1}{\tilde\mt}(x_0)}{\del_{x}\oo{1}{\tilde\mt}(x_0)}-\frac{\del_{x^i}\oo{1}{\mt_*}(x_0)}{\del_{x}\oo{1}{\mt_*}(x_0)}\right| 
\leq \frac{C_0|x_1|^2}{1-|x_1|}|\tilde F-F_*|_{\Omega}
\end{align}
and similarly
\begin{align}
|\tilde A_0(x_1)-A_{0}(x_1)|
&\leq\sum_{i\geq 2}|x_1|^i\left|\frac{\del_{x^i}\oo{1}{\tilde\mt}(x_0)}{\del_{x}\oo{1}{\tilde\mt}(x_0)}-\frac{\del_{x^i}\oo{1}{\mt_*}(x_0)}{\del_{x}\oo{1}{\mt_*}(x_0)}\right|
\leq \frac{C_0|x_1|^2}{1-|x_1|}|\tilde F-F|_{\Omega}.
\end{align}
Secondly, observe that analyticity and degeneracy of $F_*$ implies  there exists a constant $C_1>0$ such that $|A_{*,0}(x_1)|<\frac{C_1|x_1|^2}{1-|x_1|}$. Therefore there exists
a $K_0>0$ such that
\begin{align}
|A_0(x_1)|
&\leq |A_{*,0}(x_1)|+|\tilde A_0(x_1)-A_{*,0}(x_1)|+|A_0(x_1)-\tilde A_0(x_1)| \\
&\leq \frac{K_0|x_1|^2}{1-|x_1|}\left(1+|\tilde F- F_*|_{\Omega}+|\tilde F-F|_{\Omega}\right) \notag
\end{align}
and, by analyticity of $A_0$, this implies the bound on its derivatives.
Next we claim there is are constants $C_2,C_3>0$ such that
\begin{equation}
 |A_1(z_1)|, |A_1'(z_1)|, |A_1''(z_1)| \leq C_2\bar\e|z_1|,
 \end{equation}
and
\begin{equation}
|\del_xA_2(z_1)|, |\del_yA_2(z_1)|, |\del_{xx}A_2(z_1)|, |\del_{xy}A_2(z_1)|, |\del_{yy}A_2(z_1)|\leq C_3\bar\e |z_1|.
\end{equation}
This can be seen by observing that all the coefficients of $A_1(z_1)$ and $A_2(z_1)$ are of the form $\del_{x^iy^j}\oo{p-1}{\mt}(z_0)/\del_x\oo{p-1}{\mt}(z_0)$, but from the Variational Formula there exists a constant $C_4>0$ such that
\begin{equation}
C_4^{-1}<|\del_{x}\oo{1}{\mt}|<C_4; \quad |\del_{x^iy^j}\oo{1}{\mt}|<C_4\bar\e,
\end{equation}
hence all coefficients are bounded by $C_4^2\bar\e$ in absolute value. Therefore, assuming $|z_1|\leq \gamma <1$, the above estimates must hold by setting $C_3=C_4^2/(1-\gamma)$.

Now differentiating $r(\MT;z_0)$ and applying the above estimates we find there exists a $C>0$ such that, for $|z_1|\leq \gamma<1$,
\begin{align}
|\del_x r(\MT;z_0)(z_1)|
&\leq |A_0'(x_1)|+|y_1||A_1'(x_1)|+|y_1|^2|\del_xA_2(x_1,y_1)| \notag \\
&\leq C(1+|f-f_*|+\bar\e) \\
|\del_y r(\MT;z_0)(z_1)|
&\leq |A_1(x_1)|+2|y_1||A_2(x_1,y_1)|+|y_1|^2|\del_y A_2(x_1,y_1)| \notag \\
&\leq C\bar\e  \\
|\del_{xx} r(\MT;z_0)(z_1)|
&\leq |A_0''(x_1)|+|y_1||A_1''(x_1)|+|y_1|^2|\del_{xx}A_2(x_1,y_1)| \notag \\
&\leq C(1+|f-f_*|+\bar\e) \\
|\del_{xy} r(\MT;z_0)(z_1)|
&\leq |A_1'(x_1)|+2|y_1||\del_xA_2(x_1,y_1)|+|y_1|^2|\del_{xy}A_2(x_1,y_1)| \notag \\
&\leq C\bar\e \\
|\del_{yy} r(\MT;z_0)(z_1)|
&\leq 2|A_2(x_1,y_2)|+4|y_1||\del_yA_2(x_1,y_1)|+|y_1|^2|\del_{yy}A_2(x_1,y_1)| \notag \\
&\leq C\bar\e
\end{align}
and hence the result is proved. 
\end{proof}

\section{Asymptotics around the Tip}\label{sect:tip}
As before, unless otherwise stated, 
throughout this section $\perm$ will be a fixed unimodal permutation of length $p>1$ and $\bar\e_0>0$ will be a constant and $\Omega\subset \CC^2$ will be a complex polydisk containing the square $B$ in its interior such that $\I_{\Omega,\perm}(\bar\e)$ is invariant under renormalisation for all $0<\bar\e<\bar\e_0$.

For a given $F\in \I_{\Omega,\perm}(\bar\e)$ we now wish to study the Cantor set $\Cantor$, and the behaviour of $F$ around it, in more detail. We will do this locally around a pre-assigned point. Let
\begin{equation}
\tau=\tau(F)=\bigcap_{n\geq 0} \oo{0^n}{B}.
\end{equation}
We call this point the \emph{tip}. The study of the orbit of this point is analogous to studying the critical orbit for a unimodal map. The remainder of our work can be viewed as the study of the behaviour of $F$ around $\tau$.

For $F\in\I_{\Omega,\perm}(\bar\e)$, as usual, let $F_n$ denote the $n$-th renormalisation and let $\MT_n\colon B\to \oo{0}{B_n}$ denote the scope map for $F_n$. 
Explicitly, $F_n(x,y)=(\phi_n(x,y),x)$ and $\MT_n(x,y)=(\oo{1}{\mt}_n(x,y),\oo{0}{\mt}_n(x,y))$.
Now let $\MT_{m,n}=\MT_m\circ\ldots\circ\MT_n$. Then $\MT_{m,n}(x,y)=(\oo{1}{\mt_{m,n}}(x,y),\oo{0}{\mt}_{m,n}(x,y))$ from height $n+1$ to height $m$. By this convention we let $\MT_{n,n}=\MT_n$. Observe that $\oo{0}{\mt}_{m,n}$ is affine and depends
upon $y$ only. 
Let us define points $\tau_n$ inductively by $\tau_0=\tau$ and $\tau_{n+1}=\MT_{n}^{-1}(\tau_n)$. We will call $\tau_n$ the \emph{tip at height $n$}. We wish to use the decompositions
\begin{align}
\MT_n(\tau_{n+1}+z)
&=\MT_{n}(\tau_{n+1})+\D{\MT_{m,n}}{\tau_{n+1}}\circ (\id+\Rem{\MT_{n}}{\tau_{n+1}})(z) \\
&=\tau_n+\D{\MT_{n}}{\tau_{n+1}}\circ (\id+\Rem{\MT_{n}}{\tau_{n+1}})(z) \notag
\end{align}
and
\begin{align}\label{eq:scope1}
\MT_{m,n}(\tau_{n+1}+z)
&=\MT_{m,n}(\tau_{n+1})+\D{\MT_{m,n}}{\tau_{n+1}}\circ (\id+\Rem{\MT_{m,n}}{\tau_{n+1}})(z_1) \\
&=\tau_m+\D{\MT_{m,n}}{\tau_{n+1}}\circ (\id+\Rem{\MT_{m,n}}{\tau_{n+1}})(z_1), \notag
\end{align}
whenever $\tau_{n+1}+z$ is in $\Dom(\MT_n)$ or $\Dom(\MT_{m,n})$ respectively. For notational simplicity let us denote the derivatives $\D{\MT_{n}}{\tau_{n+1}},
\D{\MT_{m,n}}{\tau_{n+1}}$ and remainder terms, $\Rem{\MT_n}{\tau_{n+1}}$ and $\Rem{\MT_{m,n}}{\tau_{n+1}}$, by $D_n, D_{m,n}, R_n$ and $R_{m,n}$ respectively. 

It will turn out to be fruitful to change to coordinates in which the tips are situated at the origin. Therefore let $\TT_n(z)=z-\tau_n$ and consider the maps
$\afa{\MT}_n=\TT_n\circ\MT_n\circ \TT_{n+1}^{-1}$ and their composites
\begin{equation}
\afa{\MT}_{m,n}=\afa{\MT}_m\circ\cdots\circ\afa{\MT}_n=\TT_m\MT_{m,n}\TT_{n+1}^{-1}.
\end{equation}
From Proposition~\ref{prop:remainder-1} we know, since $\TT_n$ is a translation, that $\Rem{\afa{\MT}_n}{z_0}=\Rem{\MT_n}{\TT_{n+1}^{-1}z_0}$.
Therefore using the same decomposition as above we find, 
\begin{align}
\afa{\MT}_n(z)&=\afa{\MT}_n(0+z) \\
&=\afa{\MT}_n(0)+\D{\afa{\MT}_n}{0}(\id+\Rem{\afa{\MT}_n}{0})(z) \notag \\
&=\D{\MT_n}{\tau_{n+1}}(\id+\Rem{\MT_n}{\tau_{n+1}})(z) \notag
\end{align}
and similarly
\begin{align}
\afa{\MT}_{m,n}(z)&=\afa{\MT}_{m,n}(0+z) \\
&=\afa{\MT}_{m,n}(0)+\D{\afa{\MT}_{m,n}}{0}(\id+\Rem{\afa{\MT}_{m,n}}{0})(z) \notag \\
&=\D{\MT_{m,n}}{\tau_{n+1}}(\id+\Rem{\MT_{m,n}}{\tau_{n+1}})(z) \notag
\end{align}
For notational simplicity let us denote the quantities $\D{\afa{\MT}_{n}}{0}$, $\D{\afa{\MT}_{m,n}}{0}$, $\Rem{\afa{\MT}_n}{0}$ and $\Rem{\afa{\MT}_{m,n}}{0}$, by $\afa{D}_n, \afa{D}_{m,n}, \afa{R}_n$ and $\afa{R}_{m,n}$ respectively. Observe that, because our coordinate changes were translations, these quantities are equal to $D_n, D_{m,n}, R_n$ and $R_{m,n}$ respectively.
The following follows directly from Lemma~\ref{lem:scope-decomposition}.
\begin{lem}\label{lem:D_n-decomposition}
For any $F\in\I_{\Omega,\perm}(\bar\e_0)$ let the linear map $D_n$ and the function $R_n(z)$ be as above. Then $D_n$ and $R_n(z)$ have the respective forms
\begin{equation}
D_n=\sigma_n\iibyii{s_n}{t_n}{0}{1}; \quad R_n(z)=\ibyii{r_n(z)}{0}.
\end{equation}
\end{lem}
\begin{defn}
The quantities $s_n$ and $t_n$ from the preceding Lemma will be called, respectively, the \emph{squeeze} and \emph{tilt} of $\MT_n$ at $\tau_{n+1}$.
\end{defn}
\begin{prop}\label{prop:cantor-convergence}
For $F\in\I_{\Omega,\perm}(\bar\e)$, let $\oo{\word{w}{}}{B}_n$ denote the box of height $n$ with word $\word{w}{}\in
W^*$. Then
\begin{enumerate}
\item for each $\word{w}{}\in W^*$, $\dist_{Haus}(\oo{\word{w}{}}{B}_n,\oo{\word{w}{}}{B}_*)\to 0$ exponentially;
\item for each $\word{w}{}\in \oline{W}$, $\dist_{Haus}(\oo{\word{w}{}}{\Cantor}_n,\oo{\word{w}{}}{\Cantor}_*)\to 0$ exponentially.
\end{enumerate} 

\end{prop}
\begin{prop}\label{prop:D_n-estimate}
There exist constants $C>1$, and $0<\rho<1$ such that the following holds:
given $0<\bar\e<\bar\e_0$ let $F\in\I_{\Omega,\perm}(\bar\e)$ and for each integer $n>0$ let $\sigma_n, s_n, t_n$ be the constants and $r_n(z)$ the function defined above. 
Then for any $z\in B$,
\begin{align}
&\sigma(1-C\rho^n)<|\sigma_n|<\sigma(1+C\rho^n) \\
&\sigma(1+C\rho^n)<|s_n|<\sigma(1+C\rho^n) \\
&C^{-1}\bar\e^{p^n}<|t_n|<C\bar\e^{p^n} \\
&|\del_xr_n(z)|<C|z|, |\del_yr_n(z)|<C\bar\e^{p^n}|z| \\
&|\del_{xx}r_n(z)|<C|z|, |\del_{xy}r_n(z)|<C\bar\e^{p^n}|z|, |\del_{yy}r_n(z)|<C\bar\e^{p^n}|z|
\end{align}
\end{prop}
\begin{proof}
Observe that $\sigma_n$ is the eigenvalue of $D{\III_n^{-1}}{}$, the affine bijection between $\oo{0}{B}_{n,diag}$ and $B$. 
By Proposition~\ref{prop:cantor-convergence} there exists a constant $C_0>0$ such that $\dist_{Haus}(\oo{0}{B_n},\oo{0}{B_*})<C_0\rho^n$ we see that $|\sigma_n-\sigma_*|<C_0\rho^n$. Next observe that
$s_n=\del_x\oo{1}{\mt}_n(\tau_{n+1})$ and, by Lemma~\ref{lem:univ-linearisation}, $\sigma=\del_x\oo{1}{\mt_{f_*}}(\tau_*)$ which implies
\begin{align}
|s_n-\sigma|
&\leq |\del_x\oo{1}{\mt_n}(\tau_{n+1})-\del_x\oo{1}{\mt_*}(\tau_{n+1})|+|\del_x\oo{1}{\mt_*}(\tau_{n+1})-\del_x\oo{1}{\mt_*}(\tau_*)| \\
&\leq |\oo{1}{\mt_n}-\oo{1}{\mt_*}|_{\Omega}+|\del_{xx}\oo{1}{\mt_*}|_{\Omega}|\pi_x(\tau_{n+1})-\pi_x(\tau_*)|. \notag
\end{align}
Again by Proposition~\ref{prop:cantor-convergence} $|\tau_n-\tau_*|<C_0\rho^n$. Also, a consequence of Theorem~\ref{thm:R-convergence} is that there exists a constant $C_1>0$
such that
$|\oo{1}{\mt_n}-\oo{1}{\mt_*}|_{\Omega}<C_1\rho^n$. Since fixing the combinatorial type
fixes the map $\oo{1}{\mt_*}$, we may assume $|\del_{xx}\oo{1}{\mt_*}|_{\Omega}<C_2$ for some constant $C_2>0$. Therefore
\begin{equation}
|s_n-\sigma|\leq C_1\rho^n+C_0C_2\rho^n = (C_1+C_0C_2)\rho^n.
\end{equation}
Now for each $n>0$ choose a $\tilde F_n\in\H_{\Omega,\perm}(0)$ such that $|F_n-\tilde F_n|_{\Omega}<C_3\bar\e^{p^n}$, where $C_3>0$ is the constant from
Theorem~\ref{thm:R-construction}. A consequence of convergence of Renormalisation, Theorem~\ref{thm:R-convergence}, is that there exists a constant $C_4>0$ such that $|\del_y\oo{1}{\mt}_n|=|\del_y\oo{1}{\mt}_n-\del_y\oo{1}{\mt}_{f_n}|<C_4\bar\e^{p^n}$. This concludes the first item. For the next two items we apply Proposition~\ref{prop:r-estimate}.
\end{proof}

\begin{lem}\label{lem:D_mn-decomposition}
For any $F\in\I_{\Omega,\perm}(\bar\e_0)$ let the linear map $D_{m,n}$ and the function $R_{m,n}(z)$ be as above.
Then $D_{m,n}$ and the function $R_{m,n}(z)$ have the respective form
\begin{equation}\label{eq:scope3}
D_{m,n}=\sigma_{m,n}\iibyii{s_{m,n}}{t_{m,n}}{0}{1}; \quad R_{m,n}(z)=\ibyii{r_{m,n}(z)}{0},
\end{equation}
respectively, and so if $\tau_{m}=(\xi_m,\eta_n)$,
\begin{equation}\label{eq:scope2}
\MT_{m,n}(z)=\tau_m+\sigma_{m,n}\ibyii{s_{m,n}\left((x-\xi_m)+r_{m,n}(z-\tau_m)\right)+t_{m,n}(y-\eta_m)}{y-\eta_m}.
\end{equation}
Moreover,
\begin{equation}
\sigma_{m,n}=\prod_{i=m}^n\sigma_i; \quad s_{m,n}=\prod_{i=m}^n s_i; \quad t_{m,n}=\sum_{i=m}^n s_{m,i-1}t_i.
\end{equation}
\end{lem}
\begin{proof}
From Lemma~\ref{lem:D_n-decomposition} we know it holds for $m=n$. For $m<n$ the chain rule $D_{m,n}=D_{m,n-1}D_n$ implies $D_{m,n}$ is again upper triangular and
\begin{equation}
\sigma_{m,n}=\sigma_{m,n-1}\sigma_n, \quad s_{m,n}=s_{m,n-1}s_n, \quad t_{m,n}=s_{n-1}t_n+t_{m,n-1},
\end{equation}
from which the lemma immediately follows by induction.
\end{proof}
\begin{prop}\label{prop:D_mn-estimate}\label{scopeestimate}
There exist constants $C>0$, and $0<\rho<1$ such that the following holds:
for $F\in\I_{\Omega,\perm}(\bar\e)$, let $\sigma_{m,n}, s_{m,n}, t_{m,n}$ be the constants and $r_{m,n}(z)$ the function defined above. Then
\begin{align}
&\sigma^{n-m}(1-C\rho^m)<|\sigma_{m,n}|<\sigma^{n-m}(1+C\rho^m) \label{ineq:scope1} \\ 
&\sigma^{n-m}(1-C\rho^m)<|s_{m,n}|<\sigma^{n-m}(1+C\rho^m) \label{ineq:scope2} \\
&|t_{m,n}|<C\bar\e^{p^m} \\
&|\del_xr_{m,n}(z)|<C|z|, |\del_yr_{m,n}(z)|<C\bar\e^{p^{m-1}}|z| \\
&|\del_{xx}r_{m,n}(z)|<C|z|, |\del_{xy}r_{m,n}(z)|<C\sigma^{2(n-m)}\bar\e^{p^m}|z|, |\del_{yy}r_{m,n}(z)|<C\bar\e^{p^m}|z|.
\end{align}
\end{prop}
\begin{proof}
Throughout the proof $C_0>0$ will denote the constant from Proposition~\ref{prop:D_n-estimate}.
From  Lemma~\ref{lem:D_mn-decomposition}, Proposition~\ref{prop:D_n-estimate} and Proposition~\ref{prop:1+Crho} respectively, we find there exists a constant $C_1>0$ such that
\begin{equation}
|\sigma_{m,n}|=\prod_{i=m}^n|\sigma_i|\leq \sigma^{n-m}\prod_{i=m}^n(1+C_0\rho^i)\leq \sigma^{n-m}(1+C_1\rho^m)
\end{equation}
and similarly
\begin{equation}
|s_{m,n}|=\prod_{i=m}^n|s_i|\leq \sigma^{n-m}\prod_{i=m}^n(1+C_0\rho^i)\leq \sigma^{n-m}(1+C_1\rho^m).
\end{equation}
Again by Lemma~\ref{lem:D_mn-decomposition} and Proposition~\ref{prop:D_n-estimate} above we find, for $i>m$,
\begin{equation}
\left|\frac{t_i}{t_m}\right|=\left|\frac{\del_y\oo{p-1}{\phi}_i(\tau_{i+1})}{\del_y\oo{p-1}{\phi}_m(\tau_{m+1})}\right|\left|\frac{\del_x\oo{p-1}{\phi}_m(\tau_{m+1})}{\del_x\oo{p-1}{\phi}_i(\tau_{i+1})}\right|\leq C_0^4\bar\e^{p^{i+1}-p^m}.
\end{equation}
Therefore, by Lemma~\ref{lem:e-vs-rho} there exists a constant $C_2>0$ such that
\begin{align}
|t_{m,n}|
&\leq |t_m|\sum_{i=m}^n |s_{m,i-1}|\left|\frac{t_i}{t_m}\right| \\
&\leq C_0^2\bar\e^{p^m}\sum_{i=m}^n\sigma^{i-m-1}\bar\e^{p^{i+1}-p^m}(1+C_1\rho^i). \notag \\
&\leq C_2\bar\e^{p^m}. \notag 
\end{align}
This concludes the first item. For the second and third items we will proceed by induction.
The case when $m=n$ is shown in Proposition~\ref{prop:D_n-estimate} so, for $m+1\leq n$, assume the inequalities hold for $r_{m+1,n}$ and consider $r_{m,n}$. 
Choose $z=(x,y)\in\RR^2$ such that $\tau_{n+1}+z\in\Dom(\MT_{m,n})$. Then since $\MT_{m,n}=\MT_m\circ\MT_{m+1,n}$, decomposing the left hand side gives
\begin{equation}
\MT_{m,n}(\tau_{n+1}+z)=\tau_m+D_{m,n}(\id+R_{m,n})(z)
\end{equation}
and decomposing the right hand side and applying Proposition~\ref{prop:remainder-1} gives us
\begin{align}
&\MT_m(\MT_{m+1,n}(\tau_{n+1}+z)) \\
&=\MT_n(\tau_{m+1}+D_{m+1,n}(\id+R_{m+1,n})(z)) \notag \\
&=\tau_m+D_m(\id+R_m)\left(D_{m+1,n}(\id+R_{m+1,n})(z)\right) \notag \\
&=\tau_m+D_{m,n}(\id+R_{m+1,n})(z)+D_m\left(R_m(D_{m+1,n}(\id+R_{m+1,n})(z))\right). \notag
\end{align}
Equating these and making appropriate cancellations then gives
\begin{equation}
R_{m,n}(z)=R_{m+1,n}(z)+D_{m+1,n}^{-1}\left(R_m(D_{m+1,n}(\id+R_{m+1,n})(z))\right).
\end{equation}
By definition, $R_{m,n}(z)=(r_{m,n}(z),0)$, $R_{m+1,n}(z)=(r_{m+1,n}(z),0)$ and $R_m(z)=(r_m(z),0)$. Therefore setting $z'=(x',y')=D_{m+1,n}(\id+R_{m+1,n})(z)$, that is
\begin{equation}
(x',y')=(\sigma_{m+1,n}s_{m+1,n}(x+r_{m+1,n}(x,y))+\sigma_{m+1,n}t_{m+1,n}y,\sigma_{m+1,n}y),
\end{equation}
we find that
\begin{equation}
r_{m,n}(x,y)=r_{m+1,n}(x,y)+\sigma_{m+1,n}^{-1}s_{m+1,n}^{-1}r_m(x',y').
\end{equation}
Differentiating this with respect to $x$ and $y$ gives
\begin{align}
\del_x r_{m,n}(x,y)
&=\del_xr_{m+1,n}(x,y)+(1+\del_xr_m(x,y))\del_xr_m(x',y') \\
\del_y r_{m,n}(x,y)
&=\del_yr_{m+1,n}(x,y)+s_{m+1,n}^{-1}\left(t_{m+1,n}\del_xr_m(x',y')+\del_yr_m(x',y')\right).
\end{align}
Now let $C_4>1$ be the maximum of the constant from Proposition~\ref{prop:D_n-estimate} and the constant from the first item above which ensures
\begin{equation}
|s_{m+1,n}|>C_4^{-1}\sigma^{n-m-1},\quad |t_{m+1,n}|<C_4\bar\e^{p^{m+1}}, 
\end{equation}
and
\begin{equation}
|\del_xr_m(z)|<C_4|z|, \quad |\del_yr_m(z)|<C_4\bar\e^{p^m}|z|,\quad |\del_{xy}r_m(z)|<C_4\bar\e^{p^m}|z|.
\end{equation}
As a consequence of our induction hypothesis, there exists a constant $C_5>0$ such that $|z'|<C_5\sigma^{n-m-1}|z|$. Together these imply the existence of a constant $C_6>0$ such that
\begin{align}
|\del_x r_{m,n}(z)|
&\leq |\del_xr_{m+1,n}(z)|+|\del_xr_m(z')|(1+|\del_xr_m(z)|) \\
&\leq |\del_xr_{m+1,n}(z)|+C_4|z'|(1+C_4|z|) \notag \\
&\leq |\del_xr_{m+1,n}(z)|+C_4C_5\sigma^{n-m-1}|z|(1+C_4|z|) \notag \\
&\leq |\del_xr_{m+1,n}(z)|+C_6\sigma^{n-m-1}|z| \notag
\end{align}
and a constant $C_7>0$ such that
\begin{align}
|\del_y r_{m,n}(z)|
&\leq |\del_yr_{m+1,n}(z)|+|s_{m+1,n}|^{-1}\left(|\del_xr_m(z')||t_{m+1,n}|+|\del_yr_m(z')|\right). \\
&\leq |\del_yr_{m+1,n}(z)|+C_4\sigma^{-(n-m-1)}(C_4^2\bar\e^{p^{m+1}}|z'|+C_4\bar\e^{p^m}|z'|) \notag \\
&\leq |\del_yr_{m+1,n}(z)|+C_4^2C_5(C_4\bar\e^{p^{m+1}}+\bar\e^{p^m})|z| \notag \\
&\leq  |\del_yr_{m+1,n}(z)|+C_7\bar\e^{p^m}|z| \notag
\end{align}
Next we consider the second order derivatives. As all functions are analytic the estimates for $\del_{xx}r_{m,n}$ and $\del_{yy}r_{m,n}$ follow from those of $\del_xr_{m,n}$ and
$\del_yr_{m,n}$ respectively. Therefore we only need consider the mixed second order partial derivative. This is given by
\begin{equation}
\del_{xy}r_{m,n}(z)=\del_{xy}r_{m+1,n}(z)+\sigma_{m+1,n}\del_{xy}r_m(z)\left(\del_{xx}r_m(z')t_{m+1,n}+\del_{xy}r_m(z')\right)
\end{equation}
and hence, using the above estimates, there exists a constant $C_8>0$ such that
\begin{align}
&|\del_{xy}r_{m,n}(z)| \\
&\leq |\del_{xy}r_{m+1,n}(z)|+|\sigma_{m+1,n}||\del_{xy}r_m(z)|\left(|\del_{xx}r_m(z')||t_{m+1,n}|+|\del_{xy}r_m(z')|\right) \notag \\
&\leq |\del_{xy}r_{m+1,n}(z)|+C_4^2\sigma^{n-m-1}\bar\e^{p^m}|z|\left(C_4^2\bar\e^{p^{m+1}}|z'|+C_4\bar\e^{p^m}|z'|\right) \notag \\
&\leq |\del_{xy}r_{m+1,n}(z)|+C_8\sigma^{2(n-m)}\bar\e^{2p^m}|z|. \notag
\end{align}
Therefore invoking the induction hypothesis and setting $C=\max_i C_i$ we achieve the desired result.
\end{proof}

\begin{prop}\label{prop:r_mn-convergence-1}
There exists a constant $0<\rho<1$ such that that following holds:
for $F\in\I_{\Omega,\perm}(\bar\e)$, let $r_{m,n}(x,y)$ denote the functions constructed above for integers $0<m<n$. 
Then there exists a constant $C>0$ such that for any $(x,y)\in B$,
\begin{equation}
|[x+r_{m,n}(x,y)]-v_*(x)|<C(\bar\e^{p^m}y+\rho^{n-m})
\end{equation}
and
\begin{equation}
|[1+\del_xr_{m,n}(x,y)]-\del_x v_*(x)|<C\rho^{n-m}
\end{equation}
where $v_*(x)$ is the affine rescaling of the universal function $u_*$ so that its fixed point lies at the origin with multiplier $1$.
\end{prop}
\begin{proof}
Given $F\in\I_{\Omega,\perm}(\bar\e)$ let $F_n\colon B\to B$ denote the $n$-th renormalisation and let $\MT_n\colon B\to B$ denote the $n$-th scope function. Let
$\afa{F}_n\colon \afa{B}_{n+1}\to\afa{B}_n$ and $\afa{\MT}_n\colon \afa{B}_{n+1}\to\afa{B}_n$ denote these maps under the translational change of coordinates described above.

First, let us consider the functions $\afa{\MT}_m\colon\hat{B}_{m+1}\to\hat{B}_{m}$. By construction these preserve the $x$-axis, since they preserve the family of horizontal
lines and the origin is a fixed point for each of them. 
This implies there exists a functions $\afa{\mt}_m\colon \afa{J}_{m+1}\to\afa{J}_{m}$ such that $\afa{\MT}_{m}(x,0)=(\afa{\mt}_{m}(x),0)$.
Next observe there is a constant $C_0>0$ such that for each $n\geq 0$ there exists $f_n\U_{\Omegax,\perm}$ satisfying $|F_n-(f_n\circ\pi_x,\pi_x)|_{\Omega}<C_0\bar\e^{p^n}$. 
Let $\afa{f}_n\colon \afa{J}_n\to \afa{J}_n$ denote $f_n$ under the translational change of coordinates and let $\oo{1}{\afa{\mt}}_n\colon\afa{J}_n\to\oo{1}{\afa{J}}_n$ be the
branch of its presentation function corresponding to the interval $\oo{1}{\afa{J}}_n$. Proposition~\ref{prop:1d-scope-conv} implies there is a constant $C_1>0$ such that
$|\oo{1}{\afa{\mt}}_n-\afa{\mt}_n|_{\C{2}{J}}<C_1\bar\e^{p^n}$ and Proposition~\ref{prop:1d-scope-conv} and Theorem~\ref{thm:R-convergence} implies there is a constant
$C_2>0$ such that $|\oo{1}{\afa{\mt}}_n-\oo{1}{\afa{\mt}}_{*}|_{\C{2}{J}}<C_2\rho^n$. Combining these we find there is a constant $C_3>0$ such that
\begin{equation}
|\afa{\mt}_n-\oo{1}{\afa{\mt}}_{*}|_{\C{2}{J}}<C_3\rho^n.
\end{equation}
Now observe there exist functions $\afa{\mt}_{m,n}\colon
\afa{J}_{n+1}\to \oo{0}{\afa{J}}_{m,n}\subset\afa{J}_{m}$, where $\oo{0}{\afa{J}}_{m,n}=\afa{\mt}_{m,n}(\afa{J}_{n+1})$, such that
$\afa{\MT}_{m,n}(x,0)=(\afa{\mt}_{m,n}(x),0)$.
Moreover, since $\afa{\MT}_{m,n}=\afa{\MT}_m\circ\cdots\circ\afa{\MT}_n$ we must have $\afa{\mt}_{m,n}=\afa{\mt}_m\circ\cdots\circ\afa{\mt}_n$.
Also observe that, since $\afa{\MT}_{m,n}=\TT_m\circ\MT_{m,n}\circ\TT_{n+1}^{-1}$, there are translations $\tt_m$ such that $\afa{\mt}_{m,n}=\tt_m\circ\mt_{m,n}\circ\tt_{n+1}^{-1}$.

Now let $[\mt_{m,n}]$ and $[\mt_{*,m,n}]$ denote, respectively, the orientation preserving affine rescalings of the maps $\afa{\mt}_m\circ\cdots\circ\afa{\mt}_n$ and
$\afa{\mt}_*\circ\cdots\circ\afa{\mt}_*$ to the interval $J$. Here the composition of $\afa{\mt}_*$ with itself is taken $n-m$ times.
Then Lemma 7.3 in~\cite{dCML} implies there exists a constant $C_4>0$ such that $|\mt_{m,n}-\mt_{*,m,n}|_{\C{1}{J}}<C_4\rho^{n-m}$.
This then implies, together with the second part of Lemma~\ref{lem:univ-linearisation}, that there is a constant $C_5>0$ such that
\begin{align}
|[\mt_{m,n}]-u_*|_{\C{1}{J}}
&\leq |[\mt_{m,n}]-[\mt_{*,m,n}]|_{\C{1}{J}}+|[\mt_{*,m,n}]-u_*|_{\C{1}{J}} \\
&\leq C_5\rho^{n-m}. \notag
\end{align}
where $u_*$ is the universal function from that Lemma. Next we perform an translational change of coordinates on $[\mt_{m,n}]$ and $u_*$ so that the fixed point lies at the
origin. Observe that these coordinate changes also converge exponentially. 
 Therefore, if $[\afa{\mt}_{m,n}]$, and $\afa{u}_*$
denote these functions in the new coordinates, there exists a constant $C_6>0$ such that
\begin{equation}
|[\afa{\mt}_{m,n}]-\afa{u}_*|_{\C{1}{J}} <C_6\rho^{n-m}.
\end{equation}
Now observe that,  
since multipliers of fixed points have uniform Lipschitz-type dependence (by the Cauchy estimates), this implies the difference between the multiplier $\mu_{m,n}$ of the fixed point $0$ for $[\afa{\mt}_{m,n}]$ and the
multiplier $\mu_*$ of the
fixed point $0$ for $\afa{u}_*$ decreases exponentially in $n-m$ at the same rate. This implies there exists a constant $C_7>0$ such that
\begin{equation}
|\mu_{m,n}^{-1}[\afa{\mt}_{m,n}]-\mu_*^{-1}\afa{u}_*|_{\C{1}{J}} <C_7\rho^{n-m}.
\end{equation}
Now we claim that $\mu_{m,n}^{-1}[\afa{\mt}_{m,n}]=x+r_{m,n}(x,0)$. Both come from affinely rescaling $\MT_{m,n}$ so that the origin is fixed, the horizontal line
$\{y=0\}$ is fixed and their derivatives in the $x$-direction are 1. Hence they are equal. Also, by definition, $\mu_*^{-1}\afa{u}_*=v_*$.
This then implies, by the above and Proposition~\ref{prop:D_mn-estimate}, that there is a constant $C>0$ such that
\begin{align}
&|[x+r_{m,n}(x,y)]-v_*(x)| \\
&\leq |[x+r_{m,n}(x,y)]-[x+r_{m,n}(x,0)]|+|[x+r_{m,n}(x,0)]-v_*(x)| \notag \\
&\leq |\del_yr_{m,n}||y|+|\mu_{m,n}^{-1}[\afa{\mt}_{m,n}]-\mu_*^{-1}\afa{u}_*|_{\C{0}{J}} \notag \\
&\leq C(\bar\e^{p^{m-1}}|y|+\rho^{n-m}) \notag
\end{align}
which gives the first bound while 
\begin{align}
&|[1+\del_xr_{m,n}(x,y)]-\del_xv_*(x)| \\
&\leq |[1+\del_xr_{m,n}(x,y)]-[1+\del_xr_{m,n}(x,0)]|+|[1+\del_xr_{m,n}(x,0)]-\del_xv_*(x)| \notag \\
&\leq |\del_{xy}r_{m,n}||y|+|\mu_{m,n}^{-1}[\afa{\mt}_{m,n}]-\mu_*^{-1}\afa{u}_*|_{\C{1}{J}} \notag \\
&\leq C(\sigma^{n-m}\bar\e^{p^m}|y|+\rho^{n-m}) \notag
\end{align}
which, since $z$ lies in a bounded domain and $\bar\e^{p^m}$ is bounded from above, gives us the bound for the derivate.
\end{proof}

\begin{prop}\label{prop:r_mn-convergence-2}
There exist constants $C>0, 0<\rho<1$ such that the following holds:
given $F\in\I_{\Omega,\perm}(\bar\e)$, for each integer $m>0$ there exists a constant $\kappa_{(m)}=\kappa_{(m)}(F)\in\RR$, satisfying $|\kappa_{(m)}|<C\bar\e^{p^m}$, such that
\begin{equation}\label{ineq:scope4}
|[x+r_{m,n}(x,y)]-[v_*(x)+\kappa_{(m)}y^2]|<C\rho^{n-m}
\end{equation}
\end{prop}
\begin{proof}
Observe that, since $v_*(0)=0$, Proposition~\ref{prop:r_mn-convergence-1} tells us there exists a constant $C_0>0$ and a point $\xi_{0,x}\in[0,x]$ such that
\begin{align}
&|[x+r_{m,n}(x,y)]-[v_*(x)+r_{m,n}(0,y)]| \\
&=|[x+r_{m,n}(x,y)-v_*(x)]-[0+r_{m,n}(0,y)-v_*(0)]| \notag \\
&\leq |1+\del_xr_{m,n}(\xi_{0,x},y)-\del_xv_*(\xi_{0,x})||x| \notag \\
&\leq C_0\rho^{n-m}|x| \notag
\end{align}
We now claim there exists a constant $\kappa_{(m)}$ such that $|\kappa_{(m)}|<C\bar\e^{p^m}$ and
\begin{equation}
|r_{m,n}(0,y)-\kappa_{(m)}y^2|<C_1\rho^n.
\end{equation}
To show this we use induction. Recall that $\MT_{m,n}(z)=\MT_{m,n-1}\circ\MT_{n}(z)$ for $z\in B$. This implies
\begin{equation}
R_{m,n}(z)=R_n(z)+D_n^{-1}(R_{m,n-1}(D_n(\id+R_n(z)))).
\end{equation}
Since $R_{m,n}, R_n$ and $D_n$ have the forms given by Lemmas~\ref{lem:D_n-decomposition} and~\ref{lem:D_mn-decomposition}, we find that, setting $z'=\MT_n(z)$,
\begin{equation}
(x',y')=(\sigma_ns_n(x+r_n(x,y))+\sigma_nt_ny,\sigma_ny),
\end{equation}
where we write $(x',y')$ for $z'$.
This then gives us
\begin{equation}
r_{m,n}(x,y)=r_n(x,y)+\sigma_n^{-1}s_n^{-1}r_{m,n-1}(x',y').
\end{equation}
Let $\omega_n(y)=\sigma_n(s_nr_n(0,y)+t_ny)$. Then in particular, this together with the Mean Value Theorem implies there exists a $\xi\in [0,\omega_n(y)]$ such that
\begin{align}\label{eqn:r_mn-inductive}
r_{m,n}(0,y)
&=r_n(0,y)+\sigma_{n}^{-1}s_{n}^{-1}r_{m,n-1}(\omega_n(y),\sigma_ny) \\
&=r_n(0,y)+\sigma_{n}^{-1}s_{n}^{-1}\left(r_{m,n-1}(0,\sigma_n y)+\del_xr_{m,n-1}(\xi,\sigma_ny)\omega_n(y)\right). \notag
\end{align}
Next observe that, by construction, $r_n(x,y)$ consists of degree two terms or higher. Therefore, by the above equation, so too must $r_{m,n}(x,y)$. Thus, we may write $r_n(0,y)$ and
$r_{m,n}(0,y)$ in the forms
\begin{equation}
r_n(0,y)=\kappa_ny^2+K_n(y); \quad r_{m,n}(0,y)=\kappa_{m,n}y^2+K_{m,n}(y),
\end{equation}
where $\kappa_n,\kappa_{m,n}$ are real constants and $K_n(y), K_{m,n}(y)$ are functions of the third order in $y$. This implies together with equation~\eqref{eqn:r_mn-inductive}, that
\begin{align}
\kappa_{m,n}y^2+K_{m,n}(y)
&=\kappa_ny^2+K_n(y)  \\ 
&+\sigma_{n}^{-1}s_{n}^{-1}\left(\kappa_{m,n-1}y^2+K_{m,n-1}(y)+\del_xr_{m,n-1}(\xi,\sigma_ny)\omega_n(y)\right) \notag
\end{align}
By Proposition~\ref{prop:D_n-estimate} there exists a constant $C_1>0$ such that
$|\del_yr_n(z)|<C_1\bar\e^{p^n}|z|$ for all suitable $z$. Therefore $\kappa_n$ is satisfies $|\kappa_n|<C_1\bar\e^{p^n}$ and $K_n$ satisfies $|K_n(y)|<C_1\bar\e^{p^n}|y|^3$.
Proposition~\ref{prop:D_n-estimate} also implies there exists a constant $C_2>0$ such that $|\omega(y)|<C_2\bar\e^{p^n}|y|$. Proposition~\ref{prop:D_mn-estimate} implies there
exists a constant $C_3>0$ such that $|\del_xr_{m,n-1}(x,y)|<C_3$. These imply, there is a constant $C_4>0$ such that
\begin{align}
|\kappa_{m,n}|
&\leq|\kappa_n|+|\sigma_n s_n^{-1}||\kappa_{m,n-1}|+C_4\bar\e^{p^n} \\
&\leq 2C_4\bar\e^{p^n}+(1+C_4\rho^n)|\kappa_{m,n-1}| \notag \\
|K_{m,n}(y)|
&\leq |K_n(y)|+|\sigma_n^2 s_n^{-1}||K_{m,n-1}(y)|+C_4\bar\e^{p^n} \\
&\leq \sigma(1+C_4\rho^n)|K_{m,n-1}(y)|+2C_4\bar\e^{p^n} \notag
\end{align}
which implies $\kappa_{m,n}$ converges as $n$ tends to infinity and $K_{m,n}(y)$ decreases exponentially if $n$ is sufficiently large.  Moreover, by
Proposition~\ref{lem:e-vs-rho},
$\kappa_{(m)}=\lim_{n\to\infty}\kappa_{m,n}$ satisfies $|\kappa_{(m)}|\leq C_5\bar\e^{p^n}$ for some constant $C_5>0$. Hence the Proposition is shown.
\end{proof}

\section{Three Applications}\label{sect:applications}
We extend three the results in~\cite{dCML} to the case of arbitrary combinatorics using the results of the previous section. First we will show that universality holds at the tip. By this we mean the rate of convergence to the renormalisation fixed point is  controlled by a universal quantity. In the unimodal case this is a positive real number, but here the quantity is a real-valued real analytic function. This universality is then used to show our two other results, namely the non-existence of continuous invariant linefields on the renormalisation Cantor set and the non-rigidity of these Cantor sets.
\subsection{Universality at the Tip}
\begin{thm}\label{thm:2d-universality}
There exists a constant $\bar\e_0>0$, a universal constant $0<\rho<1$ and a universal function $a\in C^\omega(J,\RR)$ such that the following holds:
Let $F\in\I_{\Omega,\perm}(\bar\e_0)$ and let the sequence of renormalisations be denoted by $F_n$. Then
\begin{equation}
F_n(x,y)=\left(f_n(x)+b^{p^n}a(x)y\left(1+\bigo\left(\rho^n\right)\right),y\right)
\end{equation}
where $b=b(F)$ denotes the average Jacobian of $F$ and $f_n$ are unimodal maps converging exponentially to $f_*$.
\end{thm}
\begin{proof}
Let $F_n=(\phi_n,\pi_x)$ denote the $n$-th renormalisation of $F$. Let $\tau_n$ denote the tip of height $n$ and let $\varsigma\in\Dom(F_n)$ be any other point. 
Applying the chain rule to $F_n=\MT_{0,n-1}^{-1}\circ \o{p^n}{F}\circ\MT_{0,n-1}$ at the point $\varsigma$ gives
\begin{equation}\label{eqn:2d-universality:chainrule}
\del_y\phi_n(\varsigma)=\jac{F_n}{\varsigma}=\jac{\o{p^n}{F}}{\MT_{0,n-1}(\varsigma)}\frac{\jac{\MT_{0,n-1}}{\varsigma}}{\jac{\MT_{0,n-1}}{F_n(\varsigma)}}.
\end{equation}
By the Distortion Lemma~\ref{lem:distortion}, since $\MT_{0,n-1}(\varsigma)\in \oo{0^n}{B}$, there exists a constant $C_0>0$ such that
\begin{equation}\label{eqn:2d-universality-distortion}
\left|\jac{\o{p^n}{F}}{\MT_{0,n-1}(\varsigma)}\right|\leq b^{p^n}\left(1+C_0\rho^n\right).
\end{equation}
It is clear from the decomposition in Lemma~\ref{lem:D_mn-decomposition} that 
\begin{equation}
\jac{\MT_{0,n-1}}{\varsigma}
=\jac{\MT_{0,n-1}}{\tau_n}\jac{\left(\id+R_{0,n-1}\right)}{\varsigma-\tau_n}
\end{equation}
and
\begin{equation}
\jac{\MT_{0,n-1}}{F_n(\varsigma)}
=\jac{\MT_{0,n-1}}{\tau_n}\jac{\left(\id+R_{0,n-1}\right)}{F_n(\varsigma)-\tau_n}.
\end{equation}
Let $\delta_n^0=\varsigma-\tau_n$ and $\delta_n^1=F_n(\varsigma)-\tau_n$. Observe that, by Theorem~\ref{thm:R-convergence} and Corollary~\ref{cor:cantor-converge}, there exists a constant $C_1>0$ such that $|\tau_n-\tau_*|, |F_n-F_*|_\Omega<C_1\rho^n$. 
Therefore there exists a constant $C_2>0$ such that, if $\varsigma_*=\tau_*+(\varsigma-\tau_n)$, $\delta_*^0=\varsigma_*-\tau_*$ and $\delta_*^1=F_*(\varsigma_*)-\tau_*$,
\begin{equation}
\left|\delta_n^0-\delta_*^0\right|=\left|\left[\varsigma-\tau_n\right]-\left[\varsigma_*-\tau_*\right]\right|=0
\end{equation}
and
\begin{equation}
\left|\delta_n^1-\delta_*^1\right|=\left|\left[F_n(\varsigma)-\tau_n\right]-\left[F_*(\varsigma_*)-\tau_*\right]\right|<C_2\rho^n.
\end{equation}
By Proposition~\ref{prop:r_mn-convergence-1} there is a constant $C_3>0$ such that
\begin{equation}
\left|1+\del_xr_{0,n-1}-v_*'\right|_{\C{0}{J}}<C_3\rho^n.
\end{equation}
Combining these and observing that $v_*$ has bounded derivatives and $\delta_n^0$ and $\delta_n^1$ both lie in a bounded domain gives us a constant $C_4>0$ satisfying
\begin{align}
&\left|\jac{\left(\id+R_{0,n-1}\right)}{\delta_n^0}-v_*'\left(\pi_x\left(\delta_*^0\right)\right)\right| \\
&\leq \left|\jac{\left(\id+R_{0,n-1}\right)}{\delta_n^0}-v_*'\left(\pi_x\left(\delta_n^0\right)\right)\right|+\left|v_*'\left(\delta_n^0\right)-v_*'\left(\delta_*^0\right)\right| \notag \\
&\leq
\left|1+\del_xr_{0,n-1}\left(\delta_n^0\right)-v_*\left(\pi_x\left(\delta_n^0\right)\right)\right|\left|\tau_n-\tau_*\right|+\left|v_*''\right|_{\C{0}{J}}\left|\tau_n-\tau_*\right| \notag \\
&\leq C_2C_3\rho^{2n}+C_2\left|v_*\right|_{\C{2}{J}}\rho^n \notag \\
&\leq C_4\rho^n \notag
\end{align}
and
\begin{align}
&\left|\jac{\left(\id+R_{0,n-1}\right)}{\delta_n^1}-v_*'\left(\pi_x\left(\delta_*^1\right)\right)\right| \\
&\leq \left|\jac{\left(\id+R_{0,n-1}\right)}{\delta_n^1}-v_*'\left(\pi_x\left(\delta_n^1\right)\right)\right|
+\left|v_*'\left(\pi_x\left(\delta_n^1\right)\right)-v_*'\left(\pi_x\left(\delta_*^1\right)\right)\right| \notag \\
&\leq
\left|1+\del_xr_{0,n-1}\left(\delta_n^1\right)-v_*'\left(\pi_x\left(\delta_n^1\right)\right)\right|\left|\delta_n^1\right|
+\left|v_*''\right|_{\C{0}{J}}\left|\delta_n^1-\delta_*^1\right| \notag \\
&\leq C_4\rho^n \notag.
\end{align}
Observe that there exists a constant $C_5>0$ such that $\left|v_*'(x)\right|\geq C_5>0$, as $v_*$ is a rescaling of a diffeomorphism onto its image. Observe also that there
exists an $N>0$ such that $\left|1+\del_xr_{0,n}\right|_{\C{0}{J}}\geq
\half \inf \left|v_*'(x)\right|\geq C_5$ for all $n>N$. Therefore there exists a constant $C_6>1$ such that for all $n>N$,
\begin{equation}
\max\left(1,\left|\frac{v_*'\left(\pi_x\left(\delta_*^0\right)\right)}{v_*'\left(\pi_x\left(\delta_*^1\right)\right)}\right|\right)<C_6; \qquad
C_6^{-1}<\left|\jac{\MT_{0,n}}{\delta_n^1}\right|.
\end{equation}
Therefore, applying Lemma~\ref{lem:ratio-perturb} we find
\begin{align}
\left|\frac{\jac{\MT_{0,n-1}}{\delta_n^0}}{\jac{\MT_{0,n-1}}{\delta_n^1}}-\frac{v_*'\left(\pi_x\left(\delta_*^0\right)\right)}{v_*'\left(\pi_x\left(\delta_*^1\right)\right)}\right|
&=\left|\frac{\jac{\left(\id+R_{0,n-1}\right)}{\delta_n^0}}{\jac{\left(\id+R_{0,n-1}\right)}{\delta_n^1}}-\frac{v_*'\left(\pi_x\left(\delta_*^0\right)\right)}{v_*'\left(\pi_x\left(\delta_*^1\right)\right)}\right| \\
&\leq C_6^2\max_{i=0,1}\left(\left|1+\del_xr_{0,n-1}\left(\delta_n^i\right)-v_*'\left(\pi_x\left(\delta_*^i\right)\right)\right|\right) \notag \\
&\leq C_4C_6^2\rho^n. \notag
\end{align}
Together with equation~\ref{eqn:2d-universality:chainrule} and~\ref{eqn:2d-universality-distortion} this implies,
\begin{equation}
\del_y\phi_n\left(\varsigma\right)=b^{p^n}a(\xi)\left(1+\bigo\left(\rho^n\right)\right)
\end{equation}
where $\varsigma=\left(\xi,\eta\right)$ and
\begin{equation}
a(\xi)=\frac{v_*'\left(\xi-\pi_x\left(\tau_*\right)\right)}{v_*'\left(f_*\left(\xi\right)-\pi_x\left(\tau_*\right)\right)}.
\end{equation}
This implies that, if $z=(x,y)\in B$, upon integrating with respect to the $y$-variable  we find
\begin{equation}
\phi_n(x,y)=g_n(x)+yb^{p^n}a(x)\left(1+\bigo\left(\rho^n\right)\right),
\end{equation}
for some function $g_n$ independent of $y$. But now let $(f_n,\e_n)$ be any parametrisation of $F_n$ such that $|\e_n|\leq C_7\bar\e^{p^n}$ and $|f_n-f_*|<C_8\rho^n$. Here
$C_7>0$ is the constant from Theorem~\ref{thm:R-construction} and $C_8>0$ is the constant from Theorem~\ref{thm:R-convergence}. Then there is a constant $C_9>0$ such that
$|g_n-f_n|=|\e_n-b^{p^n}\pi_y\circ a|\leq C_9\rho^n$. Therefore, for $n>0$ sufficiently large $g_n$ will also be unimodal and $|g_n-f_*|\leq |g_n-f_n|+|f_n-f_*|\leq (C_9+C_8)\rho^n$.
Hence we may absorb their difference into into the $\bigo(\rho^n)$ term.
\end{proof}

The following is an immediate consequence of the proof of above Theorem.
\begin{prop}\label{prop:t_mn-properties}
Let $F\in\I_{\Omega,\perm}(\bar\e)$, let $\MT_{m,n}$ denote the scope function from height $n+1$ to height $m$. Let $t_{m,n}$ denote the tilt of $\MT_{m,n}$ and let $\tau_{m+1}$ denote the tip at height $m+1$. Let $a=a(\tau_*)$ where $a(x)$ is the universal function from Theorem~\ref{thm:2d-universality} above. Then exists constants $C>0$ and $0<\rho<1$ such that for all $0<m<n$ sufficiently large,
\begin{align}\label{ineq:scope3}
ab^{p^m}(1-C\rho^m)<  \left|t_{m}(\tau_{m+1}) \ \right|<ab^{p^m}(1+C\rho^m) \\
ab^{p^m}(1-C\rho^m)<\left|t_{m,n}(\tau_{n+1})\right|<ab^{p^m}(1+C\rho^m).
\end{align}
Moreover $t_{m,*}=\lim_{n\to\infty}t_{m,n}(\tau_{n+1})$ exists and the convergence is exponential.
\end{prop}
\begin{proof}
Let $\tau_m=(\xi_m,\eta_m)$. Recall that
\begin{equation}
t_m=\pm\frac{\del_y\oo{p-1}{\phi}_m(\tau_{m})}{\del_x\oo{p-1}{\phi}_m(\tau_{m})},
\end{equation}
but by the Variational Formula~\ref{prop:var-formula} we know
\begin{equation}
\oo{p-1}{\phi}_m\left(\xi_m,\eta_m\right)=\o{p-1}{f}_m\left(\xi_m\right)+\oo{p-1}{L}_m\left(\xi_m\right)+\e_m\left(\xi_m,\eta_m\right)\left(\o{p-1}{f}_m\right)'\left(\xi_m\right)+\bigo(\bar\e^{2p^m})
\end{equation}
which implies
\begin{align}
\del_x\oo{p-1}{\phi}_m\left(\xi_m,\eta_m\right)
&=\left(\o{p-1}{f}_m\right)'\left(\xi_m\right)+\left(\oo{p-1}{L}_m\right)'\left(\xi_m\right)+\del_x\e_m\left(\xi_m,\eta_m\right)\left(\o{p-1}{f}_m\right)'\left(\xi_m\right)  \notag \\
&+\e_m\left(\xi_m,\eta_m\right)\left(\o{p-1}{f}_m\right)''\left(\xi_m\right)+\bigo(\bar\e^{2p^m}) \\
\del_y\oo{p-1}{\phi}_m(\xi_m,\eta_m)
&=\del_y\e_m(\xi_m,\eta_m)(\o{p-1}{f}_m)'(\xi_m,\eta_m)+\bigo(\bar\e^{2p^m})
\end{align}
Therefore, by the fact that $(\o{p-1}{f}_m)'(\xi_m)$ is uniformly bounded from zero if $n$ is sufficiently large,
\begin{align}
\frac{\del_y\oo{p-1}{\phi}_m\left(\xi_m,\eta_m\right)}{\del_x\oo{p-1}{\phi}_m\left(\xi_m,\eta_m\right)}
&=\frac{\del_y\e_m\left(\xi_m,\eta_m\right)\left(\o{p-1}{f}_m\right)'\left(\xi_m,\eta_m\right)+\bigo(\bar\e^{2p^m})}{\left(\o{p-1}{f}_m\right)'\left(\xi_m\right)+\bigo(\bar\e^{p^n})} \\
&=\left(\del_y\e_m\left(\xi_m,\eta_m\right)+\bigo(\bar\e^{2p^m})\right)\left(1+\bigo(\bar\e^{p^m})\right) \notag \\
&=\del_y\e_m\left(\xi_m,\eta_m\right)+\bigo(\bar\e^{2p^m}). \notag 
\end{align}
Theorem~\ref{thm:2d-universality} above and observing that the $\bigo(\bar\e^{2p^m})$ term can be absorbed into the $\bigo(\rho^m)$ then tells us
\begin{equation}
\left|t_m(\tau_{m+1})\right|=a\left(\xi_m\right)b^{p^m}\left(1+\bigo(\rho^m)\right),
\end{equation}
but by Proposition~\ref{prop:cantor-convergence} we know that $\xi_m$ converges to $\xi_*$ exponentially and so analyticity of $a$ implies $a(\xi_m)=a(\xi_*)(1+\bigo(\rho^m))$. Hence we get the first claim.
Secondly, observe by Lemma~\ref{lem:D_mn-decomposition},
\begin{align}
t_{m,n-1}(\tau_n)
&=\sum_{i=m}^{n-1}s_{m,i-1}(\tau_i)t_i(\tau_{i+1}) \\
&=t_m(\tau_{m+1})\sum_{i=m}^{n-1}s_{m,i-1}(\tau_i)
\left(\frac{\del_x\oo{p-1}{\phi}_m(\tau_m)}{\del_x\oo{p-1}{\phi}_i(\tau_i)}\right)
\left(\frac{\del_y\oo{p-1}{\phi}_i(\tau_i)}{\del_y\oo{p-1}{\phi}_m(\tau_m)}\right) \notag \\
&=t_m(\tau_{m+1})\sum_{i=m}^{n-1}s_{m+1,i}(\tau_{i+1})
\left(\frac{\del_y\oo{p-1}{\phi}_i(\tau_i)}{\del_y\oo{p-1}{\phi}_m(\tau_m)}\right) \notag
\end{align}
Therefore we can write $t_{m,n-1}(\tau_n)=t_m(\tau_{m+1})(1+K_{m,n-1}(\tau_{n+1}))$ where
\begin{equation}
K_{m,n-1}\sum_{i=m+1}^{n-1}s_{m+1,i}(\tau_{i+1})
\left(\frac{\del_y\oo{p-1}{\phi}_i(\tau_i)}{\del_y\oo{p-1}{\phi}_m(\tau_m)}\right).
\end{equation}
By Proposition~\ref{prop:D_mn-estimate} and the Variational Formula~\ref{prop:var-formula}, there exists a constant $C_7>0$ such that $|K_{m,n-1}(\tau_{n+1})|\geq C\bar\e^{p^{m+1}-p^m}$. Absorbing this error into the $\bigo(\rho^m)$ term gives us the second claim. The third claim follows as the terms in $K_{m,n}$ decrease super-exponentially as $n$ tends to infinity, but $\tau_m$ only converges exponentially to $\tau_*$.
\end{proof}

\begin{prop}\label{prop:D_mn-perturb}
Let $F\in\I_{\Omega,\perm}(\bar\e)$ be as above, let $\tau_n$ denote the tip of $F_n$ and let $\varsigma_n=\o{p}{F_n}(\tau_n)$. Then there exists a constant $C>1$ for all $0<m<n$ 
\begin{align}
&C^{-1}|s_{m,n-1}(\tau_n)|\leq |s_{m,n-1}(\varsigma_n)|\leq C|s_{m,n-1}(\tau_n)| \\
&C^{-1}|t_{m,n-1}(\tau_n)|\leq |t_{m,n-1}(\varsigma_n)|\leq C|t_{m,n-1}(\tau_n)| \\
&|s_{m,n-1}(\varsigma_n)-s_{m,n-1}(\tau_n)|>C^{-1}|\varsigma_n-\tau_n| \\
&|t_{m,n-1}(\varsigma_n)-t_{m,n-1}(\tau_n)|>C^{-1}|\varsigma_n-\tau_n|
\end{align} 
\end{prop}
\begin{proof}
These follow from the estimates on the second order terms (i.e. the functions $r_{m,n}$) given by Proposition~\ref{prop:D_mn-estimate} and the observation that $\tau_n,\varsigma_n\in\oo{0}{B_n}$ implies, for $n$ sufficiently large, that the derivatives of $s_{m,n-1}, t_{m,n-1}$ in the rectangle spanned by $\tau_n,\varsigma_n$ will be uniformly bounded.
\end{proof}

\subsection{Invariant Line Fields}\label{sect:linefields}
Let $F\in\I_{\Omega,\perm}(\bar\e)$ and let $\Cantor$ denote its renormalisation Cantor set. We will now consider the space of $F$-invariant line fields on $\Cantor$. 
As we are considering line fields, let us projectivise all the transformations under consideration. Let us take the projection onto the line $\{y=1\}$, and let us denote the projected coordinate by $X$. Then the maps $D(\MT_{m,n};z)$ and $D(\o{p}{F_n};z)$
induce the transformations
\begin{align}
\projD{\MT_{m,n}}{z}(X)
&= s_{m,n}(z)X+t_{m,n}(z) \\ 
\projD{\o{p}{F_n}}{z}(X)
&=\zeta_n(z)\frac{X+\eta_n(z)}{X+\theta_n(z)}.
\end{align}
where $s_{m,n}(z), t_{m,n}(z)$ are as in Section~\ref{sect:scope} and $\zeta_n(z),\eta_n(z),\theta_n(z)$ are given by
\begin{equation}
\zeta_n(z)=\frac{\del_x\oo{p}{\phi_n}(z)}{\del_x\oo{p-1}{\phi_n}(z)},\quad 
\eta_n(z)=\frac{\del_y\oo{p}{\phi_n}(z)}{\del_x\oo{p}{\phi_n}(z)}, \quad 
\theta_n(z)=\frac{\del_y\oo{p-1}{\phi_n}(z)}{\del_x\oo{p-1}{\phi_n}(z)}.
\end{equation}
\begin{prop}\label{prop:etas-estimate}
Let $F\in\I_{\Omega,\perm}(\bar\e)$ be as above. Then there exists a constant $C>1$ such that for all $n>0$
\begin{align}
&C^{-1}<|\zeta_n(\tau_n)|<C \\
&|\eta_n(\tau_n)|<C\bar\e^{p^{n+1}} \\
&|\theta_n(\tau_n)|<C\bar\e^{p^n}
\end{align}
\end{prop}
\begin{proof}
Let $(f_n,\e_n)$ be a parametrisation for $F_n$. Let $v_n$ denote the critical value of $f_n$. Observe, by convergence of renormalisation~\ref{thm:R-convergence}, that $v_n$ and $\pi_x\tau_n$ are exponentially close and so there is a constant $C_0>0$ such that 
\begin{equation}
\left|\left(\o{p-1}{f_n}\right)'\left(v_n\right)\right|, \left|\left(\o{p-1}{f_n}\right)'\left(\pi_x\tau_n\right)\right|>C_0,
\end{equation}
if $n>0$ is sufficiently large. Therefore by the variational formula, there is a constant $C_1>0$ such that
\begin{align}
\left|\frac{\del_x\oo{p}{\phi_n}\left(\tau_n\right)}{\del_x\oo{p-1}{\phi_n}\left(\tau_n\right)}-\frac{(\o{p}{f_n})'\left(v_n\right)}{(\o{p-1}{f_n})'\left(v_n\right)}\right|
&\leq C_1\bar\e^{p^n}.
\end{align}
Now observe, by Theorem~\ref{thm:R-convergence}, that there is a $C_2>0$ such that
\begin{align}
\left|\frac{(\o{p}{f_n})'\left(v_n\right)}{(\o{p-1}{f_n})'\left(v_n\right)}-\frac{(\o{p}{f_*})'\left(v_*\right)}{(\o{p-1}{f_*})'\left(v_*\right)}\right|<C_2\rho^n.
\end{align}
Therefore there exists a $C_3>0$ such that
\begin{align}
&\left|\zeta_n\left(\tau_n\right)-f_*'\left(\o{p}{f_*}\left(v_*\right)\right)\right| \\
&\leq
\left|\frac{\del_x\oo{p}{\phi_n}(z)}{\del_x\oo{p-1}{\phi_n}(z)}
-\frac{(\o{p}{f_n})'\left(v_n\right)}{(\o{p-1}{f_n})'\left(v_n\right)}\right|+\left|\frac{(\o{p}{f_n})'\left(v_n\right)}{(\o{p-1}{f_n})'\left(v_n\right)}-\frac{(\o{p}{f_*})'\left(v_*\right)}{(\o{p-1}{f_*})'\left(v_*\right)}\right| \notag \\
&\leq C_3\rho^n. \notag
\end{align}
Since $f_*'\left(\o{p}{f_*}\left(v_*\right)\right)\neq 0$ (infinitely renormalisable maps are never postcritically finite), this implies for $n>0$ sufficiently large the first item is true.

For the second item, taking the Jacobian of $\o{p}{F_n}$ at $\tau_n$, applying Proposition~\ref{prop:D_n-estimate} and making the same observation regarding
$f_*'\left(\o{p}{f_*}\left(v_*\right)\right)\neq 0$ as above, gives us the result.

The third item follows directly from Proposition~\ref{prop:D_n-estimate}.
\end{proof}

\begin{thm}\label{thm:no-continuous-linefields}
Let $F\in\I_{\Omega,\perm}(\bar\e)$ and let $\Cantor$ denote its renormalisation Cantor set. Then there do not exist any continuous invariant line fields on $\Cantor$. More precisely, if $X$ is an invariant line field then it must be discontinuous at the tip, $\tau$, of $F$.
\end{thm}
\begin{proof}
Let $X$ be a continuous invariant line field on $\Cantor$. 
Let $\tau_n$ denote the tip of $F_n$ and let $\varsigma_n=\o{p}{F}_n(\tau_n)$ denote its first return, under $F_n$, to $\oo{0}{B}_n$.

Before we begin let us define some constants that shall help our exposition. Let $C_0>0$ satisfy $\left|\theta_n\left(\tau_n\right)\right|<C_0\bar\e^{p^n}$ and
$\left|\eta_n\left(\tau_n\right)\right|<C_0\bar\e^{p^{n+1}}$ for all $n>0$. Such a constant exists by Proposition~\ref{prop:etas-estimate}. Let $C_1>1$ satisfy
$C_1^{-1}<\left|\zeta_n\left(\tau_n\right)\right|<C_1$ for all $n>0$. Such a constant exists by Proposition~\ref{prop:etas-estimate}. Let $C_2>0$ satisfy
$\left|t_m\left(\tau_{m+1}\right)\right|, \left|t_{m,n-1}\left(\tau_n\right)\right|<C_2\bar\e^{p^m}$  for all $0<m<n$. Such a constant exists by
Propositions~\ref{prop:D_n-estimate} and~\ref{prop:D_mn-estimate}. Let $C_3>0$ satisfy $\left|s_{m,n-1}\left(\tau_n\right)\right|>C_3\sigma^{n-m-1}$ for all $0<m<n$. Finally let $C_4>1$ satisfy 
\begin{align}
&C_4^{-1}\left|s_{m,n-1}\left(\tau_n\right)\right|\leq \left|s_{m,n-1}\left(\varsigma_n\right)\right|\leq C_4\left|s_{m,n-1}\left(\tau_n\right)\right| \\
&C_4^{-1}\left|t_{m,n-1}\left(\tau_n\right)\right|\leq \left|t_{m,n-1}\left(\varsigma_n\right)\right|\leq C_4\left|t_{m,n-1}\left(\tau_n\right)\right| \\
&\left|s_{m,n-1}\left(\varsigma_m\right)-s_{m,n-1}\left(\tau_m\right)\right|>C_4^{-1}\left|\varsigma_m-\tau_m\right| \\
&\left|t_{m,n-1}\left(\varsigma_m\right)-t_{m,n-1}\left(\tau_m\right)\right|>C_4^{-1}\left|\varsigma_m-\tau_m\right|
\end{align}
for all $0<m<n$. Such a constant exists by Proposition~\ref{prop:D_mn-perturb} above.

Observe that $X$ induces continuous invariant line fields $X_n$ for $F_n$ on $\Cantor_n$, the induced Cantor sets.
Thus
\begin{equation}
X_m(\tau_m)=\projD{\MT_{0,m}^{-1}}{\tau}X(\tau)=\left(X(\tau)-t_{0,m}(\tau_m)\right)/s_{0,m}(\tau_m).
\end{equation}
There are two possibilities: either $X(\tau)=t_{0,*}(\tau_*)=\lim t_{0,m-1}(\tau_m)$, and so $X_m(\tau_m)$ converges to zero (since $t_{0,m}$ converges super-exponentially to
$t_{0,*}$ but $s_{0,m}$ converges only exponentially to $0$), or $X(\tau)\neq t_{0,*}(\tau_*)$, and so $X_m(\tau_m)$ tends to infinity. 

First, let us show the second case cannot occur. 
Let $K,\kappa>0$ be constants. Choose $M>0$ such that $|X_m(\tau_m)|>K$ for all $m>M$. Fix such an $m>M$.
By continuity of $X_m$ there exists a $\delta>0$ such that $|x-y|<\delta$ implies $|X_m(x)-X_m(y)|<\kappa$ for any $x,y\in\Cantor_m$.
Choose $N>m$ such that, for all $n>N$, $|\MT_{m,n-1}(\tau_{n})-\MT_{m,n-1}(\varsigma_{n})|<\delta$. This then implies $|X_m(\MT_{m,n-1}(\tau_{n}))-X_m(\MT_{m,n-1}(\varsigma_{n}))|<\kappa$.

By invariance of the $X_n$, 
\begin{align}\label{eqn:linefield-1}
\left|X_n\left(\varsigma_n\right)\right|=
\left|\projD{\o{p}{F_n}}{\tau_n}\left(X_n\left(\tau_n\right)\right)\right|=\left|\zeta_n\left(\tau_n\right)\right|\left|\frac{X_n\left(\tau_n\right)+\eta_n\left(\tau_n\right)}{X_n\left(\tau_n\right)+\theta_n\left(\tau_n\right)}\right|.
\end{align}
By our above hypotheses we know $|\theta_n(\tau_n)|, |\eta_n(\tau_n)|<C_0\bar\e^{p^n}$. Since $n>m$, we also know $|X_{n}(\tau_{n})|>K$. 
Therefore
\begin{align}
\left|\frac{X_n\left(\tau_n\right)+\eta_n\left(\tau_n\right)}{X_n\left(\tau_n\right)+\theta_n\left(\tau_n\right)}\right|
&\leq \frac{1+|\eta_n\left(\tau_n\right)/X_n\left(\tau_n\right)|}{1-|\theta_n\left(\tau_n\right)/X_n\left(\tau_n\right)|} \\
&\leq \frac{1+C_0\bar\e^{p^n}/K}{1-C_0\bar\e^{p^n}/K} \notag
\end{align}
Therefore, combining this with the above equation~\ref{eqn:linefield-1} and the hypotheses of the second paragraph we find
\begin{align}
\left|X_n\left(\varsigma_n\right)\right|
&\leq C_1 \left(\frac{1+C_0\bar\e^{p^n}/K}{1-C_0\bar\e^{p^n}/K}\right)
\end{align}
Now we apply $\projD{\MT_{m,n-1}}{\varsigma_{n}}$. Then by the definition of the constant $C_4>0$ in the second paragraph and Proposition~\ref{prop:D_mn-estimate}
\begin{align}
\left|X_m\left(\MT_{m,n-1}\left(\varsigma_{n}\right)\right)\right|
&= \left|s_{m,n-1}\left(\varsigma_n\right)X_n\left(\varsigma_n\right)+t_{m,n-1}\left(\varsigma_n\right)\right| \\
&\leq \left|s_{m,n-1}\left(\varsigma_n\right)\right|\left|X_n\left(\varsigma_n\right)\right|+\left|t_{m,n-1}\left(\varsigma_n\right)\right| \notag \\
&\leq C_4\left(\left|s_{m,n-1}\left(\tau_n\right)\right|\left|X_n\left(\varsigma_n\right)\right|+\left|t_{m,n-1}\left(\tau_n\right)\right|\right) \notag \\
&\leq C_4\sigma^{n-m-1}(1+\left|X_n\left(\varsigma_n\right)\right|) \notag
\end{align}
and hence
\begin{align}
&\left|X_m\left(\MT_{m,n-1}\left(\tau_n\right)\right)-X_m\left(\MT_{m,n-1}\left(\varsigma_n\right)\right)\right| \\
&\geq \bigl|\left|X_m\left(\MT_{m,n-1}\left(\tau_n\right)\right)\right|-\left|X_m\left(\MT_{m,n-1}\left(\varsigma_n\right)\right)\right|\bigr| \notag \\
&\geq K- C_4\sigma^{n-m-1}\left[1+C_1 \left(\frac{1+C_0\bar\e^{p^n}/K}{1-C_0\bar\e^{p^n}/K}\right)\right]. \notag
\end{align}
But, by our continuity assumption, this must be less than $\kappa$. For $K>0$ sufficiently large this cannot happen.

So now let us assume $X(\tau)=t_{0,*}$. Then the induced line fields must satisfy $X_m(\tau_m)=t_{m,*}$, for all $m>0$. The idea is, as before, to look at the first returns under $F_m$ of $\oo{0}{B_m}$. We will apply $\projD{\o{p}{F_m}}{\tau_m}$ to the line $X_m(\tau_m)=t_{m,n}$ and take the limit as $n$ tends to infinity.

Proposition~\ref{prop:t_mn-properties} implies, as $t_m(\tau_{m+1})=\pm\del_y\oo{p-1}{\phi}_m(\tau_m)/\del_x\oo{p-1}{\phi}_m(\tau_m)=\pm\eta_m$, that there exists a constant $C_5>0$ for which
\begin{align}
\left|t_{m,n-1}\left(\tau_n\right)+\theta_m\left(\tau_m\right)\right|\leq \left|t_m\left(\tau_{m+1}\right)\right|\left|K_{m,n-1}\left(\tau_{n+1}\right)\right|\leq C_5\bar\e^{p^{m+1}}.
\end{align}
On the other hand, we know $|\eta_m(\tau_m)|<C_0\bar\e^{p^{m+1}}$ and $|t_{m,n-1}(\tau_n)|<C_2\bar\e^{p^m}$ and hence
\begin{align}
\left|t_{m,n-1}\left(\tau_n\right)+\eta_m\left(\tau_m\right)\right|\geq \bigl|\left|t_{m,n-1}(\tau_n)\right|-\left|\eta_m(\tau_m)\right|\bigr|\geq C_2\bar\e^{p^m}-C_0\bar\e^{p^{m+1}}
\end{align}
We also know $|\zeta_m(\tau_m)|>C_1^{-1}$. Therefore there exists a constant $C_6>0$ such that
\begin{align}
\left|\projD{\o{p}{F_m}}{\tau_m}\left(t_{m,n-1}(\tau_n)\right)\right|
&=\left|\zeta_m(\tau_m)\right|\left|\frac{t_{m,n-1}(\tau_n)+\eta_m(\tau_m)}{t_{m,n-1}(\tau_n)+\theta_m(\tau_m)}\right| \\
&\geq C_1^{-1}C_5^{-1}\bar\e^{-p^{m+1}}(C_2\bar\e^{p^m}-C_0\bar\e^{p^{m+1}}) \notag \\
&\geq C_6\bar\e^{-p^{m+1}}. \notag
\end{align}
Now recall $|t_{m,n-1}(\tau_n)|<C_2\bar\e^{p^m}$. Also observe that both of these estimates are independent of $n$. Therefore they still hold when passing to the limit, as $n$ tends to infinity, giving
\begin{equation}
\left|X_m(\varsigma_m)\right|>C_6\bar\e^{-p^{m+1}}, \quad |X_m(\tau_m)|<C_2\bar\e^{p^m}.
\end{equation}
Finally, applying $\MT_{0,m-1}$ and setting $\varsigma=\MT_{0,m-1}(\varsigma_m)$ we find that
\begin{align}
&\left|X(\varsigma)-X(\tau)\right| \\
&=\bigl|\left[s_{0,m-1}(\varsigma_m)X_m(\varsigma_m)+t_{0,m-1}(\varsigma_m)\right]-\left[s_{0,m-1}(\tau_m)X_m(\tau_m)+t_{0,m-1}(\tau_m)\right]\bigr| \notag \\
&\geq\bigl| \left|s_{0,m-1}(\varsigma_m)X_m(\varsigma_m)-s_{0,m-1}(\tau_m)X_m(\tau_m)\right|-\left|t_{0,m-1}(\varsigma_m)-t_{0,m-1}(\tau_m)\right|\bigr| \notag
\end{align}
but by our assumptions in the second paragraph
\begin{align}
&\left|s_{0,m-1}(\varsigma_m)X_m(\varsigma_m)-s_{0,m-1}(\tau_m)X_m(\tau_m)\right| \\
&\geq \bigl|\left|s_{0,m-1}(\varsigma_m)\right|\left|X_m(\varsigma_m)-X_m(\tau_m)\right|-\left|s_{0,m-1}(\varsigma_m)-s_{0,m-1}(\tau_m)\right|\left|X_m(\tau_m)\right|\bigr| \notag \\
&\geq C_4^{-1}\left|s_{0,m-1}(\tau_m)\right|\left|X_m(\varsigma_m)-X_m(\tau_m)\right|-C_4\left|\varsigma_m-\tau_m\right|\left|X_m(\tau_m)\right| \notag
\end{align}
and
\begin{equation}
\left|t_{0,m-1}(\varsigma_m)-t_{0,m-1}(\tau_m)|\leq C_4|\varsigma_m-\tau_m\right|.
\end{equation}
Therefore again by our assumptions in the second paragraph, $|s_{0,m-1}(\tau_m)|>C_3\sigma^m$. Hence, by our bounds on $|X_m(\varsigma_m)|$ and $|X_m(\tau_m)|$ and the above we find
\begin{align}
&\left|X(\varsigma)-X(\tau)\right| \\
&\geq C_4^{-1}C_3\sigma^m\left|X_m(\varsigma_m)-X_m(\tau_m)\right|-C_4\left|\varsigma_m-\tau_m\right|\left|X_m(\tau_m)\right|-C_4\left|\varsigma_m-\tau_m\right| \notag \\
&\geq C_4^{-1}C_3\sigma^m\left(C_2\bar\e^{-p^{m+1}}-C_6\bar\e^{p^m}\right)-C_4\left|\varsigma_m-\tau_m\right|\left(1+C_6\bar\e^{p^m}\right) \notag
\end{align}
However, since $\left|\varsigma_m-\tau_m\right|$ is bounded from above there is a constant $C_7>0$ such that
\begin{equation}
\left|X(\varsigma)-X(\tau)\right|\geq C_7\sigma^m\bar\e^{-p^{m+1}}.
\end{equation}
Therefore, as we increase $m>0$ the points $\tau$ and $\varsigma$ get exponentially closer but the distance between $X(\tau)$ and $X(\varsigma)$ diverges superexponentially. In particular $X$ cannot be continuous at $\tau$ as required.
\end{proof}
We now need to define the following type of convergence, which is stronger than Hausdorff convergence.
\begin{defn}
Let $\Cantor_*\subset M$ be a Cantor set, embedded in the metric space $M$, with presentation $\uline{B}_*=\{\oo{\word{w}{}}{B}_*\}_{\word{w}{}\in W^*}$. Let $\oo{\word{w}{}}{\Cantor_*}$ denote the cylinder set for $\Cantor_*$ associated to the word $\word{w}{}\in\oline{W}$. Let $\Cantor_n\subset M$ denote a sequence of Cantor sets, also embedded in $M$, with presentations $\uline{B}_n=\{\oo{\word{w}{}}{B}_n\}_{\word{w}{}\in W^*}$ combinatorially equivalent to $\uline{B}_*$. Then we say $\Cantor_n$ \emph{strongly converges} to $\Cantor_*$ if, for each $\word{w}{}\in\oline{W}$, $\oo{\word{w}{}}{\Cantor}_n\to \oo{\word{w}{}}{\Cantor}_*$.
\end{defn}
\begin{defn}
Let $X_n$ be a line field on $\Cantor_n$. Then we say $X_n$ \emph{strongly converges} to a line field $X_*$ on $\Cantor_*$ if, for each $\word{w}{}\in\oline{W}$,  $X_n(\oo{\word{w}{}}{\Cantor}_n)$ converges to $X_*(\oo{\word{w}{}}{\Cantor}_*)$ in the projected coordinates.
\end{defn}
\begin{prop} 
Let $F\in\I_{\Omega,\perm}(\bar\e)$ and let $\Cantor$ denote its renormalisation Cantor set. Given any invariant line field $X$ on $\Cantor$ the induced line fields $X_n$ on $\Cantor_n$ do not strongly converge to the tangent line field $X_*$ on $\Cantor_*$.
\end{prop}
\begin{proof}
Let us denote the correspondence between elements of $\Cantor_n$ and $\Cantor_*$ by $\pi_n$. Then a sequence of line fields $X_n$ strongly converges to $X_*$ if $X_n\circ\pi_n$ converges to $X_*$, where we have identified the line fields
with their projectivised coordinates.



Assume convergence holds and let $\epsilon>0$ and choose $N>0$ such that $|X_n\circ\pi_n-X_*|_{\Cantor_*}<\epsilon$ for all $n>N$.
Take any $m>N$ and let $n>m$ be chosen so that $\sigma^{n-m+1}\leq b^{p^{m}}\leq \sigma^{n-m}$.
Then
\begin{equation}
\left|X_m(\tau_m)-X_*(\tau_*)\right|, \ |X_m(F_m(\tau_m))-X_*(F_*(\tau_*))|<\epsilon,
\end{equation}
and the same holds if we replace $m$ by $n$. Let us denote the points $F_i(\tau_i)$ by $\varsigma_i$.

Observe that $X_*(\varsigma_*)=\del_x\phi_*(\tau_*)$. Therefore, as convergence of renormalisation implies $|\del_x\phi_m(\tau_m)-\del_x\phi_*(\tau_*)|<C\rho^m$, this tells us
\begin{equation}
\left|X_m(\varsigma_m)-\del_x\phi_m(\tau_m)\right|<\epsilon+C\rho^m.
\end{equation}
We will now show they must differ by a definite constant and achieve the required contradiction. We will show this by evaluating $X_m$ at a point near to $\varsigma_m$.
Consider the points $\varsigma=\MT_{m,n}(\varsigma_{n})$ and $\varsigma'=F_m\MT_{m,n}(\varsigma_{n})$. 
First let us evaluate $X_m$ at $\varsigma'$. By invariance this must be 
\begin{align}
\projD{F_m\MT_{m,n}}{\varsigma_{n}}\left(X_{n}(\varsigma_{n})\right)
&=\del_x\phi_m(\varsigma)+\frac{\del_y\phi_m(\varsigma)}{s_{m,n}(\varsigma_{n})+t_{m,n}(\varsigma_{n})X_{n}(\varsigma_{n})}.
\end{align}
The second term must be bounded away from zero as $X_{n}(\varsigma_{n})$ is bounded from above if $n$ is sufficiently large and the hypothesis on $m,n$ tells us $s_{m,n}$ and
$t_{m,n}$ are both comparable to $b^{p^m}$, as is the numerator $\del_y\phi_m(\varsigma)$. It is clear this bound can be made uniform in $m$.

Second, observe that $|\varsigma'-\varsigma_m|$ can be made arbitrarily small by choosing $m$ and $n-m$ sufficiently large, by the assumption that $\Cantor_n$ converges strongly to $\Cantor_*$.
Combining these gives us the required contradiction, as our hypothesis implies increasing $m$ leads to an exponential increase in $n$.
\end{proof}

\subsection{Failure of Rigidity at the Tip}
Using the same method as for the period doubling case we show that given two Cantor attractors $\Cantor$ and $\tilde\Cantor$ for some $F, \tilde
F\in \I_{\Omega,\perm}(\bar\e_0)$ with average Jacobian $b,\tilde b$ respectively, there is a bound on the Holder exponent of any conjugacy that preserves `tips'.
\begin{thm}\label{thm:nonrigid}
Let $F, \tilde F\in\I_{\Omega,\upsilon}(\bar\e)$ be two infinitely renormalisable H\'enon-like maps with respective renormalisation Cantor sets $\Cantor$ and $\tilde\Cantor$, and tips
$\tau$ and $\tilde\tau$. If there is a conjugacy $\pi\colon\tilde\Cantor\to\Cantor$ mapping $\tilde\tau$ to $\tau$ then the H\"older exponent $\alpha$ of $\pi$ satisfies
\begin{equation}
\alpha\leq \frac{1}{2}\left(1+\frac{\log \tilde b}{\log b}\right)
\end{equation}
\end{thm}
\begin{proof}
We will denote all objects associated with $F$ without tilde's and all objects associated with $\tilde F$ with them. For example $\MT$ and $\tilde \MT$ will denote the scope
function for $F$ and $\tilde F$ respectively.
 
Let $K>0$ be a positive constant which we will think of as being large. Let us choose an integer $m>0$ which ensures that $\tilde{b}^{p^m}>Kb^{p^m}$ and take an integer $n>m$
which satisfies $\sigma^{n-m+1}\leq b^{p^m}<\sigma^{n-m}$. This will be the depth of the Cantor sets $\Cantor$ and $\tilde\Cantor$ that we will consider. So let us consider $F$ and $\Cantor$. 
Let us denote the tip of $F_{n+1}$ by $\tau$ and let $\varsigma$ be its image under $F_{n+1}$.
Let $\dot\tau$ and $\dot\varsigma$ be the respective images of these points under $\MT_{m,n}$.
Let $\ddot\tau$ and $\ddot\varsigma$ be the respective images of $\dot\tau$ and $\dot\varsigma$ under $F_m$.
Let $\dddot\tau$ and $\dddot\varsigma$ be the respective images of $\ddot\tau,\ddot\varsigma$ under $\MT_{0,m-1}$.
The equivalent points for $\tilde F$ will be denoted by with tilde's.
Finally, $\tau_*$ denotes the tip of $F_*$ and $\varsigma_*$ denotes its image under $F_*$.


Observe that $\tau_*$ and $\varsigma_*$ will not lie on the same vertical or horizontal line. Therefore we know that the following constant
\begin{equation}
C_0=\half\min\left( \left|\pi_x(\varsigma_*)-\pi_x(\tau_*)\right|, \left|\pi_y(\varsigma_*)-\pi_y(\tau_*)\right|\right)
\end{equation}
is positive. By Theorem~\ref{thm:R-convergence} there exists an integer $N>0$ such that 
\begin{equation}
\left|\pi_x(\varsigma)-\pi_x(\tau)\right|, \left|\pi_y(\varsigma)-\pi_y(\tau)\right|,\left|\pi_x(\tilde\varsigma)-\pi_x(\tilde\tau)\right|,
\left|\pi_y(\tilde\varsigma)-\pi_y(\tilde\tau)\right|>C_0>0,
\end{equation}
for all integers $m>N$. Let $\delta=\left(\delta_x,\delta_y\right)=\varsigma-\tau$ and $\tilde\delta=\left(\tilde\delta_x,\tilde\delta_y\right)=\tilde\varsigma-\tilde\tau$. Clearly we also have an
upper bound for each of these quantities, namely $C_1=\diam(B)$. 

First we will derive an upper bound for the distance between $\dddot\varsigma$ and $\dddot\tau$, then we will derive a lower bound for the distance between $\dddot{\tilde\varsigma}$ and $\dddot{\tilde\tau}$.

Applying $\MT_{m,n}$ to $\varsigma$ and $\tau$ gives $\dot\varsigma-\dot\tau=D_{m,n}\left(\id+R_{m,n}\right)(\varsigma-\tau)$. 
Let $\dot\delta=\left(\dot\delta_x,\dot\delta_y\right)=\dot\varsigma-\dot\tau$. Hence by Proposition~\ref{prop:D_mn-estimate} and the above paragraph there exists a constant $C_2>0$ such that,
\begin{align}
\left|\dot\delta_x\right|
&=\left|\sigma_{m,n}s_{m,n}\left[\delta_x+r_{m,n}\left(\delta_x,\delta_y\right)\right]+\sigma_{m,n}t_{m,n}\delta_y\right| \\
&\leq C_2\sigma^{n-m}\left(\sigma^{n-m}+b^{p^m}\right) \notag \\
\left|\dot\delta_y\right|
&=\left|\sigma_{m,n}\delta_y\right| \\
&\leq C_2\sigma^{n-m} \notag
\end{align}

Next we apply $F_m=\left(\phi_m,\pi_x\right)$ which gives $\ddot\varsigma-\ddot\tau=F_m\left(\dot\varsigma\right)-F_m\left(\dot\tau\right)$. Let
$\ddot\delta=\left(\ddot\delta_x,\ddot\delta_y\right)=\ddot\varsigma-\ddot\tau$. First observe that by convergence of renormalisation, i.e. Theorem~\ref{thm:R-convergence},
there is a constant $C_2>0$ such that $\left|\del_x\phi_m\right|<C_2$. Second observe, by Theorem~\ref{thm:R-construction} there exists a constant $C_3>0$ such that
$\left|\del_y\phi_m\right|<C_3b^{p^m}$. Then by the Mean Value Theorem, if $\xi=\left(\pi_x\left(\dot\tau\right),\pi_y\left(\dot\varsigma\right)\right)$, there exist points
$\xi_y\in\left[\dot\varsigma,\xi\right], \xi_x\in\left[\xi,\dot\tau\right]$ such that
\begin{align}
\left|\ddot\delta_x\right|
&=\left|\del_x\phi_m\left(\xi_y\right)\dot\delta_x+\del_y\phi_m\left(\xi_x\right)\dot\delta_y\right| \\
&\leq C_3\left|\dot\delta_x\right|+C_4b^{p^m}\left|\dot\delta_y\right| \notag \\
&\leq C_2\sigma^{n-m}\left(C_2\left(\sigma^{n-m}+b^{p^m}\right)+C_3b^{p^m}\right) \notag \\
&\leq C_5\sigma^{n-m}\left(\sigma^{n-m}+b^{p^m}\right) \notag \\
\left|\ddot\delta_y\right|
&=\left|\dot\delta_x\right| \\
&\leq  C_2\sigma^{n-m}\left(\sigma^{n-m}+b^{p^m}\right) \notag
\end{align}

Now we apply $\MT_{0,m}$ which gives $\dddot\varsigma-\dddot\tau=D_{0,m}\left(\id+R_{0,m}\right)\left(\ddot\varsigma-\ddot\tau\right)$.
Let $\dddot\delta=\left(\dddot\delta_x,\dddot\delta_y\right)=\dddot\varsigma-\dddot\tau$. Hence, by Proposition~\ref{prop:D_mn-estimate} and the above paragraph, there is a constant $C_6>0$ such that
\begin{align}
\left|\dddot\delta_x\right|
&=
\left|\sigma_{0,m}s_{0,m}\left[\ddot\delta_x+r_{0,m}\left(\ddot\delta_x,\ddot\delta_y\right)\right]+\sigma_{0,m}t_{0,m}\ddot\delta_y\right| \\
&\leq
C_2\sigma^m\left|\sigma^{m}\left[\left|\ddot\delta_x\right|+\left|\del_xr_{0,m}\right|\left|\ddot\delta\right|\right]+b^{p^m}\left|\ddot\delta_y\right|\right| \notag \\
&\leq 
C_6\sigma^{2m}\sigma^{n-m}\left(\sigma^{n-m}+b^{p^m}\right)+C_2^2\sigma^nb^{p^m}\left(\sigma^{n-m}+b^{p^m}\right) \notag \\
&\leq 
\left(C_6\sigma^{n+m}+C_2^2\sigma^nb^{p^m}\right)\left(\sigma^{n-m}+b^{p^m}\right) \notag \\
\left|\dddot\delta_y\right|
&=\left|\sigma_{0,m}\ddot\delta_y\right| \\
&\leq C_2^2\sigma^{n}\left(\sigma^{n-m}+b^{p^m}\right) \notag
\end{align}
From the second inequality we find there exists a constant $C_7>0$ such that $\dist(\dddot\varsigma,\dddot\tau)\leq C_7\sigma^{2n-m}$.

Now we wish to a find a lower bound for $\dist(\dddot{\tilde\varsigma},\dddot{\tilde\tau})$. 
Applying $\tilde{\MT}_{m,n}$ to these points gives $\dot{\tilde{\varsigma}}-\dot{\tilde\tau}=\tilde{D}_{m,n}(\id+\tilde{R}_{m,n})(\tilde\varsigma-\tilde\tau)$. 
Let $\dot{\tilde\delta}=(\dot{\tilde\delta}_x,\dot{\tilde\delta}_y)=\dot{\tilde\varsigma}-\dot{\tilde\tau}$. 
Hence, as before, by Proposition~\ref{prop:D_mn-estimate} and the second paragraph there exists a constant $C_2>0$ such that,
$|\dot{\tilde{\delta}}_y|=|\tilde{\sigma}_{m,n}\tilde{\delta}_y|\leq C_2\sigma^{n-m}$. Let $C_8>1$ be constants satisfying 
\begin{equation}
\left|\tilde{\sigma}_{m,n}\right|>C_8^{-1}\sigma^{n-m},\quad \left|\tilde{t}_{m,n}\right|>C_8^{-1}\tilde{b}^{p^m},\quad \left|s_{m,n}\right|<C_8\sigma^{n-m},\quad
\left|\tilde{r}_{m,n}\right|<C_8.
\end{equation}
But, since $\tilde{b}^{p^m}>K\sigma^{n-m+1}$, Proposition~\ref{prop:D_mn-estimate} tells us
\begin{align}
\left|\dot{\tilde{\delta}}_x\right|
&=\left|\tilde{\sigma}_{m,n}\tilde{s}_{m,n}\left[\tilde\delta_x+\tilde{r}_{m,n}\left(\tilde\delta_x,\tilde\delta_y\right)\right]+\tilde{\sigma}_{m,n}\tilde{t}_{m,n}\tilde\delta_y\right| \\
&\geq \left|\tilde{\sigma}_{m,n}\right|
\left|
\left|\tilde{s}_{m,n}\right|\left|\tilde\delta_x+\tilde{r}_{m,n}\left(\tilde\delta_x,\tilde\delta_y\right)\right|
-\left|\tilde{t}_{m,n}\tilde\delta_y\right|
\right| \notag \\
&\geq C_8^{-1}\sigma^{n-m}\left(C_8^{-1}C_0b^{p^m}-C_8\left(C_0+C_8\right)\sigma^{n-m}\right) \notag \\
&\geq C_8^{-1}\sigma^{n-m}b^{p^m}\left(C_8^{-1}C_0-K^{-1}\sigma^{-1}C_8\left(C_0+C_8\right)\right). \notag
\end{align}
Since $K>0$ was assumed to be large (and the constants $C_8$ had no dependence upon $m$ and $n$) we find there exists a constant $C_9>0$ such that
$\left|\dot{\tilde{\delta}}_x\right|>C_9 b^{p^m}\sigma^{n-m}$.

Applying $\tilde{F}_m$ to $\dot{\tilde{\varsigma}}$ and $\dot{\tilde{\tau}}$ gives
$\ddot{\tilde\varsigma}-\ddot{\tilde\tau}=F_m\left(\dot{\tilde{\varsigma}}\right)-F_m\left(\dot{\tilde{\tau}}\right)$. Let
$\ddot{\tilde\delta}=\left(\ddot{\tilde\delta}_x,\ddot{\tilde\delta}_y\right)=\ddot{\tilde\varsigma}-\ddot{\tilde{\tau}}$. 
Then, ignoring the difference in the $x$-direction, we find $\left|\ddot{\tilde\delta}_y\right|=\left|\dot{\tilde\delta}_x\right|\geq C_{11}b^{p^m}\sigma^{n-m}$.

Now we apply $\tilde{\MT}_{0,m}$ which gives
$\dddot{\tilde\varsigma}-\dddot{\tilde\tau}=\tilde{D}_{0,m}\left(\id+\tilde{R}_{0,m}\right)\left(\ddot{\tilde\varsigma}-\ddot{\tilde\tau}\right)$. 
Let $\dddot{\tilde\delta}=\left(\dddot{\tilde\delta}_x,\dddot{\tilde\delta}_y\right)=\dddot{\tilde{\varsigma}}-\dddot{\tilde{\tau}}$. Then from Lemma~\ref{lem:D_mn-decomposition} we find 
$\left|\dddot{\tilde\delta}_y\right|=\left|\tilde{\sigma}_{0,m}\ddot{\tilde\delta}_y\right|$. But Proposition~\ref{prop:D_mn-estimate} implies there exists a constant $C_{10}>0$ such that
$\left|\tilde{\sigma}_{0,m}\right|\geq C_{10}\sigma^{m}$, so combining this with the estime from preceding paragraph gives $\left|\dddot{\tilde\delta}_y\right|\geq C_{9}C_{10}\sigma^{n}b^{p^m}$.

Now let us combine these upper and lower bounds. Let $C_{11}, C_{12}>0$ be constants satisfying $\dist\left(\dddot{\tilde\varsigma},\dddot{\tilde\tau}\right)>C_11\sigma^{n}b^{p^m}$ and
$\dist\left(\dddot\varsigma,\dddot\tau\right)<C_{12}\sigma^{2n-m}$. Then, assuming the H\"older
condition holds for some $C_{13}, \alpha>0$ we have
\begin{equation}
C_{11}\sigma^{n}\tilde b^{p^m} \leq \dist\left(\dddot{\tilde\tau},\dddot{\tilde\varsigma}\right)\leq C\dist\left(\dddot{\tau},\dddot{\varsigma}\right)^\alpha\leq C_{13}C_{12}^\alpha(\sigma^{2n-m})^\alpha
\end{equation}
which implies, after collecting all constant factors, that there is a $C>0$ such that
\begin{equation}
\sigma^m b^{p^m}\tilde b^{p^m}\leq C\left(\sigma^m b^{p^m}b^{p^m}\right)^{\alpha}
\end{equation}
and hence after taking the logarithm of both sides and passing to the limit gives 
\begin{equation}
\alpha\leq \frac{1}{2}\left(1+\frac{\log \tilde b}{\log b}\right).
\end{equation}
and hence the theorem is shown.
\end{proof}

\begin{appendices}

\renewcommand{\o}[2]{\ensuremath{{#2}^{#1}}}

\section{Elementary Results}
\begin{prop}\label{prop:1+Crho}
Let $C>0$ and $0\leq \rho\leq \delta<1$. Then the product $\prod_{i=0}^\infty (1+C\rho^i)$ converges and, moreover there exists a $C_0>0$ such that
\begin{equation}
\prod_{i=m}^\infty (1+C\rho^i)<1+C_0\rho^m
\end{equation}
\end{prop}
\begin{proof}
First let us show convergence. Observe that concavity of $\log$ implies $\log(1+C\rho^i)<\log(1)+\log'(1)C\rho^i=C\rho^i$. Therefore taking logarithms gives
\begin{align}
\log \left[\prod_0^n (1+C\rho^i)\right] \leq \sum_{i=0}^n \log (1+C\rho^i) \leq C\sum_{i=0}^n\rho^i \leq \frac{C}{1-\rho}.
\end{align}
Therefore, since the partial convergents are increasing, Bolzano-Weierstrass implies $\log\left[\prod_{i=0}^\infty(1+C\rho^i)\right]$ exists. Hence, applying $\exp$ gives us convergence.

Now let $F_{m,n}(\rho)=\prod_{i=m}^n (1+C\rho^i)$. Observe that, by the product rule,
\begin{align}
\frac{d}{d\rho}F_{m,n}(\rho)
&=\prod_{i=m}^n (1+C\rho^i)\sum_{i=m}^n\frac{Ci\rho^{i-1}}{1+C\rho^i} \\
&=C F_{m,n}(\rho)\rho^{m-1}\sum_{i=0}^{n-m}\frac{(m+i)\rho^i}{1+C\rho^i}
\end{align}
but since $C,\rho>0$,
\begin{align}
\sum_{i=0}^{n-m}\frac{(m+i)\rho^i}{1+C\rho^i}
&\leq m\sum_{i=0}^{n-m}\rho^i+\rho \sum_{i=0}^{n-m}i\rho^{i-1} \\
&\leq m\sum_{i=0}^{\infty}\rho^i+\rho \frac{d}{d\rho}\left(\sum_{i=0}^{\infty}\rho^i\right) \notag \\
&\leq \frac{m}{1-\rho}+\frac{\rho}{(1-\rho)^2} \notag
\end{align}
So, setting $M=C F_{m,n}(\delta)\left(\frac{m}{1-\delta}+\frac{\delta}{(1-\delta)^2}\right)$ and $G_m(\rho)=(1+\frac{M}{m}\rho^m)$, we find
\begin{equation}
\frac{d}{d\rho}F_{m,n}(\rho)\leq M\rho^{m-1}\leq \frac{d}{d\rho}G_m(\rho).
\end{equation}
Hence, as $F_{m,n}(0)=0=G_m(0)$ the result follows by setting $C_0=M/m$.
\end{proof}
\begin{lem}\label{lem:reciprocal}
Let $C>0$ and $0<\rho<1$. Then there exists a constant $C_0>0$ such that
\begin{equation}
\frac{1+C\rho}{1-C\rho}<1+C_0\rho^2.
\end{equation}
\end{lem}
The following Lemma and Proposition are straightforward and are left to the reader.
\begin{lem}\label{lem:e-vs-rho}
Given constants $0\leq \bar\e,\rho,\sigma<1$ and $C_0,C_1>0$ and a fixed integer $p>1$ there exists a constant $C>0$ such that for all integers $0<m<M$,
\begin{enumerate}
\item $C_0\bar\e^{p^m}+C_1\bar\e^{p^{m+1}}\leq C\bar\e^{p^m}$;
\item $C_0\bar\e^{p^m}+C_1\rho^m\leq C\rho^m$
\item $\sum_{m<n<M} \sigma^{i-m-1}\bar\e^{p^n-p^m}(1+C_0\rho^n)<C$
\item $\sum_{n>M}\bar\e^{p^n}\leq C\bar\e^{p^M}$
\end{enumerate}
\end{lem}
\begin{prop}\label{prop:exp-vs-superexp}
Given any $\rho>0$ there exists a $\oline\e>0$ such that $\sum_{i>0} \rho^i\e^{p^i}$ converges for all $\e<\oline\e$. Moreover for $0<\doline{\e}<\oline\e$ there exists a constant $C=C(\doline\e)>0$ such that $\sum_{i>0} \rho^i\e^{p^i}<C\e$ for all $0<\e<\doline\e$.
\end{prop}
\begin{lem}\label{lem:ratio-perturb}
Let $P,Q,P',Q'\in\RR$ with $P, Q'$ non-zero. Then
\begin{equation}
\left|\frac{P}{Q}-\frac{P'}{Q'}\right|\leq C\max\left(|P-P'|,|Q-Q'|\right)
\end{equation}
where $C=2|Q|^{-1}\max(1,|P'/Q'|)$.
\end{lem}
\begin{proof}
This is immediate given the following inequality,
\begin{align}
\left|\frac{P}{Q}-\frac{P'}{Q'}\right|
&=\left|\frac{1}{Q}(P-P')+P'\left(\frac{1}{Q}-\frac{1}{Q'}\right)\right| \\
&\leq\frac{1}{|Q|}\left[|P-P'|+\left|\frac{P'}{Q'}\right||Q'-Q|\right]. \notag
\end{align}
\end{proof}

\section{Variational Properties of Composition Operators}
In this section we derive properties of the composition operator. We show how the remainder term from Taylor's Theorem behaves under composition and we derive the first variation of the $n$-fold composition operator. Although we only state these for maps on $\RR^2$ or $\CC^2$ these work in full generality.

\begin{prop}\label{prop:remainder-1}
Let $F, G\in\Emb^2(\RR^2,\RR^2)$. For any $z_0,z_1\in\RR^2$, consider the decompositions
\begin{equation}
F(z_0+z_1)=F(z_0)+\D{F}{z_0}(\id+\Rem{F}{z_0})(z_1)
\end{equation}
and
\begin{equation}
G(z_0+z_1)=G(z_0)+\D{G}{z_0}(\id+\Rem{G}{z_0})(z_1).
\end{equation}
Then
\begin{align}
\Rem{FG}{z_0}(z_1)
&=\Rem{G}{z_0}(z_1)+\D{G}{z_0}^{-1}\Rem{F}{G(z_0)}(\D{G}{z_0}(\id+\Rem{G}{z_0})(z_1))
\end{align}
\end{prop}
\begin{proof}
Observe that
\begin{align}
FG(z_0+z_1)
&=FG(z_0)+\D{FG}{z_0}(z_1)+\D{FG}{z_0}(\Rem{FG}{z_0})(z_1)
\end{align}
must be equal to
\begin{align}
F(G(z_0+z_1))
&=F(G(z_0)+\D{G}{z_0}(\id+\Rem{G}{z_0})(z_1)) \\
&=F(G(z_0))+\D{F}{G(z_0)}(\id+\Rem{F}{G(z_0)})(\D{G}{z_0}(\id+\Rem{G}{z_0})(z_1)) \notag \\
&=F(G(z_0))+\D{F}{G(z_0)}\D{G}{z_0}(z_1)+\D{F}{G(z_0)}\D{G}{z_0}(\Rem{G}{z_0}(z_1)) \notag \\
&+\D{F}{G(z_0)}\Rem{F}{G(z_0)}(\D{G}{z_0}(\id+\Rem{G}{z_0})(z_1)). \notag
\end{align}
This implies that
\begin{align}
\Rem{FG}{z_0}(z_1)
&=\Rem{G}{z_0}(z_1)+\D{G}{z_0}^{-1}\Rem{F}{G(z_0)}(\D{G}{z_0}(\id+\Rem{G}{z_0})(z_1))
\end{align}
and hence the Proposition is shown.
\end{proof}

\begin{prop}\label{prop:variation-composition}
For each integer $n>0$ let $C_n\colon C^\omega(B,B)^n \to C^\omega(B,B)$ denote the $n$-fold composition operator
\begin{equation}
C_n(G_{1},\ldots,G_{n})=G_{1}\circ\cdots\circ G_{n}.
\end{equation}
For $i=1,\ldots,n$ assume we are give $F_i,G_i\in C^\omega(B,B)$ and let $E_i$ be defined by $G_i=F_i+E_i$. 
Then
\begin{equation}
C(G_{1},\ldots, G_{n})=C(F_{1},\ldots,F_{n})+\delta C_n(F_{1},\ldots,F_{n};E_{1},\ldots,E_{n})+\bigo(|E_{i}||E_j|)
\end{equation}
where
\begin{equation}
\delta C_n(F_{1},\ldots,F_{n};E_{1},\ldots,E_{n})=\sum_{i=1}^{n-1}\D{F_{1,\ldots,i}}{F_{i+1,\ldots,n}(z)}(E_{i+1}(F_{i+2,\ldots,n}(z)))
\end{equation}
where we have set $F_\emptyset, E_{n+1}=\id$.
\end{prop}
\begin{proof}
For notational simplicity let $F_{1,\ldots,n}=F_1\circ\cdots\circ F_n, G_{1,\ldots,n}=G_1\circ\cdots\circ G_n$ and let $E_{1,\ldots,n}$ satisfy
$G_{1,\ldots,n}=F_{1,\ldots,n}+E_{1,\ldots,n}$. Then equating $G_{1,2,\ldots,n}$ with $G_1\circ G_{2,\ldots,n}$ and using the power series expansion of $G_1$ gives
\begin{align}
G_{1,\ldots,n}(z)
&=G_{1}(F_{2,\ldots, n}(z)+E_{2,\ldots, n}(z)) \\
&=F_{1}(F_{2,\ldots, n}(z)+E_{2,\ldots, n}(z))+E_{1}(F_{2,\ldots, n}(z)+E_{2,\ldots, n}(z)) \notag \\
&=F_{1}(F_{2,\ldots, n}(z))+\D{F_{1}}{F_{2,\ldots, n}(z)}(E_{2,\ldots, n}(z))+\bigo(|E_{2,\ldots, n}|^2) \notag \\
&+E_{1}(F_{2,\ldots, n}(z))+\bigo(|\D{E_{1}}{}||E_{2,\ldots,n}|) \notag
\end{align}
while equating $G_{1,2,\ldots,n}$ with $G_{1,\ldots,n-1}\circ G_{n}$ and using the power series expansion of $G_{1,\ldots,n-1}$ gives
\begin{align}
G_{1,\ldots,n}(z)
&=G_{1,\ldots,n-1}(F_{n}(z)+E_{n}(z)) \\
&=F_{1,\ldots,n-1}(F_{n}(z)+E_{n}(z))+E_{1,\ldots,n-1}(F_{n}(z)+E_{n}(z)) \notag \\
&=F_{1,\ldots,n}(z)+\D{F_{1,\ldots,n-1}}{F_{n}(z)}(E_{n}(z))+\bigo(|E_{n}|^2) \notag \\
&+E_{1,\ldots,n-1}(F_{n}(z))+\bigo(|\D{E_{1,\ldots,n}}{}||E_{n}|). \notag
\end{align}
From the second of these expressions, inductively we find, setting $F_\emptyset, E_{n+1}=\id$, that
\begin{equation}
E_{1,\ldots,n}(z)=\sum_{i=1}^{n-1}\D{F_{1,\ldots,i}}{F_{i+1,\ldots,n}(z)}(E_{i+1}(F_{i+2,\ldots,n}(z)))+\bigo(|E_i||E_j|)
\end{equation}
\end{proof}

\end{appendices}

\bibliographystyle{amsplain}
\bibliography{henon1}

\end{document}